\documentclass[11pt]{article}

\usepackage{amssymb,amsmath,bm}
\usepackage{amsthm}

\usepackage{textcomp}
\usepackage{enumerate}      
\usepackage{graphicx}        
\usepackage{caption}

\usepackage{mathrsfs}

\usepackage{url}

\renewcommand{\setminus}{{\smallsetminus}}

\newcommand{\cp}[1]{\vcenter{\hbox{#1}}}

\usepackage{amssymb,amsmath,bm}
\usepackage{amsthm}
\usepackage{hyperref}
\usepackage{mathrsfs}
\usepackage{textcomp}
\usepackage{enumerate}
\usepackage{graphicx}
\usepackage{tikz}
\usepackage{verbatim}

\newtheorem{theorem}{Theorem}[section]
\newtheorem{lemma}[theorem]{Lemma}
\newtheorem{proposition}[theorem]{Proposition}

\newtheorem{corollary}[theorem]{Corollary}

\theoremstyle{remark}
\newtheorem{remark}[theorem]{Remark}

\theoremstyle{remark}

\numberwithin{equation}{section}

\title{\bf On the Volume Conjecture for hyperbolic Dehn-filled $3$-manifolds along the figure-eight knot}

\author{Ka Ho Wong \& Tian Yang}

\date{}

\begin{document}
\maketitle

\begin{abstract} Using Ohtsuki's method, we prove the Asymptotic Expansion Conjecture and the Volume Conjecture of the Reshetikhin-Turaev and the Turev-Viro invariants for all hyperbolic $3$-manifolds obtained by doing a Dehn-surgery along the figure-$8$ knot. \end{abstract}

\section{Introduction}

In \cite{W}, Witten interpreted values of the Jones polynomial using the Chern-Simons gauge theory, and constructed a sequence of complex valued 3-manifold invariants satisfying striking properties. This idea was mathematically rigorously formalized by Reshetikhin and Turaev\,\cite{RT90,RT91} though the representation theory of quantum groups and surgery descriptions\,\cite{K} of $3$-manifolds. In \cite{TV}, Turaev and Viro developed a different approach from triangulations  constructing a sequence of real valued invariants of $3$-manifolds. These two invariants turned out to be closely related\,\cite{Tu, W, Ro}, and are expected to contain geometric and topological information of the manifold.

Kashaev's Volume Conjecture\,\cite{Ka1, Ka2} (see also Murakami-Murakami\,\cite{MM}) fulfilled such expectation by relating the colored Jones polynomials of a knot to the hyperbolic geometry of its complement. More precisely, the Volume Conjecture asserts that value of the $n$-th normalized colored Jones polynomial of a hyperbolic knot evaluated at the $n$-th primitive root of unit $t=e^{\frac{2\pi\sqrt{-1}}{n}}$ grows exponentially in $n,$ and the growth rate is proportional to the hyperbolic volume of the complement of the knot. Recently, Chen and the second author\,\cite{CY} conjectured, now known as the Chen-Yang volume conjecture, that for odd $r$ the values at the root of unity $q=e^{\frac{2\pi\sqrt{-1}}{r}}$ of the $r$-th Reshetikhin-Turaev and Turaev-Viro invariants of a hyperbolic $3$-manifold  grow exponentially in $r,$ with growth rate respectively proportional to the complex volume and the hyperbolic volume of the manifold.  This conjecture was later refined independently by Ohtsuki\,\cite{O2} and Gang-Romo-Yamazaki\,\cite{GRY} to include the adjoint twisted Reidemeister torsion\,\cite{P} of the manifold in the asymptotic expansion of the invariants.

In \cite{BDKY}, Belletti, Detcherry, Kalfagianni and the second author proved the Chen-Yang volume conjecture for the family of fundamental shadow link complements. The fundamental shadow link complements were shown\,\cite{CT} to form a universal class of $3$-manifolds in the sense that any orientable 3-manifold with empty or toroidal boundary is obtained from the complement of a fundamental shadow link by doing a Dehn-surgery along suitable components. Therefore, understanding the asymptotic behavior of the invariants under Dehn-surgeries becomes a necessary step toward the solution to the Chen-Yang volume conjecture.

An earlier work of Ohtsuki\,\cite{O2} was a result along this direction, where he obtained the asymptotic expansion of the Reshetikhin-Turaev invariants of all hyperbolic $3$-manifolds obtained by doing an integral Dehn-surgery along the figure-$8$ knot. Together with a sequence of his works\,\cite{O, OT1, OY, O3, OT2}, Ohtsuki developed a method of attacking Kashaev's and Chen-Yang's volume conjectures consisting of a circle of creative ideas including the use of Faddeev's quantum dilogarithm functions, the Poisson summation formula and the saddle point approximation.

The main result of this article is our first attempt to understand the asymptotic behavior of the Reshetikhin-Turaev and the Turaev-Viro invariants under Dehn-surgeries, which generalizes Ohtsuki's result from integral Dehn-surgeries to rational Dehn-surgeries along the figure-$8$ knot. We note that our approach also works for the integral Dehn-surgeries, and is up to details the same as Ohtsuki's. A new idea in our approach is a use of the reciprocity of generalized Gaussian sum in the simplification of our formula for the Reshetikhin-Turaev invariants. Another new feature  is that we clarify the geometric meaning of the critical values of certain involved functions relating them to the desired geometric quantities (see Section \ref{geometry}), which previously could only be done by numerical computations. The argument in Section \ref{computation} can be directly applied to rationally Dehn-filled $3$-manifold along any other knot, and together with Section \ref{geometry} provide  a reinforcement of Ohtsuki's method.

\begin{theorem}\label{main1} Let $M$ be a closed oriented hyperbolic $3$-manifold obtained by a doing Dehn-surgery along the figure-$8$ knot, and let $\mathrm{RT}_r(M)$ be its $r$-th Reshetikhin-Turaev invariant evaluated at the root $q=e^{\frac{2\pi\sqrt{-1}}{r}}.$ Then as $r$ varies along positive odd integers,
$$\mathrm{RT}_r(M)=\frac{C_r}{2}\frac{1}{\sqrt{\mathrm{Tor}(M;\mathrm{Ad}_\rho)}}e^{\frac{r}{4\pi}\big(\mathrm{Vol}(M)+\sqrt{-1}\mathrm{CS}(M)\big)}\Big(1+O\Big(\frac{1}{r}\Big)\Big),$$
where $C_r$ is a constant of norm $1$ independent of the geometric structure of $M,$ $\mathrm{Tor}(M;\mathrm{Ad}_\rho)$ is the adjoint twisted Reidemeister torsion, $\mathrm{Vol}(M)$ is the hyperbolic volume and $\mathrm{CS}(M)$ is the Chern-Simons invariant of $M.$ As a consequence,
$$\lim_{r\to \infty}\frac{4\pi}{r}\log \mathrm{RT}_r(M)=\mathrm{Vol}(M)+\sqrt{-1}\mathrm{CS}(M)\quad(\text{mod }\sqrt{-1}\pi^2\mathbb Z).$$
\end{theorem}

It is proved in \cite{Tu, W, Ro}, and at the root $q=e^{\frac{2\pi\sqrt{-1}}{r}}$ in \cite{DKY}, that for a closed oriented $3$-manifold, the Turaev-Viro invariant is up to a scalar the square of the norm of the Reshetikhin-Turaev invariant.  As a consequence, we have

\begin{theorem}\label{main2} Let $M$ be a closed oriented hyperbolic $3$-manifold obtained by doing a rational Dehn-surgery along the figure-$8$ knot, and let $\mathrm{TV}_r(M)$ be its $r$-th Turaev-Viro invariant evaluated at the root $q=e^{\frac{2\pi\sqrt{-1}}{r}}.$ Then as $r$ varies along positive odd integers,
$$ \mathrm{TV}_r(M)=\frac{2^{b_2(M)-b_0(M)}}{\big|\mathrm{Tor}(M;\mathrm{Ad}_\rho)\big|}e^{\frac{r}{2\pi}\mathrm{Vol}(M)}\Big(1+O\Big(\frac{1}{r}\Big)\Big),$$
where $b_0(M)$ and $b_2(M)$ are respectively the zeroth and the second $\mathbb Z_2$ Betti-number of $M.$ As a consequence,  
$$\lim_{r\to \infty}\frac{2\pi}{r}\log \mathrm{TV}_r(M)=\mathrm{Vol}(M).$$
\end{theorem}
\bigskip

\noindent\textbf{Outline of the proof.} The proof follows the guideline of Ohtsuki's method. In Proposition \ref{formula}, we compute the Reshetikhin-Turaev invariants of $M$ and write them as a sum of values of a holomorphic function $f_r$ at integral points; and a key ingredient in the computation is Lemma \ref{sm} that  iteratively using a reciprocity of generalized Gaussian sums, we can simplify a multi-sum into a single sum. The function $f_r$ comes from Faddeev's quantum dilogarithm function. Using Poisson summation formula, we in Proposition \ref{Poisson} write the invariants as a sum of the Fourier coefficients of $f_r.$  In Proposition \ref{Vol} we show that the critical value of the functions in the two leading Fourier coefficients $\hat f_r(s^+,m^+,0)$ and $\hat f_r(s^-,m^-, 0)$ for $s^\pm$ and $m^\pm$ defined in Lemma \ref{arith} coincide with the complex volume of $M$ and the determinant of the Hessian matrix gives the adjoint twisted Reidemeister torsion of $M.$ The key observation is Lemma \ref{=} and Lemma \ref{==} that the system of critical point equations is equivalent to the system of hyperbolic gluing equations (consisting of an edge equation and a Dehn-surgery equation) for a particular ideal triangulation of the figure-$8$ knot complement. In Proposition \ref{leading} we verify the conditions for applying the saddle point approximation showing that the growth rate of the leading Fourier coefficients are those critical values, ie., the complex volume; and in Section \ref{estother}, we estimate the other Fourier coefficients. Finally, we complete the proof by showing in Proposition \ref{addup} that the two leading Fourier coefficient do not cancel each other and the sum of all the other Fourier coefficients are  neglectable, .
\\

\noindent\textbf{Acknowledgments.}  The authors would like to thank Francis Bonahon, Qingtao Chen, Effie Kalfagianni, Tomotada Ohtsuki and Hongbin Sun for helpful discussions. The second author is partially supported by NSF Grant DMS-1812008.

\section{Preliminaries}

\subsection{Reshetikhin-Turaev invariants}

In this article we will follow the skein theoretical approach of the Reshetikhin-Turaev invariants\,\cite{BHMV, Li} and focus on the case $q=e^{\frac{2\pi\sqrt{-1}}{r}},$ and equivalently $t=q^2=e^{\frac{4\pi\sqrt{-1}}{r}},$ for odd integers $r\geqslant 3.$

A framed link in an oriented $3$-manifold $M$ is a smooth embedding $L$ of a disjoint union of finitely many thickened circles $\mathrm S^1\times [0,\epsilon]$ for some $\epsilon>0$ into $M.$ The Kauffman bracket skein module $\mathrm K_r(M)$ of $M$ is the $\mathbb C$-module generated by the isotopy classes of framed links in $M$  modulo the following two relations: 

\begin{enumerate}[(1)]
\item  \emph{Kauffman Bracket Skein Relation:} \ $\cp{\includegraphics[width=1cm]{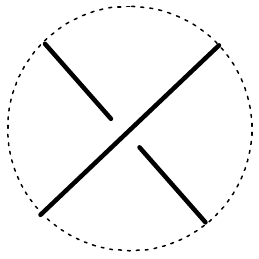}}\ =\ e^{\frac{\pi\sqrt{-1}}{r}}\ \cp{\includegraphics[width=1cm]{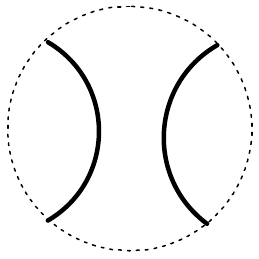}}\  +\ e^{-\frac{\pi\sqrt{-1}}{r}}\ \cp{\includegraphics[width=1cm]{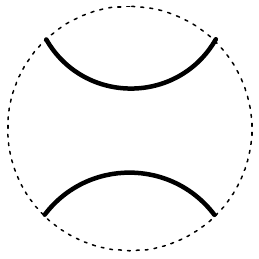}}.$ 

\item \emph{Framing Relation:} \ $L \cup \cp{\includegraphics[width=0.8cm]{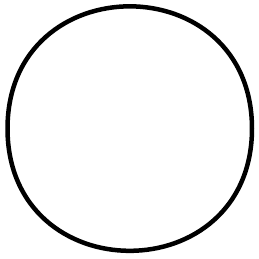}}=(-e^{\frac{2\pi\sqrt{-1}}{r}}-e^{-\frac{2\pi\sqrt{-1}}{r}})\ L.$ 
\end{enumerate}

There is a canonical isomorphism 
$$\langle\ \rangle:\mathrm K_r(\mathrm S^3)\to\mathbb C$$
defined by sending the empty link to $1.$ The image $\langle L\rangle$ of a framed link $L$ is called the Kauffman bracket of $L.$

Let $\mathrm K_r(A\times [0,1])$ be the skein module of the product of an annulus $A$ with a closed interval. For any link diagram $D$ in $\mathbb R^2$ with $k$ ordered components and $b_1, \dots, b_k\in \mathrm K_r(A\times [0,1]),$ let 
$$\langle b_1,\dots, b_k\rangle_D$$
be the complex number obtained by cabling $b_1,\dots, b_k$ along the components of $D$ considered as a element of $K_r(\mathrm S^3)$ then taking the Kauffman bracket $\langle\ \rangle.$

On $\mathrm K_r(A\times [0,1])$ there is a commutative multiplication induced by the juxtaposition of $A,$ making it a $\mathbb C$-algebra; and as a $\mathbb C$-algebra $\mathrm K_r(A\times [0,1])  \cong \mathbb C[z],$ where $z$ is the core curve of $A.$ For an integer $n\geqslant 0,$ let $e_n(z)$ be the $n$-th Chebyshev polynomial defined by the recursive relations
$e_0(z)=1,$ $e_1(z)=z$ and $e_n(z)=ze_{n-1}(z)-e_{n-2}(z).$ The Kirby coloring $\omega_r\in\mathrm K_r(A\times [0,1])$ is then defined by 
$$\omega_r=\sum_{n=0}^{r-2}(-1)^n[n+1]e_n,$$
where $[n]$ is the quantum integer defined by
$$[n]=\frac{e^{\frac{2n\pi\sqrt{-1}}{r}}-e^{-\frac{2n\pi\sqrt{-1}}{r}}}{e^{\frac{2\pi\sqrt{-1}}{r}}-e^{-\frac{2\pi\sqrt{-1}}{r}}}.$$

Suppose $M$ is obtained from $S^3$ by doing a surgery along a framed link $L,$ $D(L)$ is a standard diagram of $L,$ ie., the blackboard framing of $D(L)$ coincides with the framing of $L,$ and $\sigma(L)$ is the signature of the linking matrix of $L.$  Let $U_+$ be the diagram of the unknot with framing $1$ and let $$\mu_r=\frac{\sin\frac{2\pi}{r}}{\sqrt r}.$$
 Then the $r$-th Reshetikhin-Turaev invariant of $M$ is defined as
\begin{equation}\label{RT}
\mathrm{RT}_r(M)=\mu_r \langle \mu_r\omega_r, \dots, \mu_r\omega_r\rangle_{D(L)}\langle \mu_r \omega_r\rangle _{U_+}^{-\sigma(L)}.
\end{equation}


\subsection{Dilogarithm and Lobachevsky functions}

Let $\log:\mathbb C\setminus (-\infty, 0]\to\mathbb C$ be the standard logarithm function defined by
$$\log z=\log|z|+\sqrt{-1}\arg z$$
with $-\pi<\arg z<\pi.$
 
The dilogarithm function $\mathrm{Li}_2: \mathbb C\setminus (1,\infty)\to\mathbb C$ is defined by
$$\mathrm{Li}_2(z)=-\int_0^z\frac{\log (1-u)}{u}du$$
where the integral is along any path in $\mathbb C\setminus (1,\infty)$ connecting $0$ and $z,$ which is holomorphic in $\mathbb C\setminus [1,\infty)$ and continuous in $\mathbb C\setminus (1,\infty).$

The dilogarithm function satisfies the follow properties (see eg. Zagier\,\cite{Z})
\begin{equation}\label{Li2}
\mathrm{Li}_2\Big(\frac{1}{z}\Big)=-\mathrm{Li}_2(z)-\frac{\pi^2}{6}-\frac{1}{2}\big(\log(-z)\big)^2.
\end{equation} 
In the unit disk $\big\{z\in\mathbb C\,\big|\,|z|<1\big\},$ 
\begin{equation}\label{Li1}
\mathrm{Li}_2(z)=\sum_{n=1}^\infty\frac{z^n}{n^2},
\end{equation}
and on the unit circle $\big\{ z=e^{2\sqrt{-1}\theta}\,\big|\,0 \leqslant \theta\leqslant\pi\big\},$ 
$$\mathrm{Li}_2(e^{2\sqrt{-1}\theta})=\frac{\pi^2}{6}+\theta(\theta-\pi)+2\sqrt{-1}\Lambda(\theta),$$
where  $\Lambda:\mathbb R\to\mathbb R$  is the Lobachevsky function (see eg. Thurston's notes\,\cite[Chapter 7]{T}) defined by
$$\Lambda(\theta)=-\int_0^\theta\log|2\sin t|dt.$$

The Lobachevsky function  is an odd function of period $\pi.$ It achieves the absolute maximums at $k\pi +\frac{\pi}{6},$ $k\in \mathbb Z,$ and the absolute minimums at $k\pi+\frac{5\pi}{6},$ $k\in \mathbb Z.$ Moreover, it satisfies the functional equation 
\begin{equation*}
\frac{1}{2}\Lambda(2\theta)=\Lambda(\theta)+\Lambda\Big(\theta+\frac{\pi}{2}\Big).
\end{equation*}


\subsection{Quantum dilogarithm functions}
We will consider the following variant of Faddeev's quantum dilogarithm functions\,\cite{F, FKV}. All the result in this section are essentially due to Kashaev and some of them could also be found in \cite{CM}. For the readers' convenience, we also include a proof here.

Let $r\geqslant 3$ be an odd integer. Then the following contour integral
\begin{equation}\label{qd}
\varphi_r(z)=\frac{4\pi\sqrt{-1}}{r}\int_{\Omega}\frac{e^{(2z-\pi)x}}{4x \sinh (\pi x)\sinh (\frac{2\pi x}{r})}\ dx
\end{equation}
defines a holomorphic function on the domain $$\Big\{z\in \mathbb C \ \Big|\ -\frac{\pi}{r}<\mathrm{Re}z <\pi+\frac{\pi}{r}\Big\},$$  
  where the contour is
$$\Omega=\big(-\infty, -\epsilon\big]\cup \big\{z\in \mathbb C\ \big||z|=\epsilon, \mathrm{Im}z>0\big\}\cup \big[\epsilon,\infty\big),$$
for some $\epsilon\in(0,1).$
Note that the integrand has poles at $n\sqrt{-1},$ $n\in\mathbb Z,$ and the choice of  $\Omega$ is to avoid the pole at $0.$
\\

The function $\varphi_r(z)$ satisfies the following fundamental property.
\begin{lemma}
\begin{enumerate}[(1)]
\item For $z\in\mathbb C$ with  $0<\mathrm{Re}z<\pi,$
\begin{equation}\label{fund}
1-e^{2 \sqrt{-1} z}=e^{\frac{r}{4\pi\sqrt{-1}}\Big(\varphi_r\big(z-\frac{\pi}{r}\big)-\varphi_r\big(z+\frac{\pi}{r}\big)\Big)}.
 \end{equation}
 
 \item For $z\in\mathbb C$ with  $-\frac{\pi}{r}<\mathrm{Re}z<\frac{\pi}{r},$
 \begin{equation}\label{f2}
1+e^{r \sqrt{-1}z}=e^{\frac{r}{4\pi\sqrt{-1}}\Big(\varphi_r(z)-\varphi_r\big(z+\pi\big)\Big)}.
\end{equation}
\end{enumerate}
\end{lemma}

\begin{proof} In the region enclosed by $\Omega$ in the upper half plane, the function $\frac{e^{(2z-\pi)x}}{2x\sinh(\pi x)}$ has simple poles at $x=n\sqrt{-1},$ $n\in\mathbb Z_+.$ Hence by the Residue Theorem, 
\begin{equation*}
\begin{split}
\frac{r}{4\pi\sqrt{-1}}\Big(\varphi_r\Big(z-\frac{\pi}{r}\Big)-\varphi_r\Big(z+\frac{\pi}{r}\Big)\Big)&=-\int_\Omega\frac{e^{(2z-\pi)x}}{2x\sinh(\pi x)}dx\\
&=-\sum_{n=1}^\infty 2\pi\sqrt{-1}\cdot\mathrm {Res}_{x=n\sqrt{-1}}\Big(\frac{e^{(2z-\pi)x}}{2x\sinh(\pi x)}\Big)\\
&=-\sum_{n=1}^\infty \frac{(e^{2\sqrt{-1}z})^n}{n}=\log(1-e^{2\sqrt{-1}z}),
\end{split}
\end{equation*}
which proves (1).

In the same region, the function $\frac{e^{2zx}}{2x\sinh(\frac{2\pi x}{r})}$ has simple poles at $x=\frac{rn\sqrt{-1}}{2},$ $n\in\mathbb Z_+.$ Hence by the Residue Theorem, 
\begin{equation*}
\begin{split}
\frac{r}{4\pi\sqrt{-1}}\Big(\varphi_r(z)-\varphi_r\big(z+\pi\big)\Big)&=-\int_\Omega\frac{e^{2zx}}{2x\sinh(\frac{2\pi x}{r})}dx\\
&=-\sum_{n=1}^\infty 2\pi\sqrt{-1}\cdot\mathrm {Res}_{x={\frac{rn\sqrt{-1}}{2}}}\Big(\frac{e^{2zx}}{2x\sinh(\frac{2\pi x}{r})}\Big)\\
&=-\sum_{n=1}^\infty \frac{(e^{r \sqrt{-1}z})^n}{(-1)^nn}=\log(1+e^{r \sqrt{-1}z}),
\end{split}
\end{equation*}
which proves (2).
\end{proof}

Using (\ref{fund}) and (\ref{f2}), for $z\in\mathbb C$ with $\pi+\frac{2(n-1)\pi}{r}< \mathrm{Re}z< \pi+\frac{2n\pi}{r},$ we can define $\varphi_r(z)$  inductively by the relation
\begin{equation}\label{extension}
\prod_{k=1}^n\Big(1-e^{2 \sqrt{-1} \big(z-\frac{(2k-1)\pi}{r}\big)}\Big)=e^{\frac{r}{4\pi\sqrt{-1}}\Big(\varphi_r\big(z-\frac{2n\pi}{r}\big)-\varphi_r(z)\Big)},
\end{equation}
extending $\varphi_r(z)$ to a meromorphic function on $\mathbb C.$  The poles of $\varphi_r(z)$ have the form $(a+1)\pi+\frac{b\pi}{r}$ or $-a\pi-\frac{b\pi}{r}$ for all nonnegative integer $a$ and positive odd integer $b.$

Let $t=e^{\frac{4\pi\sqrt{-1}}{r}},$
and let $$(t)_n=\prod_{k=1}^n(1-t^k).$$

\begin{lemma}\label{factorial}
\begin{enumerate}[(1)]
\item For $0\leqslant n \leqslant r-1,$
\begin{equation}
(t)_n=e^{\frac{r}{4\pi\sqrt{-1}}\Big(\varphi_r\big(\frac{\pi}{r}\big)-\varphi_r\big(\frac{2\pi n}{r}+\frac{\pi}{r}\big)\Big)}.
\end{equation}
\item For $\frac{r-1}{2}\leqslant n \leqslant r-1,$
\begin{equation} \label{move}
(t)_n=2e^{\frac{r}{4\pi\sqrt{-1}}\Big(\varphi_r\big(\frac{\pi}{r}\big)-\varphi_r\big(\frac{2\pi n}{r}+\frac{\pi}{r}-\pi\big)\Big)}.
\end{equation}
\end{enumerate}
\end{lemma}

\begin{proof} Inductively using (\ref{fund}), we have (1). To see (2), we by (1) have
$$(t)_n=e^{\frac{r}{4\pi\sqrt{-1}}\Big(\varphi_r\big(\frac{\pi}{r}\big)-\varphi_r\big(\frac{2\pi n}{r}+\frac{\pi}{r}-\pi\big)\Big)}e^{\frac{r}{4\pi\sqrt{-1}}\Big(\varphi_r\big(\frac{2\pi n}{r}+\frac{\pi}{r}-\pi\big)-\varphi_r\big(\frac{2\pi n}{r}+\frac{\pi}{r}\big)\Big)}.$$
By analyticity, (\ref{f2}) holds for all $z$ that is not a pole. In particular, it holds for $z=\frac{2\pi n}{r}+\frac{\pi}{r}-\pi,$ and we have
$$e^{\frac{r}{4\pi\sqrt{-1}}\Big(\varphi_r\big(\frac{2\pi n}{r}+\frac{\pi}{r}-\pi\big)-\varphi_r\big(\frac{2\pi n}{r}+\frac{\pi}{r}\big)\Big)}=1+e^{r \sqrt{-1}\big(\frac{2\pi n}{r}+\frac{\pi}{r}-\pi\big)}=2,$$
which proves (2).
\end{proof} 

We consider (\ref{move}) because there are poles in $(\pi,2\pi),$ and we move everything into $(0,\pi)$ to avoid those poles.

The function $\varphi_r(z)$ and the dilogarithm function are closely related as follows.

\begin{lemma}\label{converge}  
\begin{enumerate}[(1)]
\item For every $z$ with $0<\mathrm{Re}z<\pi,$ 
\begin{equation}\label{conv}
\varphi_r(z)=\mathrm{Li}_2(e^{2\sqrt{-1}z})+\frac{2\pi^2e^{2\sqrt{-1}z}}{3(1-e^{2\sqrt{-1}z})}\frac{1}{r^2}+O\Big(\frac{1}{r^4}\Big).
\end{equation}
\item For every $z$ with $0<\mathrm{Re}z<\pi,$ 
\begin{equation}\label{conv}
\varphi_r'(z)=-2\sqrt{-1}\log(1-e^{2\sqrt{-1}z})+O\Big(\frac{1}{r^2}\Big).
\end{equation}
\item 
As $r\to \infty,$ $\varphi_r(z)$ uniformly converges to $\mathrm{Li}_2(e^{2\sqrt{-1}z})$ and $\varphi_r'(z)$ uniformly converges to $-2\sqrt{-1}\log(1-e^{2\sqrt{-1}z})$ on a compact subset of $\{z\in \mathbb C\ |\ 0<\mathrm{Re}z<\pi\}.$ 
\end{enumerate}
\end{lemma}

\begin{proof}
For (1), since 
$$\frac{1}{\sinh ( \frac{2\pi x}{r})}=\frac{r}{2\pi x}-\frac{\pi x}{3r}+O\Big(\frac{1}{r^3}\Big),$$ we have
$$\varphi_r(z)=\sqrt{-1}\int_\Omega\frac{e^{(2z-\pi)x}}{2x^2\sinh (\pi x)}dx-\frac{\pi^2\sqrt{-1}}{r^2}\int_\Omega\frac{e^{(2z-\pi)x}}{3\sinh (\pi x)}dx+O\Big(\frac{1}{r^4}\Big).$$
By the Residue Theorem, 
\begin{equation*}
\begin{split}
\sqrt{-1}\int_\Omega\frac{e^{(2z-\pi)x}}{2x^2\sinh (\pi x)}dx&=-2\pi\sum_{n=1}^\infty\mathrm {Res}_{x=n\sqrt{-1}}\Big(\frac{e^{(2z-\pi)x}}{2x^2\sinh (\pi x)}\Big)\\
&=\sum_{n=1}^\infty\frac{(e^{2\sqrt{-1}z})^n}{n^2}=\mathrm{Li}_2(e^{2\sqrt{-1}z}),
\end{split}
\end{equation*}
where the last equality holds by (\ref{Li1}) for $z$ so that $e^{2\sqrt{-1}z}$ is in the unit disk, and holds by analyticity for all $z$ with $0<\mathrm{Re}z<\pi.$

By the Residue Theorem again,
\begin{equation*}
\begin{split}
-\frac{\pi^2\sqrt{-1}}{r^2}\int_\Omega\frac{e^{(2z-\pi)x}}{3\sinh (\pi x)}dx&=\frac{2\pi^3}{r^2}\sum_{n=1}^\infty\mathrm {Res}_{x=n\sqrt{-1}}\Big(\frac{e^{(2z-\pi)x}}{3\sinh (\pi x)}\Big)\\
&=\frac{2\pi^2}{3r^2}\sum_{n=1}^\infty(e^{2\sqrt{-1}z})^n=\frac{2\pi^2e^{2\sqrt{-1}z}}{3(1-e^{2\sqrt{-1}z})}\frac{1}{r^2}.
\end{split}
\end{equation*}
This proves (\ref{conv}). 

(2) follows from (1), and (3) follows from (1) and (2).
\end{proof}

\subsection{A geometric proposition}

\begin{proposition}\label{small} There is an $\epsilon>0$ such that for any relatively prime pair $(p,q)\neq(\pm 5, \pm 1)$ so that the closed oriented $3$-manifold $M$ obtained by doing a $\frac{p}{q}$ Dehn-surgery along the figure-$8$ knot $K_{4_1}$  is hyperbolic,  
$$\mathrm{Vol}(M)>\frac{1}{2}\mathrm{Vol}(S^3\setminus K_{4_1})+\epsilon.$$
\end{proposition}

\begin{proof}
By Futer-Kalfagianni-Purcell\,\cite[Theorem 1.1]{FKP}, if $M$ is obtained from the complement of a hyperbolic knot $K$ in $\mathrm S^3$ by a Dehn-surgery along a boundary curve $\gamma,$ then 
$$\mathrm{Vol}(M)\geqslant \Big(1-\Big(\frac{2\pi}{L(\gamma)}\Big)^2\Big)^{\frac{3}{2}}\mathrm{Vol}(\mathrm S^3\setminus K),$$
where $L(\gamma)$ is the length of $\gamma$ in the induced Euclidean metric on the boundary of the embedded horoball neighborhood of the cusp. For the $K_{4_1}$ complement, the boundary of the maximum horoball neighborhood is a tiling of eight regular triangles of side $1.$ Hence as drawn in Figure \ref{holonomy} (see also Thurston's notes \cite{T}), $L(x)=4$ and $L(y)=1.$ As a consequence, the meridian $m=y$ and the longitude $l=x+2y$ are perpendicular, $L(m) =1,$ $L(l) =2\sqrt3$ and $L(pm+ql)=\sqrt{p^2+12q^2}.$ As a consequence,
$$\mathrm{Vol}(M)\geqslant\Big(1-\frac{4\pi^2}{p^2+12q^2}\Big)^{\frac{3}{2}}\mathrm{Vol}(S^3\setminus K_{4_1}).$$
If $p^2+12q^2>\frac{4\pi^2}{1-\big(\frac{1}{2}\big)^{\frac{2}{3}}}\approx 106.67,$ then $\Big(1-\frac{4\pi^2}{p^2+12q^2}\Big)^{\frac{3}{2}}>\frac{1}{2}$
and
$$\mathrm{Vol}(M)>\frac{1}{2}\mathrm{Vol}(S^3\setminus K_{4_1}).$$
Therefore, by the symmetry of $K_{4_1}$ complement, we only need to check for the pairs $(p,q)=(6, 1),$ $(7, 1),$ $(8, 1),$ $(9, 1)$ $(1, 2),$ $(3, 2),$ $(5, 2)$ and $(7, 2)$ where $p^2+12q^2<107,$ which could be numerically done by using SnapPy\,\cite{D}. \end{proof}

\begin{remark}\label{only} The end of the proof of Proposition \ref{small} is the only place in this article where we need a numerical computation.
\end{remark}


\section{Computation of the Reshetikhin-Turaev invariants}\label{computation}

The main result of this section is  Proposition \ref{formula} where we compute the Reshetikhin-Turaev invariants of the closed oriented $3$-manifold obtained by doing a $\frac{p}{q}$ Dehn-surgery along the figure-$8$ knot $K_{4_1}.$ Recall that if $M$ is the  $3$-manifold obtained from $S^3$ by doing a $\frac{p}{q}$ Dehn-surgery along a knot $K,$ then it can also be obtained by doing a surgery along a framed link $L$ (see Figure \ref{link}) of $k$ components with framings $a_1,\dots, a_k$ coming from the continued fraction
$$\frac{p}{q}=a_k-\frac{1}{a_{k-1}-\frac{1}{\cdots-\frac{1}{a_1}}}.$$ See eg. \cite[p.273]{R}. 
\begin{figure}[htbp]
\centering
\includegraphics[scale=0.4]{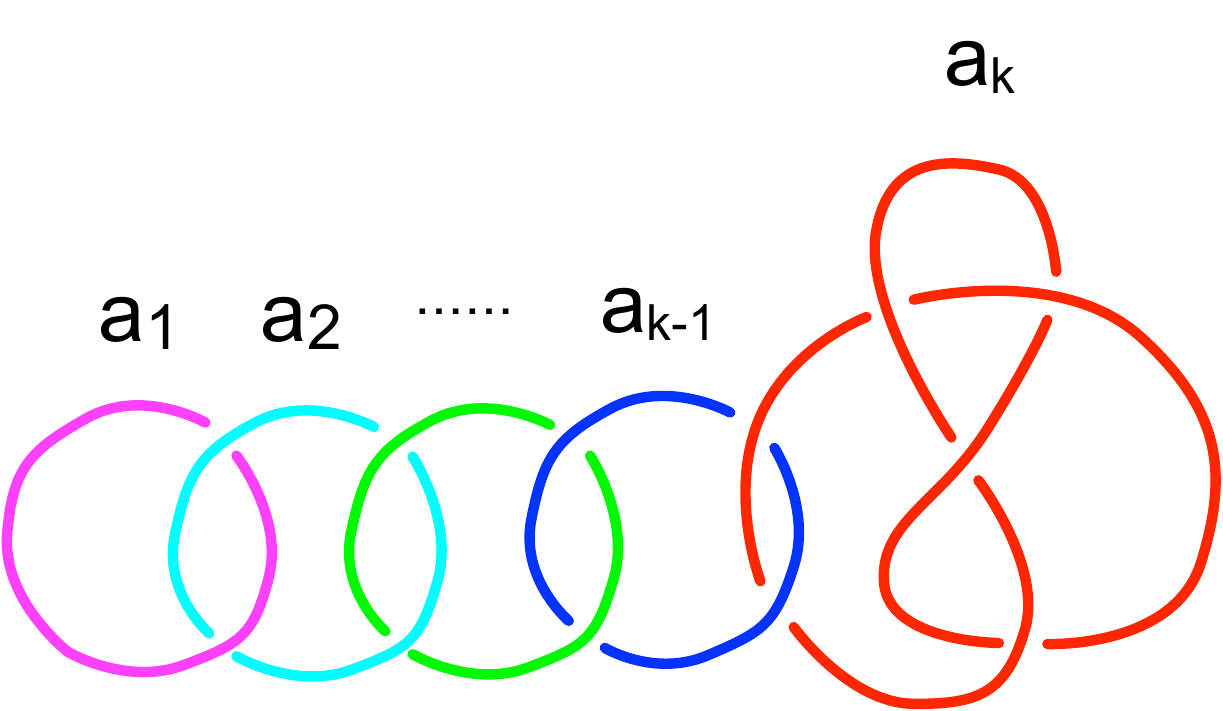}
\caption{The link L}
\label{link}
\end{figure}

\subsection{Continued fractions}\label{CF}

We recall  some notations related to the continued fraction of $\frac{p}{q},$ which will be used in the computation of the Reshetikhin-Turaev invariants.
For a pair of relatively prime integers $(p,q),$ let $$\frac{p}{q}=a_k-\frac{1}{a_{k-1}-\frac{1}{\cdots-\frac{1}{a_1}}}$$ be a  continued fraction.
For each $i\in\{1,\dots,k\},$ consider the matrix
$$ \begin{bmatrix}
A_i & B_i \\
C_i & D_i \end{bmatrix}= T^{a_i}S\cdots T^{a_1}S, $$
where $$S= \begin{bmatrix}
0 & -1 \\
1 & 0 \end{bmatrix} \quad\text{and}\quad T= \begin{bmatrix}
1 & 1 \\
0 & 1 \end{bmatrix},$$  
and as a convention let
\begin{equation}
 \begin{bmatrix}
A_0  \\
C_0  \end{bmatrix}= \begin{bmatrix}
1    \\
0  \end{bmatrix}.
\end{equation}

\begin{lemma}\label{cf} 

\begin{enumerate}[(1)]

\item For $i\in\{1,\dots,k\},$ $A_i=a_iA_{i-1}-C_{i-1}$ and $C_i=A_{i-1}.$


\item 
$$\frac{A_k}{C_k}=\frac{p}{q}.$$

\end{enumerate}
\end{lemma}

\begin{proof} (1)  follows directly from induction. For (2), we show that $$\frac{A_i}{C_i}=a_i-\frac{1}{a_{i-1}-\frac{1}{\cdots-\frac{1}{a_1}}}$$ for each $i\in\{1,\dots,k\}.$
For $i=1,$ we have 
\begin{equation}
 \begin{bmatrix}
A_1  \\
C_1  \end{bmatrix}= \begin{bmatrix}
a_1    \\
1  \end{bmatrix},
\end{equation}
and $\frac{A_1}{C_1}={a_1}.$ Assume (2) holds for $i-1.$ Then by (1), we have
$$\frac{A_i}{C_i}=\frac{a_iA_{i-1}-C_{i-1}}{A_{i-1}}=a_i-\frac{C_{i-1}}{A_{i-1}}=a_i-\frac{1}{a_{i-1}-\frac{1}{\cdots-\frac{1}{a_1}}}.$$ 
\end{proof}

We observe that $A_k$ and $C_k$ are relatively prime because $A_kD_k-B_kC_k=\det(T^{a_k}S\cdots T^{a_1}S)=1.$ Then by Lemma \ref{cf} (2), $ \begin{bmatrix}
A_k\\
C_k \end{bmatrix}= \pm\begin{bmatrix}
p\\
q\end{bmatrix}. $
Since a $(p,q)$ Dehn-surgery and a $(-p,-q)$ Dehn-surgery provide the same $3$-manifold $M,$ we may without loss of generality assume that 
\begin{equation}
 \begin{bmatrix}
A_k \\
C_k\end{bmatrix}= \begin{bmatrix}
p\\
q\end{bmatrix}.
\end{equation}
 As a consequence, we have
\begin{equation}
\begin{bmatrix}
A_{k-1} \\
C_{k-1} \end{bmatrix}= \begin{bmatrix}
q\\
-p+a_kq\end{bmatrix}. 
\end{equation}
We also let 
\begin{equation}\label{p'}
 \begin{bmatrix}
p' \\
q'\end{bmatrix}= \begin{bmatrix}
D_k\\
-B_k\end{bmatrix}
\end{equation}
so that $pp'+qq'=1.$

For $i\in\{1,\dots, k\},$ we also consider the quantity
\begin{equation}\label{K}
K_i=\frac{(-1)^{i+1}\sum_{j=1}^ia_jC_j}{C_i};
\end{equation}
and the following Lemmas \ref{arith'} and \ref{arith} are crucial in the computation of the Reshetikhin-Turaev invariants of $M$ and the study of their asymptotics. The proofs are elementary, and the readers can skip them at the first time and come back later when needed.

\begin{lemma}\label{arith'}
 $C_{k-1}K_{k-1}+C_{k-1}q$ is an even integer.
 \end{lemma}
 
 \begin{proof}
We by Lemma \ref{cf} and the definition of $K_{k-1}$ have 
\begin{equation}\label{mod2}
\begin{split}
K_{k-1}C_{k-1} =& (-1)^k \sum_{i=1}^{k-1}a_iC_i \\
=&  (-1)^k \Big( a_1+a_2A_1+\sum_{i=3}^{k-1}a_iA_{i-1}\Big)\\
=& (-1)^k \Big( A_1+(A_2+C_1)+\sum_{i=3}^{k-1}(A_i+C_{i-1})\Big)\\
=&  (-1)^k \Big(A_1+(A_2+1)+\sum_{i=3}^{k-1}(A_i+A_{i-2})\Big)\\
=& (-1)^k \Big(1+ 2\sum_{i=1}^{k-3}A_i+A_{k-1}+A_{k-2}\Big)
\equiv 1+q+C_{k-1}\quad\quad(\text{mod } 2),
\end{split}
\end{equation}
and
\begin{equation*}
\begin{split}
K_{k-1}C_{k-1}+C_{k-1}q \equiv&  1+q+C_{k-1}+ C_{k-1}q \quad\quad(\text{mod } 2)\\
=&(1+q)(1-p+a_kq).
\end{split}
\end{equation*}
Now if $q$ is odd, then $1+q$ is even and the product is even. If $q$ is even, then $p$ must be odd since $p$ and $q$ are relatively prime. As a result, $1-p+a_kq$ is even and the product is even.
\end{proof}

\begin{lemma}\label{arith}
\begin{enumerate}[(1)] 

\item Let
$$I:\{0,\dots, |q|-1\}\to\{0,\dots,2|q|-1\}$$ be  the map defined by 
$$I(s)=-C_{k-1}(2s+1+K_{k-1})\quad(\text{mod }2|q|).$$
Then $I$ is injective with image the set of integers in $\{0,\dots,2|q|-1\}$ with parity that of $1-q.$ 

In particular, there exist a unique $s^+\in\{0,\dots, |q|-1\}$ and a unique integer $m^+$ such that 
$$I(s^+)=1-q+2m^+q,$$
 and  a unique $s^-\in\{0,\dots, |q|-1\}$  and a unique  integer $m^-$ such that
$$I(s^-)=-1-q+2m^-q.$$
Moreover, 
\begin{equation}\label{+--}
s^+-s^-\equiv p' \quad(\text{mod }q).
\end{equation}

\item Let
$$J:\{0,\dots,|q|-1\}\to\mathbb Q$$
be  the map defined by
$$J(s)=\frac{2s+1}{q}+(-1)^k\sum_{i=1}^{k-1}\frac{(-1)^{i+1}K_i}{C_{i+1}}.$$
Then for the $s^+$ and $s^-$ in (1), 
$$J(s^+)\equiv\frac{p'}{q}\quad(\text{mod }\mathbb Z)$$
and
$$J(s^-)\equiv -\frac{p'}{q}\quad(\text{mod }\mathbb Z).$$

Moverover, 
$$J(s^+)\equiv  -J(s^-)\quad(\text{mod }2\mathbb Z).$$

\item Let 
$$K:\{0,\dots,|q|-1\}\to\mathbb Q$$
be  the map defined by 
$$K(s)=\frac{C_{k-1}(2s+1+K_{k-1})^2}{q}+\sum_{i=1}^{k-2}\frac{C_iK_i^2}{C_{i+1}}.$$
Then for the $s^+$ and $s^-$ in (1), 
$$K(s^+)\equiv -\frac{p'}{q}\quad(\text{mod }\mathbb Z)$$
and
$$K(s^-)\equiv -\frac{p'}{q}\quad(\text{mod }\mathbb Z).$$
\end{enumerate}
\end{lemma}

 \begin{proof}
 For (1), suppose otherwise that there exist distinct $s$ and $s'$ in $\{0,\dots, |q|-1\}$ such that $I(s)=I(s')$ modulo $2|q|.$ Then $2C_{k-1}(s-s')=2hq$ for some integer $h,$ and $\frac{C_{k-1}}{q}=\frac{h}{s-s'}.$ Since $A_{k-1}D_{k-1}-B_{k-1}C_{k-1}=\det(T^{a_{k-1}}S\cdots T^{a_1}S)=1$ and $C_{k-1}$ and $q=A_{k-1}$ are relatively prime, $q\ |\ (s-s'),$ which is a contradiction because $|s-s'|\leqslant |q|-1.$ This proves that the map $I$ is injective. To determine the image of $I,$ for each $s\in\{0,\dots,|q|-1\},$ we by (\ref{mod2}) have
\begin{equation*}
\begin{split}
I(s)=&-C_{k-1}(2s+1)-C_{k-1}K_{k-1}\\
\equiv &-C_{k-1}(2s+1)-1-q-C_{k-1}\quad(\text{mod }2)\\
=& -2C_{k-1}(s+1)-1-q\\
 \equiv& \ 1-q \quad(\text{mod }2).
\end{split}
\end{equation*}
Since $\{0,\dots,2|q|-1\}$ contains exactly $|q|$ even integers and $|q|$ odd integers and $I$ is injective, the image of $I$ consists of all integers with parity that of $1-q.$

For (\ref{+--}), by computing $I(s^+)-I(s^-),$ we have
$$-2C_{k-1}(s^+-s^-)=2+2(m^+-m^-)q,$$
which by the fact that $-C_{k-1}=p-a_kq$ implies 
$$p(s^+-s^-)-\big(a_k(s^+-s^-)+(m^+-m^-)\big)q=1.$$
This completes the proof of (\ref{+--}).
\\

For the first two identities of (2), it suffices to show that
$$pqJ(s^\pm)=\pm 1\quad(\text{mod }q).$$
To this end, for each $i\in\{2,\dots, k\},$ we consider the quantity
$$E_i=C_i\sum_{j=1}^{i-1}\frac{(-1)^{j+1}K_j}{C_{j+1}}.$$ Then by that $p=a_kq-C_{k-1}=a_kC_k-C_{k-1},$ we have
\begin{equation*}
\begin{split}
pqJ(s^\pm)=&(a_kq-C_{k-1})(2s^\pm+1)+(-1)^k(a_kC_k-C_{k-1})q\sum_{i=1}^{k-1}\frac{(-1)^{i+1}K_i}{C_{i+1}}\\
=&-C_{k-1}(2s^\pm+1+K_{k-1})+q\bigg(a_k(2s^\pm+1)\\
&\quad\quad\quad\quad\quad\quad\quad\quad+(-1)^ka_k\Big(C_k\sum_{i=1}^{k-1}\frac{(-1)^{i+1}K_i}{C_{i+1}}\Big)-(-1)^k\Big(C_{k-1}\sum_{i=1}^{k-2}\frac{(-1)^{i+1}K_i}{C_{i+1}}\Big)\bigg)\\
=&I(s^\pm)+q\Big(a_k(2s^\pm+1)+a_k(-1)^kE_k-(-1)^kE_{k-1}\Big).
\end{split}
\end{equation*}
Since $I(s^\pm)\equiv \pm 1$ $(\text{mod }q)$ by (1), the result will follow if we can prove that both $E_{k-1}$ and $E_k$ are integers. For this, by a direct computation we have  $E_2=a_1$ and $E_3=a_2a_1+a_2$ are integers. For $i\geqslant 4,$ by Lemma \ref{CF}, we have $C_i=A_{i-1}=a_{i-1}A_{i-2}-C_{i-2}=a_{i-1}C_{i-1}-C_{i-2}.$ Then 
\begin{equation*}
\begin{split}
E_i=\,&C_i\sum_{j=1}^{i-2}\frac{(-1)^{j+1}K_j}{C_{j+1}}+(-1)^kK_{i-1}\\
=\,&a_{i-1}C_{i-1}\sum_{j=1}^{i-2}\frac{(-1)^{j+1}K_j}{C_{j+1}}-C_{i-2}\sum_{j=1}^{i-2}\frac{(-1)^{j+1}K_j}{C_{j+1}}+(-1)^kK_{i-1}\\
=\,&a_{i-1}C_{i-1}\sum_{j=1}^{i-2}\frac{(-1)^{j+1}K_j}{C_{j+1}}-C_{i-2}\sum_{j=1}^{i-3}\frac{(-1)^{j+1}K_j}{C_{j+1}}-C_{i-2}\frac{(-1)^{k-1}K_{i-2}}{C_{i-1}}+(-1)^kK_{i-1}\\
=\,&a_{i-1}E_{i-1}-E_{i-2}+a_{i-1},
\end{split}
\end{equation*}
where the last equality comes from that
\begin{equation*}
\begin{split}
-C_{i-2}\frac{(-1)^{k-1}K_{i-2}}{C_{i-1}}+(-1)^kK_{i-1}=&\frac{-(-1)^{k-1}C_{i-2}K_{i-2}+(-1)^kK_{i-1}C_{i-1}}{C_{i-1}}\\
=&\frac{-\sum_{j=1}^{i-2}a_jC_j+\sum_{j=1}^{i-1}a_jC_j}{C_{i-1}}=\frac{a_{i-1}C_{i-1}}{C_{i-1}}=a_{i-1}.
\end{split}
\end{equation*}
By induction, all $E_i$'s,  in particular $E_{k-1}$ and $E_k,$ are integers; and the first two identities follow.

For the last identity of (2), by the first two identities, it suffices to show that 
$$J(s^+)-J(s^-)\equiv \frac {2p'}{q}\quad(\text{mod }2\mathbb Z).$$
To see this, we by (\ref{+--}) have $s^+-s^-=p'+nq$ for some integer $n.$ 
Then by the definition of $J,$ 
$$J(s^+)-J(s^-)=\frac{2(s^+-s^-)}{q}=\frac{2p'}{q}+2n,$$
which completes the proof.
\\

For (3), by the definition of $K_i$ and $E_i,$ we first compute 
\begin{equation*}
\begin{split}
\sum_{i=1}^{k-1}\frac{C_iK_i^2}{C_{i+1}}=&\sum_{i=1}^{k-1}\sum_{j=1}^i(-1)^{i+1}\frac{a_jC_jK_i}{C_{i+1}}\\
=&\sum_{j=1}^{k-1}\sum_{i=j}^{k-1}(-1)^{i+1}\frac{a_jC_jK_i}{C_{i+1}}\\
=&\sum_{j=1}^{k-1}\sum_{i=1}^{k-1}(-1)^{i+1}\frac{a_jC_jK_i}{C_{i+1}}-\sum_{j=2}^{k-1}\sum_{i=1}^{j-1}(-1)^{i+1}\frac{a_jC_jK_i}{C_{i+1}}\\
=&\bigg(\sum_{j=1}^{k-1}a_jC_j\bigg)\bigg(\sum_{i=1}^{k-1}\frac{(-1)^{i+1}K_i}{C_{i+1}}\bigg)-\sum_{j=2}^{k-1}a_j\bigg(C_j\sum_{i=1}^{j-1}\frac{(-1)^{i+1}K_i}{C_{i+1}}\bigg)\\
=&(-1)^kC_{k-1}K_{k-1} \sum_{i=1}^{k-1}\frac{(-1)^{i+1}K_i}{C_{i+1}} -\sum_{j=2}^{k-1}a_jE_j\equiv (-1)^kC_{k-1}K_{k-1} \sum_{i=1}^{k-1}\frac{(-1)^{i+1}K_i}{C_{i+1}}\quad(\text{mod } \mathbb Z),
\end{split}
\end{equation*}
where the last equality uses that $E_i$'s are integers from the proof of (3). Then by  that $C_k=q,$
\begin{equation*}
\begin{split}
K(s^\pm)=&\frac{C_{k-1}\big((2s^\pm+1)^2+2(2s^\pm+1)K_{k-1}\big)}{q}+\sum_{i=1}^{k-1}\frac{C_iK_i^2}{C_{i+1}}\\
\equiv &\frac{C_{k-1}\big((2s^\pm+1)^2+2(2s^\pm+1)K_{k-1}\big)}{q}+(-1)^kC_{k-1}K_{k-1}\sum_{i=1}^{k-1}\frac{(-1)^{i+1}K_i}{C_{i+1}}\quad(\text{mod } \mathbb Z)\\
=&\frac{2s^\pm+1}{q}C_{k-1}(2s^\pm+1+K_{k-1})+C_{k-1}K_{k-1}\bigg(\frac{2s^\pm+1}{q}+(-1)^k\sum_{i=1}^{k-1}\frac{(-1)^{i+1}K_i}{C_{i+1}}\bigg)\\
=&-\frac{2s^\pm+1}{q}I(s^\pm)+C_{k-1}K_{k-1}J(s^\pm)\equiv  \mp\frac{2s^\pm+1}{q}\pm\frac{C_{k-1}K_{k-1}p'}{q}\quad(\text{mod } \mathbb Z),
\end{split}
\end{equation*}
where the last equality comes from (1) and (2). To prove the result, it suffices to show that
$$pq\bigg( \mp\frac{2s^\pm+1}{q}\pm\frac{C_{k-1}K_{k-1}p'}{q}\bigg)\equiv -1\quad(\text{mod }q).$$
To this end, since $pp'+qq'=1$ and $p=a_kq-C_{k-1},$ we have 
\begin{equation*}
\begin{split}
pq\bigg( -\frac{2s^++1}{q}+\frac{C_{k-1}K_{k-1}p'}{q}\bigg)=&-p(2s^++1)+C_{k-1}K_{k-1}pp'\\
\equiv & -p(2s^++1)+C_{k-1}K_{k-1}\quad(\text{mod }q) \\
=&-(a_kq-C_{k-1})(2s^++1)+C_{k-1}K_{k-1} \\
=&-a_kq(2s^++1) -I(s^+) \equiv -1\quad(\text{mod }q),
\end{split}
\end{equation*}
where the last equality comes from (2); and
\begin{equation*}
\begin{split}
pq\bigg( \frac{2s^-+1}{q}-\frac{C_{k-1}K_{k-1}p'}{q}\bigg)&=p(2s^-+1)-C_{k-1}K_{k-1}pp'\\
&\equiv p(2s^-+1)-C_{k-1}K_{k-1}\quad(\text{mod }q) \\
&=(a_kq-C_{k-1})(2s^-+1)-C_{k-1}K_{k-1} \\
&=a_kq(2s^-+1) +I(s^-) \equiv -1\quad(\text{mod }q),
\end{split}
\end{equation*}
where the last equality comes from (2). This completes the proof. \end{proof}


\subsection{The computation}

\begin{proposition}\label{formula}  For an odd integer $r\geqslant 3$ and at the root of unity $t=e^{\frac{4\pi\sqrt{-1}}{r}},$ the $r$-th Reshetikin-Turaev invariant of the closed oriented $3$-manifold $M$ obtained by doing the $\frac{p}{q}$ Dehn-surgery along the figure-$8$ knot can be computed as
$$
\mathrm{RT}_r(M)=\kappa_r\sum_{s=0}^{|q|-1}\sum_{m=-\frac{r-2}{2}}^{\frac{r-2}{2}}\sum_{n=\max\{-m,m\}}^{\frac{r-2}{2}}g_r(s,m,n)$$
where
$$\kappa_r=\frac{(-1)^{\frac{3(k+1)}{4}+\sum_{i=1}^ka_i}e^{\frac{\pi\sqrt{-1}}{r}\big(3\sigma(L)-\sum_{i=1}^ka_i-\sum_{i=2}^k\frac{1}{C_{i-1}C_i}\big)+\frac{\pi\sqrt{-1}r}{4}\big(\sigma(L)+3a_k\big)}}{2r\sqrt{q}}$$
with $\sigma(L)$ the signature of the linking matrix of the link $L$ in Figure \ref{link}, $C_i$ and $K_i$ as defined in Section \ref{CF}, 
the first summation over integers $s$ in between $0$ and $|q|-1,$
the second summation over half-integers $m$ in  between $-\frac{r-2}{2}$ and $\frac{r-2}{2}$ and the third summation over half-integers $n$ in between $\max\{-m,m\}$ and $\frac{r-2}{2}.$
For $s\in\{1,\dots, |q|-1\},$ let $I(s),$ $J(s)$ and $K(s)$ be as defined in Lemma \ref{arith}. Then
$$g_r(s,m,n)=\sin\bigg(\frac{x}{q}-J(s)\pi\bigg)\epsilon\bigg(\frac{2\pi m}{r},\frac{2\pi n}{r}\bigg)e^{-\frac{2\pi m\sqrt{-1}}{r}+\frac{r}{4\pi\sqrt{-1}}V_r\big(s,\frac{2\pi m}{r},\frac{2\pi n}{r}\big)}$$
where the functions $\epsilon(x,y)$ and $V_r(s,x,y)$ are defined as follows: Let $\varphi_r$ be the quantum dilogarithm function as defined by (\ref{qd}).
\begin{enumerate}[(1)]
\item If both $0<y\pm x< \pi,$ then $\epsilon(x,y)=2$ and
$$V_r(s,x,y)=-\frac{px^2}{q}+I(s)\frac{2\pi x}{q}+4xy-\varphi_r\Big(\pi-y-x-\frac{\pi}{r}\Big)+\varphi_r\Big(y-x+\frac{\pi}{r}\Big)+K(s)\pi^2.$$
\item If $0< y+x<\pi$ and $\pi < y-x < 2\pi,$ then $\epsilon(x,y)=1$ and
$$V_r(s,x,y)=-\frac{px^2}{q}+I(s)\frac{2\pi x}{q}+4xy-\varphi_r\Big(\pi-y-x-\frac{\pi}{r}\Big)+\varphi_r\Big(y-x+\pi-\frac{\pi}{r}\Big)+K(s)\pi^2.$$
\item If $\pi < y+x< 2\pi$ and $0 < y-x < \pi,$ then $\epsilon(x,y)=1$ and
$$V_r(s,x,y)=-\frac{px^2}{q}+I(s)\frac{2\pi x}{q}+4xy-\varphi_r\Big(2\pi-y-x-\frac{\pi}{r}\Big)+\varphi_r\Big(y-x+\frac{\pi}{r}\Big)+K(s)\pi^2.$$
\end{enumerate}
\end{proposition}

\begin{remark}Here by half-integers, we mean rational numbers of the form $n+\frac{1}{2},$ $n\in\mathbb Z.$
\end{remark}

\begin{proof}[Proof of Proposition \ref{formula}] A  direct computation shows that
$$\langle\mu_r \omega_r\rangle _{U_+}=e^{\big(-\frac{3}{r}-\frac{r+1}{4}\big)\pi\sqrt{-1}}.$$
Let 
$$\kappa_r'=\mu_r^{k+1}\langle\mu_r \omega_r\rangle _{U_+}^{-\sigma(L)}=\bigg(\frac{\sin\frac{2\pi}{r}}{\sqrt r}\bigg)^{k+1}e^{-\sigma(L)\big(-\frac{3}{r}-\frac{r+1}{4}\big)\pi\sqrt{-1}}.$$
Then by (\ref{RT}), we have
\begin{equation*}
\begin{split}
\mathrm{RT}_r(M)&=\kappa_r' \langle \omega_r,\dots,\omega_r\rangle_{D(L)}  \\
&=\kappa_r'\sum_{m_1,\dots, m_k=0}^{r-2}(-1)^{m_k+\sum_{i=1}^ka_im_i}t^{\sum_{i=1}^n\frac{a_im_i(m_i+2)}{4}}[m_1+1]\prod_{i=1}^{k-1}[(m_i+1)(m_{i+1}+1)]\langle e_{m_k} \rangle_{D(K_{4_1})},
\end{split}
\end{equation*}
where the second equality comes from the fact that $e_m$ is an eigenvector of the positive and the negative twist operator of eigenvalue $(-1)^mt^{\pm\frac{m(m+2)}{4}}$, and is also an eigenvector of the circle operator $c(e_n)$ (defined by enclosing $e_m$ by $e_n$) of eigenvalue $(-1)^n\frac{[(m+1)(n+1)]}{[m+1]}.$
By Habiro's formula\,\cite{Ha} (see also \cite{Ma} for a skein theoretical computation) 
\begin{equation*}
\begin{split}
\langle e_m \rangle_{D(K_{4_1})}&=(-1)^m[m+1]J'_{m+1}(K_{4_1})\\
&=\frac{(-1)^m}{\{1\}}\sum_{n=0}^{\min\{m,r-2-m\}}\frac{\{m+1+n\}!}{\{m-n\}!}\\
&=\frac{(-1)^{m+1}}{\{1\}}\sum_{n=0}^{\min\{m,r-2-m\}}t^{-(m+1)(n+\frac{1}{2})}\frac{(t)_{m+1+n}}{(t)_{m-n}},
\end{split}
\end{equation*}
where $J'_{m}(K)$ is the $m$-th normalized colored Jones polynomial so that $J'_m(\text{unknot})=1.$ Here $\{m\}=t^{\frac{m}{2}}-t^{-\frac{m}{2}},$ $\{m\}!=\prod_{k=1}^m\{k\}$ and $(t)_m=\prod_{k=1}^m(1-t^k).$

Then
\begin{equation*}\label{sum}
\begin{split}
\mathrm{RT}_r(M)=&-\frac{\kappa_r'}{\{1\}}\sum_{m_1,\dots, m_k=0}^{r-2}\sum_{n=0}^{\min\{m_k,r-2-m_k\}}(-1)^{\sum_{i=1}^ka_im_i}t^{\sum_{i=1}^k\frac{a_im_i(m_i+2)}{4}-(m_k+1)(n+\frac{1}{2})}\\
&\quad\quad\quad\quad\quad\quad\quad\quad\quad\quad\quad\quad [m_1+1]\prod_{i=1}^{k-1}[(m_i+1)(m_{i+1}+1)]\frac{(t)_{m_k+1+n}}{(t)_{m_k-n}}\\
=&-\frac{\kappa_r'}{\{1\}}\sum_{m_1,\dots, m_{k-1}=0}^{r-1}\sum_{m_k=1}^{r-1}\sum_{n=0}^{\min\{m_k-1,r-1-m_k\}}(-1)^{\sum_{i=1}^ka_i(m_i-1)}t^{\sum_{i=1}^k\frac{a_i(m_i^2-1)}{4}-m_k(n+\frac{1}{2})}\\
&\quad\quad\quad\quad\quad\quad\quad\quad\quad\quad\quad\quad\quad\quad\quad\quad\quad\quad\quad [m_1]\prod_{i=1}^{k-1}[m_im_{i+1}]\frac{(t)_{m_k+n}}{(t)_{m_k-n-1}}\\
\end{split}
\end{equation*}
where the last equality is obtained by changing the variables $m_i$ to $m_i-1$ for $i\in\{1,\dots,k\}$ and the fact that $[0]=0.$ By reordering the summations, we have
\begin{equation*}
\begin{split}
&\mathrm{RT}_r(M)\\
=&-\frac{\kappa_r'}{\{1\}}\sum_{m_k=1}^{r-1}\sum_{n=0}^{\min\{m_k-1,r-1-m_k\}}\bigg(\sum_{m_1,\dots, m_{k-1}=0}^{r-1}(-1)^{\sum_{i=1}^{k-1}a_i(m_i-1)}t^{\sum_{i=1}^{k-1}\frac{a_i(m_i^2-1)}{4}}[m_1]\prod_{i=1}^{k-1}[m_im_{i+1}]\bigg)\\
&\quad\quad\quad\quad\quad\quad\quad\quad\quad\quad\quad\quad\quad\quad\quad\quad(-1)^{a_k(m_k-1)}t^{\frac{a_k(m_k^2-1)}{4}-m_k(n+\frac{1}{2})}\frac{(t)_{m_k+n}}{(t)_{m_k-n-1}} \\
=&\,\kappa_r''\sum_{m_k=1}^{r-1}\sum_{n=0}^{\min\{m_k-1,r-1-m_k\}}S(m_k)(-1)^{a_km_k}t^{\frac{a_km_k^2}{4}-m_k(n+\frac{1}{2})}\frac{(t)_{m_k+n}}{(t)_{m_k-n-1}},
\end{split}
\end{equation*}
where
$$\kappa_r''=-\frac{(-1)^{\sum_{i=1}^ka_i}t^{-\sum_{i=1}^k\frac{a_i}{4}}}{\{1\}^{k+1}}\kappa_r'$$
is a constant independent of $m_1,\dots,m_k,$ and
$$S(m_k)=\sum_{m_1,\dots, m_{k-1}=0}^{r-1}(-1)^{\sum_{i=1}^{k-1}a_im_i}t^{\sum_{i=1}^{k-1}\frac{a_im_i^2}{4}}\big(t^{\frac{m_1}{2}}-t^{-\frac{m_1}{2}}\big)\prod_{i=1}^{k-1}\big(t^{\frac{m_im_{i+1}}{2}}-t^{-\frac{m_im_{i+1}}{2}}\big)$$
is a quantity depending only on $m_k.$ By Lemma \ref{sm} in Section \ref{Lem}, the multi-sum $S(m_k)$ can be computed as the following single sum
\begin{equation*}
\begin{split}
S(m_k)=\,&\tau^+\sum_{s=0}^{2|q|-1}e^{-\frac{\pi\sqrt{-1}}{r}\frac{C_{k-1}}{q}\big(m_k+sr+\frac{K_{k-1}r}{2}+\frac{(-1)^k}{C_{k-1}}\big)^2}-\tau^-\sum_{s=0}^{2|q|-1}e^{-\frac{\pi\sqrt{-1}}{r}\frac{C_{k-1}}{q}\big(m_k+sr+\frac{K_{k-1}r}{2}-\frac{(-1)^k}{C_{k-1}}\big)^2},
\end{split}
\end{equation*}
where 
$$\tau^{\pm}=\frac{(-1)^{\frac{k-1}{4}}}{\sqrt q}2^{k-2}r^{\frac{k-1}{2}}e^{-\frac{\pi\sqrt{-1}}{r}\sum_{i=1}^{k-2}\frac{1}{C_iC_{i+1}}-\frac{\pi\sqrt{-1}r}{4}\sum_{i=1}^{k-2}\frac{C_iK_i^2}{C_{i+1}}\mp\pi\sqrt{-1}\sum_{i=1}^{k-2}\frac{(-1)^{i+1}K_i}{C_{i+1}}}$$
are  constants independent of $m_k.$

Then we observe that for each $s\in\{0,\dots, |q|-1\},$ 
$$e^{-\frac{\pi\sqrt{-1}}{r}\frac{C_{k-1}}{q}\big(m_k+(s+q)r+\frac{K_{k-1}r}{2}\pm\frac{(-1)^k}{C_{k-1}}\big)^2}=e^{-\frac{\pi\sqrt{-1}}{r}\frac{C_{k-1}}{q}\big(m_k+sr+\frac{K_{k-1}r}{2}\pm\frac{(-1)^k}{C_{k-1}}\big)^2}.$$
Indeed, since all of $C_{k-1},$ $m_k,$ $s,$ $q$ and $r$ are integers, a direct computation shows that
$$\frac{e^{-\frac{\pi\sqrt{-1}}{r}\frac{C_{k-1}}{q}\big(m_k+(s+q)r+\frac{K_{k-1}r}{2}\pm\frac{(-1)^k}{C_{k-1}}\big)^2}}{e^{-\frac{\pi\sqrt{-1}}{r}\frac{C_{k-1}}{q}\big(m_k+sr+\frac{K_{k-1}r}{2}\pm\frac{(-1)^k}{C_{k-1}}\big)^2}}=e^{-\pi\sqrt{-1}r(K_{k-1}C_{k-1}+C_{k-1}q)}=1,$$
where the last equality comes from Lemma \ref{arith'}  that  $K_{k-1}C_{k-1}+C_{k-1}q$ is an even integer. 

As a consequence we have
\begin{equation*}
\begin{split}
S(m_k)=\,&2\tau^+\sum_{s=0}^{|q|-1}e^{-\frac{\pi\sqrt{-1}}{r}\frac{C_{k-1}}{q}\big(m_k+sr+\frac{K_{k-1}r}{2}+\frac{(-1)^k}{C_{k-1}}\big)^2}-2\tau^-\sum_{s=0}^{|q|-1}e^{-\frac{\pi\sqrt{-1}}{r}\frac{C_{k-1}}{q}\big(m_k+sr+\frac{K_{k-1}r}{2}-\frac{(-1)^k}{C_{k-1}}\big)^2}\\
=\,&2\tau\Bigg(\sum_{s=0}^{|q|-1}e^{-\frac{\pi\sqrt{-1}}{r}\frac{C_{k-1}}{q}\Big(\big(m_k+sr+\frac{K_{k-1}r}{2}\big)^2+\frac{1}{C^2_{k-1}}\Big)-\pi\sqrt{-1}\Big(\frac{(-1)^k}{rq}\big(2m_k+2sr+K_{k-1}r\big)+\sum_{i=1}^{k-2}\frac{(-1)^{i+1}K_i}{C_{i+1}}\Big)}\\
&\quad-\sum_{s=0}^{|q|-1}e^{-\frac{\pi\sqrt{-1}}{r}\frac{C_{k-1}}{q}\Big(\big(m_k+sr+\frac{K_{k-1}r}{2}\big)^2+\frac{1}{C^2_{k-1}}\Big)+\pi\sqrt{-1}\Big(\frac{(-1)^k}{rq}\big(2m_k+2sr+K_{k-1}r\big)+\sum_{i=1}^{k-2}\frac{(-1)^{i+1}K_i}{C_{i+1}}\Big)}\Bigg)\\
=\,&\tau' \sum_{s=0}^{|q|-1}e^{-\frac{\pi\sqrt{-1}}{r}\frac{C_{k-1}}{q}\big(m_k+sr+\frac{K_{k-1}r}{2}\big)^2}\sin\bigg(-\pi\Big(\frac{(-1)^k}{rq}\big(2m_k+2sr\big)+\sum_{i=1}^{k-1}\frac{(-1)^{i+1}K_i}{C_{i+1}}\Big)\bigg),
\end{split}
\end{equation*}
where  in the computation of the quantity in the sine we use that $C_k=q,$
$$\tau=\frac{(-1)^{\frac{k-1}{4}}}{\sqrt q}2^{k-2}r^{\frac{k-1}{2}}e^{-\frac{\pi\sqrt{-1}}{r}\sum_{i=1}^{k-2}\frac{1}{C_iC_{i+1}}-\frac{\pi\sqrt{-1}r}{4}\sum_{i=1}^{k-2}\frac{C_iK_i^2}{C_{i+1}}}$$
and 
\begin{equation*}
\begin{split}
\tau'=&4\sqrt{-1}\tau e^{-\frac{\pi\sqrt{-1}}{r}\frac{C_{k-1}}{q}\frac{1}{C^2_{k-1}}}=\frac{(-1)^{\frac{k+1}{4}}}{\sqrt q}2^kr^{\frac{k-1}{2}}e^{-\frac{\pi\sqrt{-1}}{r}\sum_{i=1}^{k-1}\frac{1}{C_iC_{i+1}}-\frac{\pi\sqrt{-1}r}{4}\sum_{i=1}^{k-2}\frac{C_iK_i^2}{C_{i+1}}}.
\end{split}
\end{equation*}
 As a consequence, 
\begin{equation*}
\begin{split}
\mathrm{RT}_r(M)=\, & \kappa_r''\tau' \sum_{s=0}^{|q|-1} \sum_{m_k=1}^{r-1}\sum_{n=0}^{\min\{m_k-1,r-1-m_k\}}e^{-\frac{\pi\sqrt{-1}}{r}\frac{C_{k-1}}{q}\big(m_k+sr+\frac{K_{k-1}r}{2}\big)^2}\\
&\sin\bigg(-\pi\Big(\frac{(-1)^k}{rq}\big(2m_k+2sr\big)+\sum_{i=1}^{k-1}\frac{(-1)^{i+1}K_i}{C_{i+1}}\Big)\bigg)(-1)^{a_km_k}t^{\frac{a_km_k^2}{4}-m_k(n+\frac{1}{2})}\frac{(t)_{m_k+n}}{(t)_{m_k-n-1}}.
\end{split}
\end{equation*}

Letting $m'=\frac{r}{2}-m_k,$ $n'=\frac{r-2}{2}-n$ and
$$\kappa_r=(-1)^{\frac{3a_kr}{4}+k}e^{\frac{\pi\sqrt{-1}r}{4}\sum_{i=1}^{k-2}\frac{C_iK_i^2}{C_{i+1}}}\kappa_r''\tau',$$  we have
\begin{equation}\label{estim}
\begin{split}
\mathrm{RT}_r(M)=\,&\kappa_r\sum_{s=0}^{|q|-1}\sum_{m'=-\frac{r-2}{2}}^{\frac{r-2}{2}}\sum_{n'=\max\{-m',m'\}}^{\frac{r-2}{2}}e^{-\frac{\pi\sqrt{-1}}{r}\frac{C_{k-1}}{q}\big(-m'+\frac{(2s+1+K_{k-1})r}{2}\big)^2-\frac{\pi\sqrt{-1}r}{4}\sum_{i=1}^{k-2}\frac{C_iK_i^2}{C_{i+1}}}\\
&\quad\quad\quad\quad\quad\quad\quad\quad \sin\bigg(\frac{2\pi m'}{r}\frac{1}{q}-J(s)\pi\bigg)t^{\frac{a_km'^2}{4}-m'(n'+\frac{1}{2})}\frac{(t)_{r-m'-n'-1}}{(t)_{n'-m'}}.
\end{split}
\end{equation}

By a direct computation and that $C_k=q,$ we have
$$\kappa_r=\frac{(-1)^{\frac{3(k+1)}{4}+\sum_{i=1}^ka_i}e^{\frac{\pi\sqrt{-1}}{r}\big(3\sigma(L)-\sum_{i=1}^ka_i-\sum_{i=2}^k\frac{1}{C_{i-1}C_i}\big)+\frac{\pi\sqrt{-1}r}{4}\big(\sigma(L)+3a_k\big)}}{2r\sqrt{q}},$$
\begin{equation*}
\begin{split}
&e^{-\frac{\pi\sqrt{-1}}{r}\frac{C_{k-1}}{q}\big(-m'+\frac{(2s+1+K_{k-1})r}{2}\big)^2-\frac{\pi\sqrt{-1}r}{4}\sum_{i=1}^{k-2}\frac{C_iK_i^2}{C_{i+1}}}=e^{\frac{r}{4\pi\sqrt{-1}}\Big(\frac{C_{k-1}}{q}\big(\frac{2\pi m'}{r}\big)^2+\frac{2\pi I(s)}{q}\big(\frac{2\pi m'}{r}\big)+K(s)\pi^2\Big)},
\end{split}
\end{equation*}
and
\begin{equation*}
\begin{split}
t^{\frac{a_km'}{4}-m'(n'+\frac{1}{2})}=e^{-\frac{2\pi m'\sqrt{-1}}{r}+\frac{r}{4\pi \sqrt{-1}}\Big(-a_k\big(\frac{2\pi m'}{r}\big)^2+4\big(\frac{2\pi m'}{r}\big)
\big(\frac{2\pi n'}{r}\big)\Big)};
\end{split}
\end{equation*}
and by Lemma \ref{factorial}, we have:
\begin{enumerate} [(1)]
\item If $0\leqslant n'\pm m'\leqslant \frac{r-1}{2},$ then
$$\frac{(t)_{r-m'-n'-1}}{(t)_{n-m'}}=2e^{\frac{4\pi\sqrt{-1}}{r}\Big(-\varphi_r\big(\pi-\frac{2\pi n'}{r}-\frac{2\pi m'}{r}-\frac{\pi}{r}\big)+\varphi_r\big(\frac{2\pi n'}{r}-\frac{2\pi m'}{r}+\frac{\pi}{r}\big)\Big)}.$$

\item If $0\leqslant n'+ m'\leqslant \frac{r-1}{2}$ and $\frac{r-1}{2}\leqslant n'-m'\leqslant r-1,$ then
$$\frac{(t)_{r-m'-n'-1}}{(t)_{n-m'}}=e^{\frac{4\pi\sqrt{-1}}{r}\Big(-\varphi_r\big(\pi-\frac{2\pi n'}{r}-\frac{2\pi m'}{r}-\frac{\pi}{r}\big)+\varphi_r\big(\frac{2\pi n'}{r}-\frac{2\pi m'}{r}-\pi+\frac{\pi}{r}\big)\Big)}.$$

\item If $\frac{r-1}{2}\leqslant n'+m'\leqslant r-1$ and $0\leqslant n'-m'\leqslant \frac{r-1}{2},$ then
$$\frac{(t)_{r-m'-n'-1}}{(t)_{n-m'}}=e^{\frac{4\pi\sqrt{-1}}{r}\Big(-\varphi_r\big(2\pi-\frac{2\pi n'}{r}-\frac{2\pi m'}{r}-\frac{\pi}{r}\big)+\varphi_r\big(\frac{2\pi n'}{r}-\frac{2\pi m'}{r}+\frac{\pi}{r}\big)\Big)}.$$
\end{enumerate}
Putting all these together and using the fact that $C_{k-1}=-p+a_kq,$ we complete the proof.
\end{proof}


\subsection{Lemma \ref{sm}} \label{Lem}

\begin{lemma}\label{sm} For each $i\in\{1,\dots,k\}$ and two non-zero integers $m_0$ and $m_{i+1},$ let 
$$S_i(m_0,m_{i+1})=\sum_{m_1,\dots,m_i=0}^{r-1}(-1)^{\sum_{j=1}^ia_j}t^{\sum_{j=1}^i\frac{a_jm_j^2}{4}}\prod_{j=0}^i(t^{\frac{m_jm_{j+1}}{2}}-t^{-\frac{m_jm_{j+1}}{2}}).$$ Let $A_i,$ $C_i$ and $K_i$ be the quantities introduced in Section \ref{CF}. Then
\begin{equation*}
\begin{split}
S_i(m_0,m_{i+1})=\,&\tau^+_i\sum_{s=0}^{2|A_i|-1}e^{-\frac{\pi\sqrt{-1}}{r}\frac{C_i}{A_i}\big(m_{i+1}+sr+\frac{K_ir}{2}+\frac{(-1)^{i+1}m_0}{C_i}\big)^2}\\
&\quad\quad\quad\quad\quad\quad\quad\quad-\tau^-_i\sum_{s=0}^{2|A_i|-1}e^{-\frac{\pi\sqrt{-1}}{r}\frac{C_i}{A_i}\big(m_{i+1}+sr+\frac{K_ir}{2}-\frac{(-1)^{i+1}m_0}{C_i}\big)^2},
\end{split}
\end{equation*}
where 
$$\tau_i^{\pm}=\frac{(-1)^{\frac{i}{4}}}{\sqrt A_i}2^{i-1}r^{\frac{i}{2}}e^{-\frac{\pi\sqrt{-1}}{r}m_0^2\sum_{j=1}^{i-1}\frac{1}{C_jC_{j+1}}-\frac{\pi\sqrt{-1}r}{4}\sum_{j=1}^{i-1}\frac{C_jK_j^2}{C_{j+1}}\mp\pi\sqrt{-1}m_0\sum_{j=1}^{i-1}\frac{(-1)^{j+1}K_j}{C_{j+1}}}.$$
In particular, the quantity $S(m_k)$ in the proof of Proposition \ref{formula} equals $S_{k-1}(1,m_k),$ and by $A_{k-1}=q,$ we have
\begin{equation*}
\begin{split}
S(m_k)=\,&\tau^+_{k-1}\sum_{s=0}^{2|q|-1}e^{-\frac{\pi\sqrt{-1}}{r}\frac{C_{k-1}}{q}\big(m_k+sr+\frac{K_{k-1}r}{2}+\frac{(-1)^k}{C_{k-1}}\big)^2}\\
&\quad\quad\quad\quad\quad\quad\quad\quad-\tau^-_{k-1}\sum_{s=0}^{2|q|-1}e^{-\frac{\pi\sqrt{-1}}{r}\frac{C_{k-1}}{q}\big(m_k+sr+\frac{K_{k-1}r}{2}-\frac{(-1)^k}{C_{k-1}}\big)^2}.
\end{split}
\end{equation*}
\end{lemma}

The proof of Lemma \ref{sm} relies on the following reciprocity of generalized Gaussian sums.

\begin{proposition}\cite[Proposition 2.3]{J}\label{R}
For $m,n\in\mathbb Z,$ if $mn$ is even and $n\psi\in\mathbb Z,$ then
$$\sum_{\lambda=0}^{|n|-1}e^{\frac{m\lambda^2\pi \sqrt{-1} }{n}}e^{2\psi\lambda \pi \sqrt{-1} }=\bigg(\frac{\sqrt{-1}n}{m}\bigg)^{\frac{1}{2}}\sum_{s=0}^{|m|-1}e^{-\frac{n(s+\psi)^2\pi \sqrt{-1}}{m}}.$$
\end{proposition}

\begin{proof}[Proof of Lemma \ref{sm}] We first observe that for any integer $a,$ the quantities
$(-1)^{am}t^{\frac{am^2}{4}}$ and $t^{\frac{am}{2}}$
are periodic in $m$ with period $r.$ As a consequence, we have
\begin{equation}\label{+r}
S_i(m_0,m_{i+1}+r)=S_{i}(m_0,m_{i+1})\quad\text{and}\quad S_i(m_0,-m_{i+1})=-S_{i}(m_0,m_{i+1}).
\end{equation}

Now we use induction. For $i=1$ and non-zero integers $m_0$ and $m_2,$ we have
\begin{equation*}
\begin{split}
S_1(m_0,m_2)=&\sum_{m_1=0}^{r-1}(-1)^{a_1m_1}t^{\frac{a_1m_1^2}{4}}(t^{\frac{m_0m_1}{2}}-t^{-\frac{m_0m_1}{2}})(t^{\frac{m_1m_2}{2}}-t^{-\frac{m_1m_2}{2}})\\
=&\sum_{m_1=0}^{r-1}(-1)^{a_1m_1}t^{\frac{a_1m_1^2}{4}}(t^{\frac{m_0m_1}{2}}-t^{-\frac{m_0m_1}{2}})t^{\frac{m_1m_2}{2}}\\
&\quad\quad\quad\quad-\sum_{m_1=0}^{r-1}(-1)^{a_1m_1}t^{\frac{a_1m_1^2}{4}}(t^{\frac{m_0m_1}{2}}-t^{-\frac{m_0m_1}{2}})t^{-\frac{m_1m_2}{2}}.
\end{split}
\end{equation*}
For the second sum, we have
\begin{equation*}
\begin{split}
-\sum_{m_1=0}^{r-1}(-1)^{a_1m_1}t^{\frac{a_1m_1^2}{4}}&(t^{\frac{m_0m_1}{2}}-t^{-\frac{m_0m_1}{2}})t^{-\frac{m_1m_2}{2}}\\
=&\sum_{m_1=0}^{r-1}(-1)^{a_1(-m_1)}t^{\frac{a_1(-m_1)^2}{4}}(t^{\frac{m_0(-m_1)}{2}}-t^{-\frac{m_0(-m_1)}{2}})t^{\frac{(-m_1)m_2}{2}}\\
=&\sum_{m_1=-r+1}^0(-1)^{a_1m_1}t^{\frac{a_1m_1^2}{4}}(t^{\frac{m_0m_1}{2}}-t^{-\frac{m_0m_1}{2}})t^{\frac{m_1m_2}{2}}\\
=&\sum_{m_1=r}^{2r-1}(-1)^{a_1m_1}t^{\frac{a_1m_1^2}{4}}(t^{\frac{m_0m_1}{2}}-t^{-\frac{m_0m_1}{2}})t^{\frac{m_1m_2}{2}},
\end{split}
\end{equation*}
where the last equality comes from the periodicity of the summands in $m_1$ and that both $\{-r+1,\dots, 0\}$ and $\{r,\dots, 2r-1\}$ are a full period. Therefore,
\begin{equation*}
\begin{split}
S_1(m_0,m_2)=&\sum_{m_1=0}^{2r-1}(-1)^{a_1m_1}t^{\frac{a_1m_1^2}{4}}(t^{\frac{m_0m_1}{2}}-t^{-\frac{m_0m_1}{2}})t^{\frac{m_1m_2}{2}}\\
=&\sum_{m_1=0}^{2r-1}(-1)^{a_1m_1}t^{\frac{a_1m_1^2}{4}}t^{\frac{m_0m_1}{2}}t^{\frac{m_1m_2}{2}}-\sum_{m_1=0}^{2r-1}(-1)^{a_1m_1}t^{\frac{a_1m_1^2}{4}}t^{-\frac{m_0m_1}{2}}t^{\frac{m_1m_2}{2}}\\
=&\sum_{m_1=0}^{2r-1}e^{\frac{2a_1m_1^2\pi\sqrt{-1}}{2r}}e^{2\big(\frac{m_2+m_0}{r}+\frac{a_1}{2}\big)m_1\pi\sqrt{-1}}-\sum_{m_1=0}^{2r-1}e^{\frac{2a_1m_1^2\pi\sqrt{-1}}{2r}}e^{2\big(\frac{m_2-m_0}{r}+\frac{a_1}{2}\big)m_1\pi\sqrt{-1}}.
\end{split}
\end{equation*}
Then by Proposition \ref{R} with $m=2a_1,$ $n=2r$ and $\psi=\frac{m_2\pm m_0}{r}+\frac{a_1}{2},$ we have 
\begin{equation*}
\begin{split}
S_1(m_0,m_2)=&\Big(\frac{\sqrt{-1}r}{a_1}\Big)^{\frac{1}{2}}\sum_{s=0}^{2|a_1|-1} e^{-\frac{2r\big(s+\frac{m_2+m_0}{r}+\frac{a_1}{2}\big)^2\pi\sqrt{-1}}{2a_1}}\\
&\quad\quad\quad\quad\quad\quad\quad\quad-\Big(\frac{\sqrt{-1}r}{a_1}\Big)^{\frac{1}{2}}\sum_{s=0}^{2|a_1|-1} e^{-\frac{2r\big(s+\frac{m_2-m_0}{r}+\frac{a_1}{2}\big)^2\pi\sqrt{-1}}{2a_1}}\\
=&\Big(\frac{\sqrt{-1}r}{A_1}\Big)^{\frac{1}{2}}\sum_{s=0}^{2|A_1|-1} e^{-\frac{\pi\sqrt{-1}}{r}\frac{C_1}{A_1}\big(m_2+sr+\frac{K_1r}{2}+\frac{(-1)^2m_0}{C_1}\big)^2}\\
&\quad\quad\quad\quad\quad\quad\quad\quad-\Big(\frac{\sqrt{-1}r}{A_1}\Big)^{\frac{1}{2}}\sum_{s=0}^{2|A_1|-1} e^{-\frac{\pi\sqrt{-1}}{r}\frac{C_1}{A_1}\big(m_2+sr+\frac{K_1r}{2}-\frac{(-1)^2m_0}{C_1}\big)^2}
\end{split}
\end{equation*}
where the last equality uses the fact that $A_1=K_1=a_1$ and $C_1=1.$ In this case, we have $\tau^\pm_1=1.$
\\

Now assume that 
\begin{equation}\label{i-1}
\begin{split}
S_{i-1}(m_0,m_i)=\,&\tau^+_{i-1}\sum_{s=0}^{2|A_{i-1}|-1}e^{-\frac{\pi\sqrt{-1}}{r}\frac{C_{i-1}}{A_{i-1}}\big(m_{i}+sr+\frac{K_{i-1}r}{2}+\frac{(-1)^{i}m_0}{C_{i-1}}\big)^2}\\
&\quad\quad\quad\quad\quad-\tau^-_{i-1}\sum_{s=0}^{2|A_{i-1}|-1}e^{-\frac{\pi\sqrt{-1}}{r}\frac{C_{i-1}}{A_{i-1}}\big(m_{i}+sr+\frac{K_{i-1}r}{2}-\frac{(-1)^{i}m_0}{C_{i-1}}\big)^2}.
\end{split}
\end{equation}
By (\ref{+r}), we have
\begin{equation*}
\begin{split}
S_i(m_0,m_{i+1})=&\sum_{m_i=0}^{r-1}(-1)^{a_im_i}t^{\frac{a_im_i^2}{4}}(t^{\frac{m_im_{i+1}}{2}}-t^{-\frac{m_im_{i+1}}{2}})S_{i-1}(m_0,m_i)\\
=&\sum_{m_i=0}^{r-1}(-1)^{a_im_i}t^{\frac{a_im_i^2}{4}}t^{\frac{m_im_{i+1}}{2}}S_{i-1}(m_0,m_i)\\
&\quad\quad+\sum_{m_i=0}^{r-1}(-1)^{a_i(-m_i)}t^{\frac{a_i(-m_i)^2}{4}}t^{\frac{(-m_i)m_{i+1}}{2}}S_{i-1}(m_0,-m_i)\\
=&\Big(\sum_{m_i=0}^{r-1}+\sum_{m_i=-r+1}^0\Big)(-1)^{a_im_i}t^{\frac{a_im_i^2}{4}}t^{\frac{m_im_{i+1}}{2}}S_{i-1}(m_0,m_i)\\
=&2\sum_{m_i=0}^{r-1}(-1)^{a_im_i}t^{\frac{a_im_i^2}{4}}t^{\frac{m_im_{i+1}}{2}}S_{i-1}(m_0,m_i),\\
\end{split}
\end{equation*}
where the last equality comes from that both $\{-r+1,\dots,0\}$ and $\{0,\dots, r-1\}$ are a full period for $m_i.$ Since the quantities
$(-1)^{am}t^{\frac{am^2}{4}}$ and $t^{\frac{am}{2}}$
 have period $r$ in $m,$ we have for any integer $s,$ 
\begin{equation*}
\begin{split}
(-1)^{a_im_i}t^{\frac{a_im_i^2}{4}}t^{\frac{m_im_{i+1}}{2}}=&(-1)^{a_i(m_i+sr)}t^{\frac{a_i(m_i+sr)^2}{4}}t^{\frac{(m_i+sr)m_{i+1}}{2}}\\
=&(-1)^{-(-1)^ia_i(m_i+sr)}t^{\frac{a_i(m_i+sr)^2}{4}}t^{\frac{(m_i+sr)m_{i+1}}{2}}\\
=&e^{\frac{\pi\sqrt{-1}}{r}a_i(m_i+sr)^2+2\pi\sqrt{-1}\big(\frac{m_{i+1}}{r}-\frac{(-1)^ia_i}{2}\big)(m_i+sr)}.
\end{split}
\end{equation*}
Then by this and (\ref{i-1}), we have
\begin{equation*}
\begin{split}
&S_i(m_0,m_{i+1})\\
=&2\tau^+_{i-1}\sum_{m_i=0}^{r-1}\sum_{s=0}^{2|A_{i-1}|-1}e^{\frac{\pi\sqrt{-1}}{r}a_i(m_i+sr)^2+2\pi\sqrt{-1}\big(\frac{m_{i+1}}{r}-\frac{(-1)^ia_i}{2}\big)(m_i+sr)-\frac{\pi\sqrt{-1}}{r}\frac{C_{i-1}}{A_{i-1}}\big(m_{i}+sr+\frac{K_{i-1}r}{2}+\frac{(-1)^{i}m_0}{C_{i-1}}\big)^2}\\
&-2\tau^-_{i-1}\sum_{m_i=0}^{r-1}\sum_{s=0}^{2|A_{i-1}|-1}e^{\frac{\pi\sqrt{-1}}{r}a_i(m_i+sr)^2+2\pi\sqrt{-1}\big(\frac{m_{i+1}}{r}-\frac{(-1)^ia_i}{2}\big)(m_i+sr)-\frac{\pi\sqrt{-1}}{r}\frac{C_{i-1}}{A_{i-1}}\big(m_{i}+sr+\frac{K_{i-1}r}{2}-\frac{(-1)^{i}m_0}{C_{i-1}}\big)^2}.\\
\end{split}
\end{equation*}
Here we observe that, as $s$ runs over $\{0, \dots,2|A_{i-1}|-1\}$ and $m_i$ runs over $\{0,\dots, r-1\},$ the quantity $m_i+sr$ runs over all integers in $\{0, \dots, 2|A_{i-1}|r-1\}.$ Thus, by letting $\lambda=m_i+sr,$ we can change the double sum above into the following single sum
\begin{equation*}
\begin{split}
&S_i(m_0,m_{i+1})\\
=&2\tau^+_{i-1}\sum_{\lambda=0}^{2|A_{i-1}|r-1}e^{\frac{\pi\sqrt{-1}}{r}a_i\lambda^2+2\pi\sqrt{-1}\big(\frac{m_{i+1}}{r}-\frac{(-1)^ia_i}{2}\big)\lambda-\frac{\pi\sqrt{-1}}{r}\frac{C_{i-1}}{A_{i-1}}\big(\lambda+\frac{K_{i-1}r}{2}+\frac{(-1)^{i}m_0}{C_{i-1}}\big)^2}\\
&\quad\quad\quad-2\tau^-_{i-1}\sum_{\lambda=0}^{2|A_{i-1}|r-1}e^{\frac{\pi\sqrt{-1}}{r}a_i\lambda^2+2\pi\sqrt{-1}\big(\frac{m_{i+1}}{r}-\frac{(-1)^ia_i}{2}\big)\lambda-\frac{\pi\sqrt{-1}}{r}\frac{C_{i-1}}{A_{i-1}}\big(\lambda+\frac{K_{i-1}r}{2}-\frac{(-1)^{i}m_0}{C_{i-1}}\big)^2}\\
=&2\tau^+_{i-1}\sum_{\lambda=0}^{2|A_{i-1}|r-1}e^{\frac{\pi\sqrt{-1}}{r}\big(a_i-\frac{C_{i-1}}{A_{i-1}}\big)\lambda^2+2\pi\sqrt{-1}\big(\frac{m_{i+1}}{r}-\frac{(-1)^ia_i}{2}-\frac{K_{i-1}C_{i-1}}{2A_{i-1}}-\frac{(-1)^im_0}{A_{i-1}r}\big)\lambda-\frac{\pi\sqrt{-1}}{r}\frac{C_{i-1}}{A_{i-1}}\big(\frac{K_{i-1}r}{2}+\frac{(-1)^im_0}{C_{i-1}}\big)^2}\\
&-2\tau^+_{i-1}\sum_{\lambda=0}^{2|A_{i-1}|r-1}e^{\frac{\pi\sqrt{-1}}{r}\big(a_i-\frac{C_{i-1}}{A_{i-1}}\big)\lambda^2+2\pi\sqrt{-1}\big(\frac{m_{i+1}}{r}-\frac{(-1)^ia_i}{2}-\frac{K_{i-1}C_{i-1}}{2A_{i-1}}+\frac{(-1)^im_0}{A_{i-1}r}\big)\lambda-\frac{\pi\sqrt{-1}}{r}\frac{C_{i-1}}{A_{i-1}}\big(\frac{K_{i-1}r}{2}-\frac{(-1)^im_0}{C_{i-1}}\big)^2}.
\end{split}
\end{equation*}
By Lemma \ref{cf}, we have $A_{i-1}=C_i,$ $a_i-\frac{C_{i-1}}{A_{i-1}}=\frac{A_i}{C_i},$ and by the definition of $K_i$ we have $(-1)^ia_i+\frac{K_{i-1}C_{i-1}}{C_i}=-K_i.$ As a consequence, 
\begin{equation*}
\begin{split}
&S_i(m_0,m_{i+1})\\
=&2\tau^+_{i-1} e^{-\frac{\pi\sqrt{-1}}{r}\frac{C_{i-1}}{A_{i-1}}\big(\frac{K_{i-1}r}{2}+\frac{(-1)^im_0}{C_{i-1}}\big)^2}\sum_{\lambda=0}^{2|C_i|r-1}e^{\frac{2A_i\lambda^2\pi\sqrt{-1}}{2C_ir}}e^{2\big(\frac{m_{i+1}}{r}+\frac{K_i}{2}+\frac{(-1)^{i+1}m_0}{C_ir}\big)\lambda\pi\sqrt{-1}}\\
&\quad\quad-2\tau^-_{i-1} e^{-\frac{\pi\sqrt{-1}}{r}\frac{C_{i-1}}{A_{i-1}}\big(\frac{K_{i-1}r}{2}-\frac{(-1)^im_0}{C_{i-1}}\big)^2}\sum_{\lambda=0}^{2|C_i|r-1}e^{\frac{2A_i\lambda^2\pi\sqrt{-1}}{2C_ir}}e^{2\big(\frac{m_{i+1}}{r}+\frac{K_i}{2}-\frac{(-1)^{i+1}m_0}{C_ir}\big)\lambda\pi\sqrt{-1}}.
\end{split}
\end{equation*}
Finally, letting
$$\tau^\pm_i=2e^{-\frac{\pi\sqrt{-1}}{r}\frac{C_{i-1}}{A_{i-1}}\big(\frac{K_{i-1}r}{2}\pm\frac{(-1)^im_0}{C_{i-1}}\big)^2}\Big(\frac{\sqrt{-1}C_ir}{A_i}\Big)^{\frac{1}{2}}\tau^\pm_{i-1}$$
and by Proposition \ref{R} with $m=2A_i,$ $n=2C_ir$ and $\psi=\frac{m_{i+1}}{r}+\frac{K_i}{2}\pm\frac{(-1)^im_0}{C_i},$ we have
\begin{equation*}
\begin{split}
&S_i(m_0,m_{i+1})\\
=&\tau^+_i\sum_{s=0}^{2|A_i|-1}e^{-\frac{2C_ir\big(s+\frac{m_{i+1}}{r}+\frac{K_i}{2}+\frac{(-1)^{i+1}m_0}{C_ir}\big)^2\pi\sqrt{-1}}{2A_i}}-\tau^-_i\sum_{s=0}^{2|A_i|-1}e^{-\frac{2C_ir\big(s+\frac{m_{i+1}}{r}+\frac{K_i}{2}-\frac{(-1)^{i+1}m_0}{C_ir}\big)^2\pi\sqrt{-1}}{2A_i}}\\
=&\tau^+_i\sum_{s=0}^{2|A_i|-1}e^{-\frac{\pi\sqrt{-1}}{r}\frac{C_i}{A_i}\big(m_{i+1}+sr+\frac{K_ir}{2}+\frac{(-1)^{i+1}m_0}{C_i}\big)^2}-\tau^-_i\sum_{s=0}^{2|A_i|-1}e^{-\frac{\pi\sqrt{-1}}{r}\frac{C_i}{A_i}\big(m_{i+1}+sr+\frac{K_ir}{2}-\frac{(-1)^{i+1}m_0}{C_i}\big)^2}.
\end{split}
\end{equation*}
By induction and that $A_{j-1}=C_j$ for each $j,$ 
\begin{equation*}
\begin{split}
\tau^\pm_i=&\bigg(\frac{\sqrt{-1}r}{A_1}\bigg)^{\frac{1}{2}}\prod_{j=2}^i\bigg(2e^{-\frac{\pi\sqrt{-1}}{r}\frac{C_{j-1}}{A_{j-1}}\big(\frac{K_{j-1}r}{2}\pm\frac{(-1)^jm_0}{C_{j-1}}\big)^2}\Big(\frac{\sqrt{-1}C_jr}{A_j}\Big)^{\frac{1}{2}}\bigg)\\
=&\frac{(-1)^{\frac{i}{4}}}{\sqrt A_i}2^{i-1}r^{\frac{i}{2}}e^{-\frac{\pi\sqrt{-1}}{r}m_0^2\sum_{j=1}^{i-1}\frac{1}{C_jC_{j+1}}-\frac{\pi\sqrt{-1}r}{4}\sum_{j=1}^{i-1}\frac{C_jK_j^2}{C_{j+1}}\mp\pi\sqrt{-1}m_0\sum_{j=1}^{i-1}\frac{(-1)^{j+1}K_j}{C_{j+1}}}.
\end{split}
\end{equation*}
\end{proof}

\section{Poisson summation formula}\label{psf}

The main result of this section is Proposition \ref{Poisson} where, using the Poisson Summation Formula,  we write the $r$-th Reshetikhin-Turaev invariant of $M$ as a sum of integrals. 

\begin{proposition}\label{bound} For $\epsilon>0$ and $s\in\{1,\dots, |q|-1\},$ we can choose a sufficiently small $\delta>0$ so that if one of 
$\frac{2\pi n}{r}+\frac{2\pi m}{r}$ and $\frac{2\pi n}{r}-\frac{2\pi m}{r}$ is not in $\big(\delta,\frac{\pi}{2}-\delta\big)\cup\big(\pi+\delta,\frac{3\pi}{2}-\delta \big),$
then
$$|g_r(s,m,n)|<O\Big(e^{\frac{r}{4\pi}\big(\frac{1}{2}\mathrm{Vol}(\mathrm S^3\setminus K_{4_1})+\epsilon\big)}\Big).$$
\end{proposition}

To prove Proposition \ref{bound}, we need the following estimate, which first appeared in \cite[Proposition 8.2]{GL} for $t=e^{\frac{2\pi\sqrt{-1}}{r}},$ and in \cite[Proposition 4.1]{DK} for $t=e^{\frac{4\pi\sqrt{-1}}{r}}.$

\begin{lemma}\label{est}
 For any integer $0<n<r$ and at $t=e^{\frac{4\pi\sqrt{-1}}{r}},$
 $$ \log\left|\{n\}!\right|=-\frac{r}{2\pi}\Lambda\left(\frac{2n\pi}{r}\right)+O\big(\log r \big ).$$
\end{lemma}

\begin{proof}[Proof of Proposition \ref{bound}] By (\ref{estim}), we have
$$|g_r(s,m,n)|=\Big|\sin\bigg(\frac{2\pi m}{r}\frac{1}{q}-J(s)\bigg)\Big|\Big|\frac{\{r-m-n-1\}!}{\{n-m\}!}\Big|,$$
and by Lemma \ref{est}, we have
$$\log |g_r(s,m,n)|=-\frac{r}{2\pi}\Lambda\Big(\frac{2\pi(r-m-n-1)}{r}\Big)+\frac{r}{2\pi}\Lambda\Big(\frac{2\pi(n-m)}{r}\Big)+O\big(\log r\big).$$
Choose $\delta>0$ so that 
$$\Lambda(\delta)<\frac{\epsilon}{4}.$$
Now  if one of 
$\frac{2\pi n}{r}+\frac{2\pi m}{r}$ and $\frac{2\pi n}{r}-\frac{2\pi m}{r}$ is not in $\big(\delta,\frac{\pi}{2}-\delta\big)\cup\big(\pi+\delta,\frac{3\pi}{2}-\delta \big),$
then
$$\log |g_r(s,m,n)|<\frac{r}{2\pi}\Big(\Lambda\Big(\frac{\pi}{6}\Big)+\frac{\epsilon}{2}\Big)=\frac{r}{4\pi}\Big(\frac{1}{2}\mathrm{Vol}(\mathrm S^3\setminus K_{4_1})+\epsilon\Big).$$
The last equality is true because by properties of the Lobachevsky function $\Lambda\big(\frac{\pi}{6}\big)=\frac{3}{2}\Lambda\big(\frac{\pi}{3}\big),$ and the volume of $\mathrm S^3\setminus K_{4_1}$ equals $6\Lambda(\frac{\pi}{3}).$
\end{proof}

For $\delta\geqslant 0,$ we let 
$$D_{\delta}=\Big\{(x,y)\in \mathbb R^2\ \Big|\ \delta < y+x <\frac{\pi}{2}-\delta, \delta<y-x<\frac{\pi}{2}-\delta\Big\},$$
$$D'_{\delta}=\Big\{(x,y)\in \mathbb R^2\ \Big|\ \delta< y+x <\frac{\pi}{2}-\delta, \pi+\delta< y-x < \frac{3\pi}{2}-\delta \Big\}$$
and
$$D''_{\delta}=\Big\{(x,y)\in \mathbb R^2\ \Big|\ \pi+\delta <  y+x<\frac{3\pi}{2}-\delta, \delta<y-x <\frac{\pi}{2}-\delta\Big\},$$
and let  $\mathcal D_\delta=D_\delta\cup D'_\delta\cup D''_\delta.$ If $\delta=0,$ we omit the subscript and write
$D=D_0,$ $D'=D'_0,$ $D''=D''_0$ and $\mathcal D=D\cup D'\cup D''.$

\begin{figure}[htbp]
\centering
\includegraphics[scale=0.7]{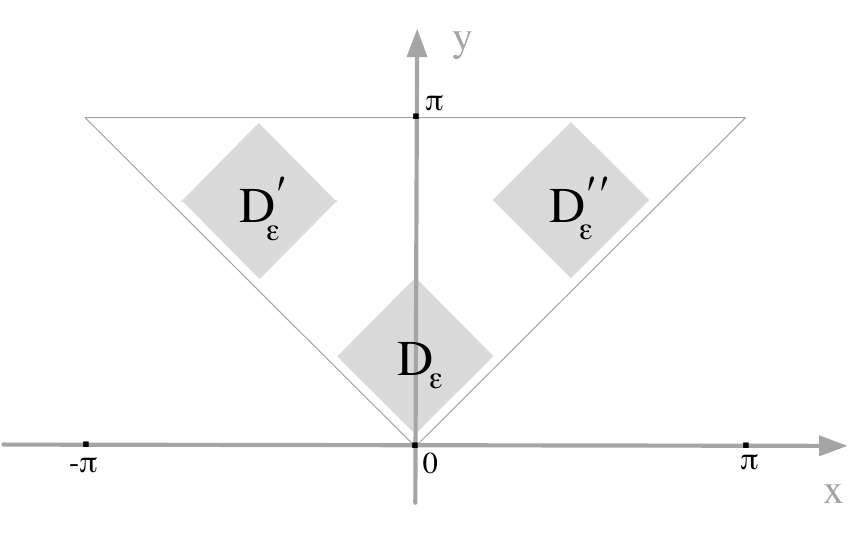}
\caption{Regions $D_\delta,$ $D_\delta'$ and $D_\delta''$}
\label{Delta} 
\end{figure}

For a sufficiently small $\delta>0,$ we consider a $C^{\infty}$-smooth bump function $\psi$ on $\mathbb R^2$ such that 
\begin{equation*}
\left \{\begin{array}{rl}
\psi(x,y)=1, & (x,y)\in\overline{\mathcal D_{\frac{\delta}{2}}}\\
0<\psi(x,y)<1, & (x,y)\in \mathcal D\setminus \overline{\mathcal D_{\frac{\delta}{2}}}\\
\psi(x,y)=0, & (x, y)\notin \mathcal D,\\
\end{array}\right.
\end{equation*}
and let 
$$f_r(s,m,n)=\psi\bigg(\frac{2\pi m}{r},\frac{2\pi n}{r}\bigg)g_r(s,m,n).$$
Then by Proposition \ref{bound}, we have
\begin{equation}\label{halfinteger}
\mathrm{RT}_r(M)=\kappa_r \sum_{s=0}^{|q|-1}\sum_{(m,n)\in\big(\mathbb Z+\frac{1}{2}\big)^2}f_r(s,m,n)+O\Big(e^{\frac{r}{4\pi}\big(\frac{1}{2}\mathrm{Vol}(\mathrm{S}^3\setminus K_{4_1})+\epsilon\big)}\Big).
\end{equation}

Since $f_r$ is $C^{\infty}$-smooth and equals zero outside of $\mathcal D,$ it is in the Schwartz space on $\mathbb R^2.$ Recall that by the Poisson Summation Formula (see e.g. \cite[Theorem 3.1]{SS}), for any function $f$ in the Schwartz space on $\mathbb R^{k},$
$$\sum_{(m_1,\dots,m_k)\in\mathbb Z^{k}}f(m_1,\dots,m_k)=\sum_{(n_1,\dots,n_{k})\in\mathbb Z^{k}}\hat f(n_1,\dots,n_{k}),$$
where  $\hat f(n_1,\dots,n_{k})$ is the $(n_1,\dots,n_{k})$-th Fourier coefficient of $f$ defined by
$$\hat f(n_1,\dots,n_{k})=\int_{\mathbb R^{k}}f(x_1,\dots, x_k)e^{\sum_{j=1}^k2\pi\sqrt{-1}n_jx_j}dx_1\dots dx_k.$$
As a consequence, we have

\begin{proposition}\label{Poisson}
$$\mathrm{RT}_r(M)=\kappa_r \sum_{s=0}^{|q|-1}\sum_{(m,n)\in\mathbb Z^2}\hat f_r(s,m,n)+O\Big(e^{\frac{r}{4\pi}\big(\frac{1}{2}\mathrm{Vol}(\mathrm{S}^3\setminus K_{4_1})+\epsilon\big)}\Big),$$
where
\begin{equation*}
\begin{split}
\hat f_r(s,m,n)=(-1)^{m+n}\Big(\frac{r}{2\pi}\Big)^{2}\int_{\mathcal D}&\psi(x,y)\sin\bigg(\frac{x}{q}-J(s)\pi\bigg)\epsilon(x,y)\\
&e^{-x\sqrt{-1}+\frac{r}{4\pi\sqrt{-1}}\Big(V_r(s,x,y)-4\pi m x-4\pi n y\Big)}dxdy.
\end{split}
\end{equation*}
\end{proposition}

\begin{proof} To apply the Poisson Summation Formula, we need to make the summation in (\ref{halfinteger}) over integers instead of half-integers. To do this, we let $m'=m+\frac{1}{2}$ and $n'=n+\frac{1}{2}.$ Then for each $s\in\{0,\dots,|q|-1\},$
$$\sum_{(m,n)\in(\mathbb Z+\frac{1}{2})^2}f_r(s,m,n)=\sum_{(m',n')\in\mathbb Z^2}f_r\Big(s, m'-\frac{1}{2},n'-\frac{1}{2}\Big).$$
Now by the Poisson Summation Formula, the right hand side equals
$$\sum_{(a,b)\in\mathbb Z^2}\int_{\mathbb R^2}f_r\Big(s,m'-\frac{1}{2},n'_k-\frac{1}{2}\Big)e^{2\pi\sqrt{-1}am'+2\pi\sqrt{-1}bn'}dm'dn'.$$
Using the change of variable $x=\frac{2\pi m}{r}=\frac{2\pi m'}{r}-\frac{\pi}{r}$ and $y=\frac{2\pi n}{r}=\frac{2\pi n'}{r}-\frac{\pi}{r},$ we get the result.
\end{proof}

In Section \ref{asymp}, we will show that among all the Fourier coefficients in this summation, the two $\hat f_r(s^+,m^+,0)$ and $\hat f_r(s^-,m^-,0)$ are the leading ones, where $s^\pm$ and $m^\pm$ are as defined in Lemma \ref{arith} (1). In the rest of the paper, we let 
$$V_r^+(x,y)=V_r(s^+,x,y)-4\pi m^+x$$
and let
$$V_r^-(x,y)=V_r(s^-,x,y)-4\pi m^-x.$$
Then by Lemma \ref{arith} (1), we have
\begin{equation*}
\begin{split}
V_r^+(x,y)=\frac{-px^2+2\pi x}{q}-2\pi x+4xy-\varphi_r\Big(\pi-y-x-\frac{\pi}{r}\Big)+\varphi_r\Big(y-x+\frac{\pi}{r}\Big)+K(s^+)\pi^2,
\end{split}
\end{equation*}
and
\begin{equation*}
\begin{split}
V_r^-&(x,y)=\frac{-px^2-2\pi x}{q}-2\pi x+4xy-\varphi_r\Big(\pi-y-x-\frac{\pi}{r}\Big)+\varphi_r\Big(y-x+\frac{\pi}{r}\Big)+K(s^-)\pi^2
\end{split}
\end{equation*}
on the region $D,$ and a similar formula on the regions $D'$ and $D''.$   As stated in Lemma \ref{LA} in Section \ref{asymp},  the functions $V_r^\pm$ are closely related to the functions 
\begin{equation*}
\begin{split}
V^+(x,y)=&\frac{-px^2+2\pi x}{q}-2\pi x+4xy-\mathrm{Li}_2\big(e^{-2\sqrt{-1}(y+x)}\big)+\mathrm{Li}_2\big(e^{2\sqrt{-1}(y-x)}\big)+K(s^+)\pi^2\\
\end{split}
\end{equation*}
and
\begin{equation*}
\begin{split}
V^-(x,y)=&\frac{-px^2-2\pi x}{q}-2\pi x+4xy-\mathrm{Li}_2\big(e^{-2\sqrt{-1}(y+x)}\big)+\mathrm{Li}_2\big(e^{2\sqrt{-1}(y-x)}\big)+K(s^-)\pi^2\\
\end{split}
\end{equation*}
whose critical values   will determine the exponential growth rate of the invariants. 

 We also notice that $V_r^\pm(x,y)$ and $V^\pm(x,y)$ define holomorphic functions on the regions $D_{\mathbb C, \delta},$ $D'_{\mathbb C, \delta}$ and $D''_{\mathbb C,  \delta}$ of $\mathbb C^2,$  where for $ \delta\geqslant 0,$
$$D_{\mathbb C,  \delta}=\Big\{(x,y)\in \mathbb C^2\ \Big|\  \delta<\mathrm{Re}(y)+\mathrm{Re}(x)<\frac{\pi}{2}- \delta,  \delta < \mathrm{Re}(y)-\mathrm{Re}(x)< \frac{\pi}{2}- \delta\Big\},$$
$$D'_{\mathbb C,  \delta}=\Big\{(x,y)\in \mathbb C^2\ \Big|\  \delta < \mathrm{Re}(y)+\mathrm{Re}(x)<\frac{\pi}{2}- \delta, \pi+ \delta< \mathrm{Re}(y)-\mathrm{Re}(x) < \frac{3\pi}{2}- \delta \Big\}$$
and
$$D''_{\mathbb C,  \delta}=\Big\{(x,y)\in \mathbb C^2\ \Big|\ \pi+ \delta < \mathrm{Re}(y)+\mathrm{Re}(x) < \frac{3\pi}{2}- \delta,  \delta<\mathrm{Re}(y)-\mathrm{Re}(x) <\frac{\pi}{2}- \delta\Big\}.$$
When $\ \delta=0,$ we denote the corresponding regions by $D_{\mathbb C},$ $D'_{\mathbb C}$ and $D''_{\mathbb C},$ and let $\mathcal D_\mathbb{C}=D_{\mathbb C}\cup D'_{\mathbb C}\cup D''_{\mathbb C}.$


\section{Geometry of the critical points}\label{geometry}

The goal of this section is to understand the geometric meaning of the critical points and the critical values of the functions $V^\pm$ defined in the previous section. The main result is Proposition \ref{Vol} which shows that the critical values of $V^\pm$ have real and imaginary parts the volume of $M$ and modulo $\pi^2\mathbb Z$ the Chern-Simons invariant of $M,$ and the determinants of the Hessian matrices of $V^\pm$ at the critical points give the adjoint twisted Reidemeister torsion of $M.$ The key observation is Lemma \ref{=} and Lemma \ref{==} that the system of critical point equations of $V^\pm$ is equivalent to the system of hyperbolic gluing equations (consisting of an edge equation and a $\frac{p}{q}$ Dehn-surgery equation) for a particular ideal triangulation of the figure-$8$ knot complement.

According to Thurston's notes \cite{T}, the complement of the figure-$8$ knot has an ideal triangulation as drawn in Figure \ref{figure-8}. We let $A$ and $B$ be the shape parameters of the two ideal tetrahedra 
and let $A'=\frac{1}{1-A},$ $A''=1-\frac{1}{A},$ $B'=\frac{1}{1-B}$ and $B''=1-\frac{1}{B}.$ 

\begin{figure}[htbp]
\centering
\includegraphics[scale=0.3]{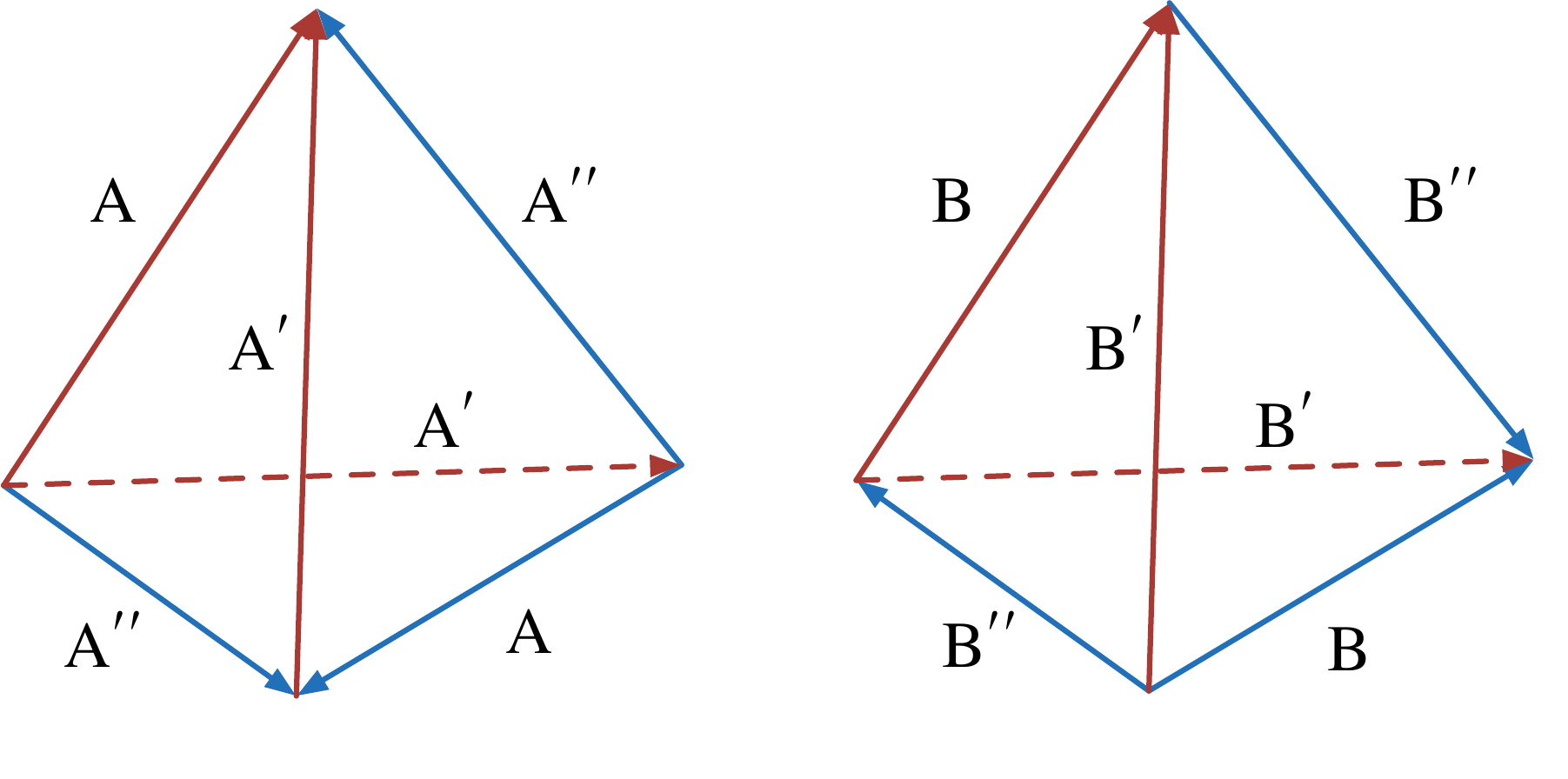}
\caption{Thurston's ideal triangulation of the figure-$8$ knot complement}
\label{figure-8}
\end{figure}

In Figure \ref{holonomy} is a fundamental domain of the boundary of the complement of a tubular neighborhood of the figure-$8$ knot.

\begin{figure}[htbp]
\centering
\includegraphics[scale=0.3]{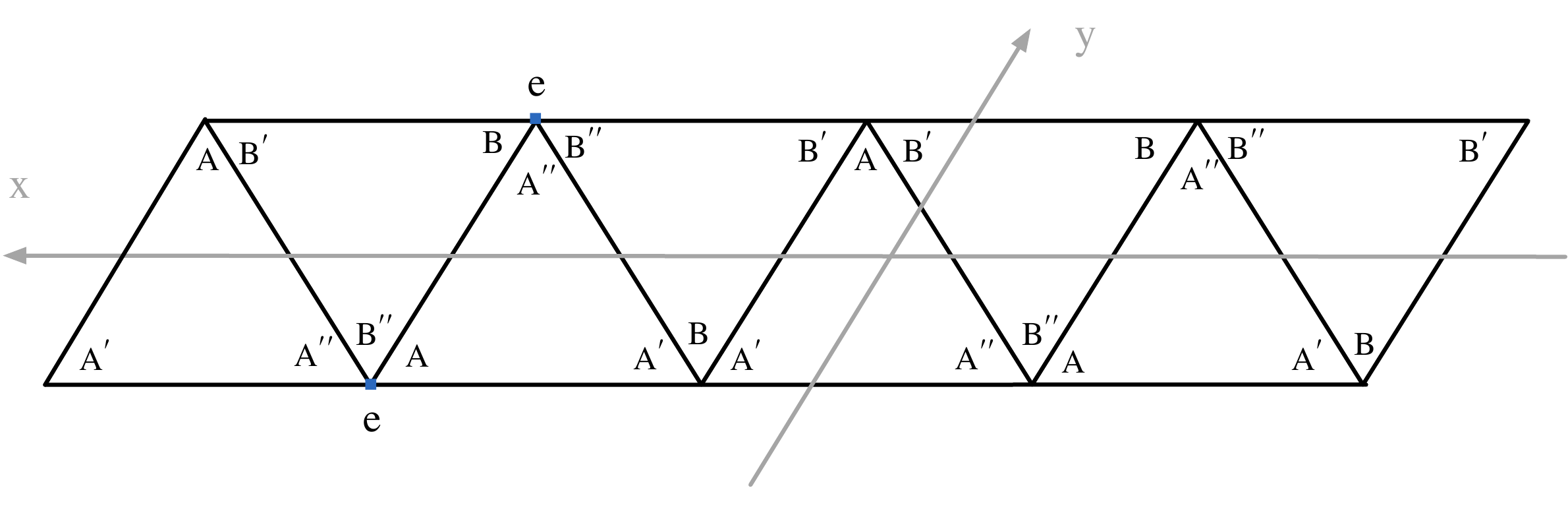}
\caption{Combinatorics around the boundary}
\label{holonomy}
\end{figure}

Recall that for $z\in \mathbb C\setminus (-\infty,0],$ the logarithmic function is defined by
$$\log z=\ln|z|+\sqrt{-1}\arg z$$ with $-\pi<\arg z<\pi.$

Then the holonomy around the edge $e$ is
$$\mathrm{H}(e)=\log A+2\log A''+\log B+2\log B'',$$
and the holonomies of the curves $x$ and $y$ are respectively
$$\mathrm{H}(x)=2\log B+2\log B''-2\log A-2\log A''$$ and
$$\mathrm{H}(y)=\log B'-\log A''.$$

By \cite{T}, we can choose the meridian $m=y$ and the longitude $l=x+2y.$ Hence
$$\mathrm{H}(m)=\log B'-\log A'',$$
and
\begin{equation*}
\begin{split}
\mathrm{H}(l)&=2\log B+2\log B'+2\log B''-2\log A-4\log A''\\
&=2\pi\sqrt{-1}-2\log A-4\log A''.
\end{split}
\end{equation*}
Then the system of hyperbolic gluing equations 
\begin{equation*}
\left \{\begin{array}{l}\mathrm{H}(e)=2\pi\sqrt{-1}\\
\\
p \mathrm{H}(m)+q\mathrm{H}(l)= 2\pi\sqrt{-1}
\end{array}\right.
\end{equation*} 
can be written as
\begin{equation}\label{HG}
\left \{\begin{array}{l}
\log A+2\log A''+\log B+2\log B''=2\pi\sqrt{-1}\\
\\
p(\log B'-\log A'')+q(2\pi\sqrt{-1}-2\log A-4\log A'')=2\pi\sqrt{-1}.
\end{array}\right.
\end{equation} 


Now for the critical point equations  of $V^\pm,$ by taking the partial derivatives, we have 
$$\frac{\partial V^\pm}{\partial x}=\frac{-2px\pm 2\pi}{q}+4y-2\pi-2\sqrt{-1}\log(1-e^{-2\sqrt{-1}(y+x)})+2\sqrt{-1}\log(1-e^{2\sqrt{-1}(y-x)})$$
and
$$\frac{\partial V^\pm}{\partial y}=4x-2\sqrt{-1}\log(1-e^{-2\sqrt{-1}(y+x)})-2\sqrt{-1}\log(1-e^{2\sqrt{-1}(y-x)}).$$
Hence the system of critical point equations of $V^\pm(x,y)$ is 
\begin{equation}\label{C}
\left \{\begin{array}{l}
4x-2\sqrt{-1}\log(1-e^{-2\sqrt{-1}(y+x)})-2\sqrt{-1}\log(1-e^{2\sqrt{-1}(y-x)})=0\\
\\
\frac{-2px\pm 2\pi}{q}+4y-2\pi-2\sqrt{-1}\log(1-e^{-2\sqrt{-1}(y+x)})+2\sqrt{-1}\log(1-e^{2\sqrt{-1}(y-x)})=0.
\end{array}\right.
\end{equation}

\begin{lemma}\label{=}
In $D_{\mathbb C},$ if we let $A=e^{2\sqrt{-1}(y+x)}$ and $B=e^{2\sqrt{-1}(y-x)},$  then
 the system of critical point equations (\ref{C}) of $V^+$ is equivalent to the system of hyperbolic glueing equations (\ref{HG}).
\end{lemma}
 
\begin{proof}  In $D_{\mathbb C},$ we have
\begin{equation*}
\left \{\begin{array}{l}
\log A= 2\sqrt{-1}(y+x),\\
 \log A'=\pi\sqrt{-1}-2\sqrt{-1}(y+x)-\log(1-e^{-2\sqrt{-1}(y+x)}),\\
 \log A''=\log(1-e^{-2\sqrt{-1}(y+x)}),\\
 \log B=2\sqrt{-1}(y-x),\\
\log B'=-\log(1-e^{2\sqrt{-1}(y-x)}),\\
 \log B''=\pi\sqrt{-1}-2\sqrt{-1}(y-x)+\log (1-e^{2\sqrt{-1}(y-x)}).
\end{array}\right.
\end{equation*}

For one direction, we assume that $(x,y)\in D_{\mathbb C}$ is a solution of (\ref{C}) with the ``$+$'' chosen. Then
\begin{equation*}
\begin{split}
\mathrm{H}(e)&=\log A+2\log A''+\log B+2\log B''\\
&=4x\sqrt{-1}+2\log (1-e^{-2\sqrt{-1}(y+x)})+2\log (1-e^{2\sqrt{-1}(y-x)})+2\pi\sqrt{-1}\\
&=2\pi\sqrt{-1},
\end{split}
\end{equation*}
where the last equality comes from the first equation of (\ref{C}). Hence the edge equation is satisfied. 

Next, we compute $\mathrm{H}(m)$ and $\mathrm{H}(l).$ We have
\begin{equation}\label{m}
\begin{split}
\mathrm{H}(m)&=\log B'-\log A''\\
&=-\log (1-e^{2\sqrt{-1}(y-x)})-\log (1-e^{-2\sqrt{-1}(y+x)})\\\
&=2x\sqrt{-1},
\end{split}
\end{equation}
where the last equality comes from the first equation of (\ref{C}); and 
\begin{equation}\label{l}
\begin{split}
\mathrm{H}(l)&=2\pi\sqrt{-1}-2\log A-4\log A''\\
&=2\pi\sqrt{-1}-2\log A+(4x\sqrt{-1}-2\log B')-2\log A''\\
&=-4y\sqrt{-1}+2\pi\sqrt{-1}-2\log(1-e^{-2\sqrt{-1}(y+x)})+2\log(1-e^{2\sqrt{-1}(y-x)}),
\end{split}
\end{equation}
where the second equality comes from (\ref{m}). Equations (\ref{m}), (\ref{l}) and the second equation of (\ref{C}) then imply that 
$$\frac{p\mathrm{H}(m)\sqrt{-1}+2\pi}{q}+\mathrm{H}(l)\sqrt{-1}=0,$$
which is equivalent to the $\frac{p}{q}$ Dehn-surgery equation 
$$p\mathrm{H}(m)+q\mathrm{H}(l)=2\pi\sqrt{-1}.$$

For the other direction, assume that $(A,B)$ is a  solution of (\ref{HG}). Then the edge equation implies the first equation of (\ref{C}); and (\ref{m}), (\ref{l}) and the Dehn-surgery equation  imply that the second equation of (\ref{C}).
\end{proof}

\begin{lemma}\label{==}
In $D_{\mathbb C},$ if we let $A=e^{2\sqrt{-1}(y-x)}$ and $B=e^{2\sqrt{-1}(y+x)},$ then the system of critical point equations (\ref{C}) of $V^-$ is equivalent to the system of hyperbolic glueing equations (\ref{HG}).
\end{lemma}

\begin{proof} This time we have \begin{equation*}
\left \{\begin{array}{l}
 \log A=2\sqrt{-1}(y-x),\\
\log A'=-\log(1-e^{2\sqrt{-1}(y-x)}),\\
 \log A''=\pi\sqrt{-1}-2\sqrt{-1}(y-x)+\log (1-e^{2\sqrt{-1}(y-x)}),\\
\log B= 2\sqrt{-1}(y+x),\\
 \log B'=\pi\sqrt{-1}-2\sqrt{-1}(y+x)-\log(1-e^{-2\sqrt{-1}(y+x)}),\\
 \log B''=\log(1-e^{-2\sqrt{-1}(y+x)}).\\
\end{array}\right.
\end{equation*} 
The rest of the proof is similar to that of Lemma \ref{=}. For one direction, we assume that $(x,y)\in D_{\mathbb C}$ is a solution of (\ref{C}) with the ``$-$'' chosen. Then
\begin{equation*}
\begin{split}
\mathrm{H}(e)=\log A+2\log A''+\log B+2\log B''=2\pi\sqrt{-1}.
\end{split}
\end{equation*}
Hence the edge equation is satisfied. 

For the computation of $\mathrm{H}(m)$ and $\mathrm{H}(l),$ we have
\begin{equation}\label{m'}
\begin{split}
\mathrm{H}(m)&=\log B'-\log A''\\
&=-4x\sqrt{-1}-\log (1-e^{2\sqrt{-1}(y-x)})-\log (1-e^{-2\sqrt{-1}(y+x)})\\\
&=-2x\sqrt{-1},
\end{split}
\end{equation}
where the last equality comes from the first equation of (\ref{C}); and
\begin{equation}\label{l'}
\begin{split}
\mathrm{H}(l)&=2\pi\sqrt{-1}-2\log A-4\log A''\\
&=2\pi\sqrt{-1}-2\log A+(-4x\sqrt{-1}-2\log B')-2\log A''\\
&=4y\sqrt{-1}-2\pi\sqrt{-1}+2\log(1-e^{-2\sqrt{-1}(y+x)})-2\log(1-e^{2\sqrt{-1}(y-x)}),
\end{split}
\end{equation}
where the second equality comes from (\ref{m'}). Equations (\ref{m'}), (\ref{l'}) and the second equation of (\ref{C}) then imply that 
$$\frac{-p\mathrm{H}(m)\sqrt{-1}-2\pi}{q}-\mathrm{H}(l)\sqrt{-1}=0,$$
which is equivalent to the $\frac{p}{q}$ Dehn-surgery equation 
$$p\mathrm{H}(m)+q\mathrm{H}(l)=2\pi\sqrt{-1}.$$

For the other direction, assume that $(A,B)$ is a  solution of (\ref{HG}). Then the edge equation implies the first equation of (\ref{C}); and (\ref{m'}), (\ref{l'}) and the Dehn-surgery equation  imply that the second equation of (\ref{C}).
\end{proof}

By Thurston's notes\,\cite{T}, for each relatively primed $(p,q)\neq (\pm 1, 1),$ $(\pm 2, \pm1),$ $(\pm 3,  \pm1)$ and $(\pm 4,  \pm1),$  there is a unique solution $A_0$ and $B_0$ of (\ref{HG}) with $\mathrm{Im}A_0>0$ and $\mathrm{Im}B_0>0.$ Then by Lemma \ref{=} and \ref{==}, we have 

\begin{corollary}\label{c} The point $$(x_0,y_0)=\Big(\frac{\log A_0-\log B_0}{4\sqrt{-1}},\frac{\log A_0+\log B_0}{4\sqrt{-1}}\Big)$$ is the unique critical point of $V^+$ in $D_{\mathbb C},$ and $(-x_0,y_0)$ is the unique critical point of $V^-$ in $D_{\mathbb C}.$
\end{corollary}

\begin{proposition}\label{Vol} 
\begin{enumerate}[(1)]
\item Let $\mathrm{Vol}(M)$ and $\mathrm{CS}(M)$ respectively be the hyperbolic volume and the Chern-Simons invariant of $M.$ Then
 $$V^+(x_0,y_0)-\bigg(K(s^+)+\frac{p'}{q}\bigg)\pi^2=V^-(-x_0,y_0)-\bigg(K(s^-)+\frac{p'}{q}\bigg)\pi^2$$
  for the integer $p'$ defined in (\ref{p'}), and
  $$V^+(x_0,y_0)\equiv V^-(-x_0,y_0)\equiv \sqrt{-1}\Big(\mathrm{Vol}(M)+\sqrt{-1}\mathrm{CS}(M)\Big)\quad(\mathrm{mod}\ \pi^2\mathbb Z).$$

\item Let $\mu=pm+ql,$ let $\rho$  be the holonomy representation of the hyperbolic structure on $M$ restricted to $S^3\setminus K_{4_1}$ and let $\mathrm{Tor}_\mu(S^3\setminus K_{4_1}; \mathrm{Ad}_\rho)$ be the Reidemeister torsion of $S^3\setminus K_{4_1}$ twisted by the adjoint action of $\rho$ with respect to the curve $\mu$\,\cite{P}. Then
$$\det(\mathrm{Hess}V^+)(x_0,y_0)=\det(\mathrm{Hess}V^-)(-x_0,y_0)=\frac{16}{q}\,\mathrm{Tor}_\mu(S^3\setminus K_{4_1}; \mathrm{Ad}_\rho)\neq 0.$$

\item Let $\rho$ be the holonomy representation of the hyperbolic structure on $M$ and let $\mathrm {Tor}(M;\mathrm{Ad}_\rho)$ be the Reidemeister torsion of $M$ twisted by the adjoint action of $\rho$\,\cite{P}. Then
$$\frac{\sin\Big(\frac{x_0}{q}-J(s^+)\pi\Big)}{\sqrt{-\det(\mathrm{Hess}V^+)(x_0,y_0)}}=-\frac{\sin\Big(\frac{-x_0}{q}-J(s^-)\pi\Big)}{\sqrt{-\det(\mathrm{Hess}V^-)(-x_0,y_0)}}=\pm\frac{\sqrt{-q}}{8\sqrt{\mathrm {Tor}(M;\mathrm{Ad}_\rho)}},$$
with the sign $\pm$ equals $(-1)^{J(s^+)-\frac{p'}{q}}.$
\end{enumerate}
 \end{proposition}

\begin{proof} For (1), we for $(x,y)\in D_{\mathbb C}$ have
 \begin{equation*}
 \begin{split}
-\mathrm{Li}_2(e^{-2\sqrt{-1}(y\pm x)})&=\mathrm{Li}_2(e^{2\sqrt{-1}(y\pm x)})+\frac{\pi^2}{6}+\frac{1}{2}\big(\log(-e^{2\sqrt{-1}(y\pm x)})\big)^2\\
&=\mathrm{Li}_2(e^{2\sqrt{-1}(y\pm x)})+\frac{\pi^2}{6}-2y^2-2x^2-\pi^2\mp 4xy+2\pi y\pm 2\pi x,\\
 \end{split}
 \end{equation*}
where the first equality comes from (\ref{Li2}), and the second equality comes from that $0<Re(y)\pm Re(x)<\frac{\pi}{2}$ and hence
$$\log(-e^{2\sqrt{-1}(y\pm x)})=2\sqrt{-1}(y\pm x)-\pi\sqrt{-1}.$$ 
From this we have for all $(x,y)$ with $0<Re(y)\pm Re(x)<\frac{\pi}{2},$
\begin{equation}\label{equal}
 \begin{split}
&V^+(x,y)-\bigg(K(s^+)+\frac{p'}{q}\bigg)\pi^2\\
=& \Big(-\frac{p}{q}-2\Big)x^2+\frac{2\pi x}{q}-2y^2+2\pi y-\frac{5\pi^2}{6}+\mathrm{Li}_2(e^{2\sqrt{-1}(y+x)})+\mathrm{Li}_2(e^{2\sqrt{-1}(y-x)})-\frac{p'\pi^2}{q},
 \end{split}
 \end{equation}
 and
 \begin{equation}\label{equall}
  \begin{split}
&V^-(-x,y)-\bigg(K(s^-)+\frac{p'}{q}\bigg)\pi^2\\
= & \Big(-\frac{p}{q}-2\Big)x^2+\frac{2\pi x}{q}-2y^2+2\pi y-\frac{5\pi^2}{6}+\mathrm{Li}_2(e^{2\sqrt{-1}(y+x)})+\mathrm{Li}_2(e^{2\sqrt{-1}(y-x)})-\frac{p'\pi^2}{q},
 \end{split}
 \end{equation}
which proves the first equility of (1).
 
For the second equality of (1), by (\ref{equal}) and (\ref{equall}) and Lemma \ref{arith} (3) that $K(s^\pm)+\frac{p'}{q}$ is an integer, we have
 $$V^+(x_0,y_0)\equiv V^-(-x_0,y_0)\quad(\text{mod }\pi^2\mathbb Z),$$
and it suffices to show that
 $$V^+(x_0,y_0)\equiv \sqrt{-1}(\mathrm{Vol}(M)+\sqrt{-1}\mathrm{CS}(M))\quad(\text{mod }\pi^2\mathbb Z).$$
 
To this end, we need the following result of Yoshida \cite[Theorem 2]{Y} that if the manifold $M$ is obtained by doing a hyperbolic Dehn-filling from the complement of a hyperbolic knot $K$ in $\mathrm S^3,$ $m$ and $l$ are respectively the meridian and longitude of the boundary of a tubular neighborhood of $K,$ $\gamma$ is isotopic to the core curve of the filled solid torus, and $\mathrm H(m),$ $\mathrm H(l)$ and $\mathrm H(\gamma)$ are respectively the holonomy of them, then 
$$\mathrm{Vol}(M)+\sqrt{-1}\mathrm{CS}(M)=\frac{\Phi(\mathrm{H}(m))}{\sqrt{-1}}-\frac{\mathrm H(m)\mathrm H(l)}{4\sqrt{-1}}+\frac{\pi\mathrm H(\gamma)}{2}\quad(\mathrm{mod}\ \sqrt{-1}\pi^2\mathbb Z),$$
where $\Phi$ is the function (see Neumann-Zagier\,\cite{NZ}) defined on the deformation space of hyperbolic structures on $S^3\setminus K$ parametrized by the holonomy of the meridian $u=\mathrm H(m),$ characterized by 
\begin{equation}\label{char}
\left \{\begin{array}{l}
\frac{\partial \Phi(u)}{\partial u}=\frac{\mathrm H(l)}{2},\\
\\
\Phi(0)=\sqrt{-1}\Big(\mathrm{Vol}(\mathrm S^3\setminus K)+\sqrt{-1}\mathrm{CS}(\mathrm S^3\setminus K)\Big).
\end{array}\right.
\end{equation} 
We will show that 
\begin{equation}\label{Phi}
\Phi(\mathrm H(m))=4x_0y_0-2\pi x_0-\mathrm{Li}_2(e^{-2(y_0+x_0)})+\mathrm{Li}_2(e^{2\sqrt{-1}(y_0-x_0)}),
\end{equation}
\begin{equation}\label{uv}
-\frac{\mathrm H(m)\mathrm H(l)}{4}=\frac{-px_0^2+\pi x_0}{q},
\end{equation}
and 
\begin{equation}\label{gamma}
\frac{\pi\sqrt{-1}}{2}\mathrm H(\gamma)=\frac{\pi x_0 }{q}-\frac{p'\pi^2}{q}
\end{equation}
so that
$$V^+(x_0,y_0)-\bigg(K(s^+)+\frac{p'}{q}\bigg)\pi^2= \Phi(\mathrm{H}(m))-\frac{ \mathrm H(m)\mathrm H(l)}{4}+\frac{\pi\sqrt{-1}}{2}\mathrm H(\gamma),$$
from which the result follows.

For (\ref{Phi}), we let 
$$U(x,y)=4xy-2\pi x-\mathrm{Li}_2(e^{-2\sqrt{-1}(y+x)})+\mathrm{Li}_2(e^{2\sqrt{-1}(y-x)}),$$ and define
$$\Psi(u)=U(x,y(x)),$$
where  $u=2x\sqrt{-1}$  and $y(x)$ is such that 
$$\frac{\partial V^+}{\partial y}\Big|_{(x,y(x))}=0.$$

Since 
$$\frac{\partial U}{\partial y}=\frac{\partial V^+}{\partial y}\quad\text{and}\quad\frac{\partial V^+}{\partial y}\Big|_{(x,y(x))}=0,$$
we have
\begin{equation*}
\begin{split}
\frac{\partial \Psi(u)}{\partial u}=\Big(\frac{\partial U}{\partial x}+\frac{\partial U}{\partial y}\Big|_{(x,y(x))}\frac{\partial y}{\partial x}\Big)\frac{\partial x}{\partial u}=\frac{\partial U}{\partial x}\frac{\partial x}{\partial u}=\frac{\mathrm H(l)}{2},
\end{split}
\end{equation*}
where the last equality comes from (\ref{l}). Also, a direct computation shows $y(0)=\frac{\pi}{6},$ and hence
$$\Psi(0)=U\Big(0,\frac{\pi}{6}\Big)=4\sqrt{-1}\Lambda\Big(\frac{\pi}{6}\Big)=\sqrt{-1}\Big(\mathrm{Vol}(\mathrm S^3\setminus K_{4_1})+\sqrt{-1}\mathrm{CS}(\mathrm S^3\setminus K_{4_1})\Big).$$  Therefore, $\Psi$ satisfies (\ref{char}), and hence $\Psi(u)=\Phi(u).$ 

Since $y(x_0)=y_0,$ and by (\ref{m}) $\mathrm H(m)=2x_0\sqrt{-1},$ we have
$$\Phi(\mathrm H(m))=\Psi(2x_0\sqrt{-1})=U(x_0,y_0),$$
which verifies (\ref{Phi}).

For (\ref{uv}), we have by (\ref{m}) that $\mathrm H(m)=2x_0\sqrt{-1}$ and 
$$\mathrm H(l)=\frac{2\pi\sqrt{-1}-p\mathrm H(m)}{q}=\frac{2\pi\sqrt{-1}-2px_0\sqrt{-1}}{q}.$$
Then $$\mathrm H(m)\mathrm H(l)=2x_0\sqrt{-1}\cdot\frac{2\pi\sqrt{-1}-2px_0\sqrt{-1}}{q}=-4\cdot \frac{-px_0^2+\pi x_0}{q},$$
from which (\ref{uv}) follows.

For (\ref{gamma}), since $pp'+qq'=1,$ we can choose $\gamma=-q'm+p'l$ so that $\mu\cdot\gamma=(pm+ql)\cdot(-q'm+p'l)=1.$ Then 
\begin{equation}\label{hgamma}
\mathrm H(\gamma)=-q'\mathrm H(m)+p'\mathrm H(l)=-q'\cdot 2x_0\sqrt{-1}+p'\cdot\frac{2\pi\sqrt{-1}-2px_0\sqrt{-1}}{q}=\frac{2}{\pi\sqrt{-1}}\cdot\frac{\pi x_0-p'\pi^2}{q},
\end{equation}
from which (\ref{gamma}) follows.

This completes the proof of (1).
\\

For (2), we have by (\ref{equal})
$$\mathrm{Hess}V^+(x_0,y_0)=\begin{bmatrix}
-\frac{2p}{q}-4-\frac{4e^{2\sqrt{-1}(y_0+x_0)}}{1-e^{2\sqrt{-1}(y_0+x_0)}}-\frac{4e^{2\sqrt{-1}(y_0-x_0)}}{1-e^{2\sqrt{-1}(y_0-x_0)}} & -\frac{4e^{2\sqrt{-1}(y_0+x_0)}}{1-e^{2\sqrt{-1}(y_0+x_0)}}+\frac{4e^{2\sqrt{-1}(y_0-x_0)}}{1-e^{2\sqrt{-1}(y_0-x_0)}} \\
-\frac{4e^{2\sqrt{-1}(y_0+x_0)}}{1-e^{2\sqrt{-1}(y_0+x_0)}}+\frac{4e^{2\sqrt{-1}(y_0-x_0)}}{1-e^{2\sqrt{-1}(y_0-x_0)}}  & -4-\frac{4e^{2\sqrt{-1}(y_0+x_0)}}{1-e^{2\sqrt{-1}(y_0+x_0)}}-\frac{4e^{2\sqrt{-1}(y_0-x_0)}}{1-e^{2\sqrt{-1}(y_0-x_0)}}
  \end{bmatrix},$$
 and by (\ref{equall})
  $$\mathrm{Hess}V^-(-x_0,y_0)=\begin{bmatrix}
-\frac{2p}{q}-4-\frac{4e^{2\sqrt{-1}(y_0-x_0)}}{1-e^{2\sqrt{-1}(y_0-x_0)}}-\frac{4e^{2\sqrt{-1}(y_0+x_0)}}{1-e^{2\sqrt{-1}(y_0+x_0)}} & \frac{4e^{2\sqrt{-1}(y_0+x_0)}}{1-e^{2\sqrt{-1}(y_0+x_0)}}-\frac{4e^{2\sqrt{-1}(y_0-x_0)}}{1-e^{2\sqrt{-1}(y_0-x_0)}} \\
\frac{4e^{2\sqrt{-1}(y_0+x_0)}}{1-e^{2\sqrt{-1}(y_0+x_0)}}-\frac{4e^{2\sqrt{-1}(y_0-x_0)}}{1-e^{2\sqrt{-1}(y_0-x_0)}}  & -4-\frac{4e^{2\sqrt{-1}(y_0-x_0)}}{1-e^{2\sqrt{-1}(y_0-x_0)}}-\frac{4e^{2\sqrt{-1}(y_0+x_0)}}{1-e^{2\sqrt{-1}(y_0+x_0)}} 
  \end{bmatrix}.$$
  Hence 
$$\det(\mathrm{Hess}V^+)(x_0,y_0)=\det(\mathrm{Hess}V^-)(-x_0,y_0),$$
and it suffices to prove that 
$$\det(\mathrm{Hess}V^+)(x_0,y_0)=\frac{16}{q}\mathrm{Tor}_\mu(S^3\setminus K_{4_1}; \mathrm{Ad}_\rho).$$

To this end, we for simplicity let $X_0=e^{2\sqrt{-1}x_0}$ and $Y_0=e^{2\sqrt{-1}y_0}.$ Then the first equation of (\ref{C}) implies
\begin{equation}\label{key0}
X_0+X_0^{-1}=Y_0+Y_0^{-1}+1.
\end{equation}
We have 
\begin{equation*}
\begin{split}
\frac{\partial^2V^+}{\partial x^2}(x_0,y_0)=&-\frac{2p}{q}-4-\frac{4X_0Y_0}{1-X_0Y_0}-\frac{4X_0^{-1}Y_0}{1-X_0^{-1}Y_0}\\
=& -\frac{2p}{q}-\frac{4(Y_0-Y_0^{-1})}{Y_0+Y_0^{-1}-X_0-X_0^{-1}}=-\frac{2p}{q}-4(Y_0-Y_0^{-1}),
\end{split}
\end{equation*}
where the last equality comes from (\ref{key0}). Similarly, by (\ref{key0}), we have
\begin{equation*}
\frac{\partial^2V^+}{\partial y^2}(x_0,y_0)=-4-\frac{4X_0Y_0}{1-X_0Y_0}-\frac{4X_0^{-1}Y_0}{1-X_0^{-1}Y_0}=-4(Y_0-Y_0^{-1})
\end{equation*}
and
\begin{equation*}
\frac{\partial^2V^+}{\partial x\partial y}(x_0,y_0)=-\frac{4X_0Y_0}{1-X_0Y_0}+\frac{4X_0^{-1}Y_0}{1-X_0^{-1}Y_0}=-4(X_0-X_0^{-1}).
\end{equation*}
Therefore,
\begin{equation}\label{Hessian}
\begin{split}
\det(\mathrm{Hess}V^+)(x_0,y_0)=& \bigg(-\frac{2p}{q}-4(Y_0-Y_0^{-1})\bigg)\Big(-4(Y_0-Y_0^{-1})\Big)-\Big(-4(X_0-X_0^{-1})\Big)^2\\
=&\frac{8p}{q}(Y_0-Y_0^{-1})+16\Big((Y_0-Y_0^{-1})^2-(X_0-X_0^{-1})^2\Big)\\
=&\frac{8p}{q}(Y_0-Y_0^{-1})+16\Big((Y_0+Y_0^{-1})^2-(X_0+X_0^{-1})^2\Big)\\
=&\frac{8p}{q}(Y_0-Y_0^{-1})+16(Y_0+Y_0^{-1}+X_0+X_0^{-1})(Y_0+Y_0^{-1}-X_0-X_0^{-1})\\
=&\frac{8p}{q}(Y_0-Y_0^{-1})-16(Y_0+Y_0^{-1}+X_0+X_0^{-1}),
\end{split}
\end{equation}
where the last equality comes from (\ref{key0}).

Next we will show that $ \frac{16}{q}\mathrm{Tor}_\mu(S^3\setminus K_{4_1}; \mathrm{Ad}_\rho)$ equals the same quantity. For this, by the change of curve formula \cite[Theorem 4.1 (ii)]{P}, we have
\begin{equation}\label{COC}
\mathrm{Tor}_\mu(S^3\setminus K_{4_1}; \mathrm{Ad}_\rho)= \frac{\partial \mathrm H(\mu)}{\partial \mathrm H(m)}\bigg|_\rho\cdot\mathrm{Tor}_m(S^3\setminus K_{4_1}; \mathrm{Ad}_\rho).
\end{equation}
To compute the  right hand side, we let $X=e^{2\sqrt{-1}x}$ and $Y=e^{2\sqrt{-1}y}.$ Then as the convention in Lemma \ref{=}, the shape parameters  $A=XY$ and $B=X^{-1}Y.$
 Recall that the deformation space of the hyperbolic structures on $M$ with this ideal triangulation consists of the shape parameters $A$ and $B$ that satisfy the edge equation $\mathrm H(e)=2\pi\sqrt{-1},$ which can be written as the first equation of (\ref{HG}). Then by Lemma \ref{=}, it is equivalent to the first equation of \ref{C}, which implies
\begin{equation}\label{key}
X+X^{-1}=Y+Y^{-1}+1.
\end{equation}
As a consequence,
\begin{equation}\label{yx}
\frac{\partial Y}{\partial X}=\frac{Y^2(X^2-1)}{X^2(Y^2-1)}.
\end{equation}
By (\ref{m}) and (\ref{l}), we have 
\begin{equation}\label{hm}
\mathrm H(m)=2x\sqrt{-1}=\log X
\end{equation}
and 
\begin{equation*}
\begin{split}
\mathrm H(l)=&2\pi\sqrt{-1}-2\log A-4\log A''\\
=&2\pi\sqrt{-1}-2\log(XY)-4\log(1-X^{-1}Y^{-1}).
\end{split}
\end{equation*}
Then
$$\frac{\partial \mathrm H(m)}{\partial X}=\frac{1}{X},$$
and by (\ref{yx})
\begin{equation*}
\begin{split}
\frac{\partial \mathrm H(l)}{\partial X}=&-\frac{2}{X}-\frac{4}{X(XY-1)}-\frac{\partial Y}{\partial X}\bigg(\frac{2}{Y}+\frac{4}{Y(XY-1)}\bigg)\\
=&-\frac{2}{X}-\frac{4}{X(XY-1)}-\frac{Y^2(X^2-1)}{X^2(Y^2-1)}\bigg(\frac{2}{Y}+\frac{4}{Y(XY-1)}\bigg)\\
=&\frac{-2(X+X^{-1}+Y+Y^{-1})}{X(Y-Y^{-1})};
\end{split}
\end{equation*}
and by the Chain Rule,
\begin{equation*}
\frac{\partial \mathrm H(l)}{\partial \mathrm H(m)}=\frac{\frac{\partial \mathrm H(l)}{\partial X}}{\frac{\partial \mathrm H(m)}{\partial X}}=\frac{-2(X+X^{-1}+Y+Y^{-1})}{Y-Y^{-1}}.
\end{equation*}
As a consequence, 
\begin{equation}\label{mum}
\frac{\partial \mathrm H(\mu)}{\partial \mathrm H(m)}=\frac{\partial\big(p\mathrm H(m)+q\mathrm H(l)\big)}{\partial \mathrm H(m)}=p-\frac{2q(X+X^{-1}+Y+Y^{-1})}{Y-Y^{-1}}.
\end{equation}
On the other hand, by \cite[Example 1 on p.113 ]{P}, for the holonomy representation $\rho_{\mathrm H(m)}$ of a (possibly incomplete) hyperbolic structure on $S^3\setminus K_{4_1}$ with $\mathrm H(m)$ the holonomy of the meridian $m,$
\begin{equation}\label{Torm}
\begin{split}
\mathrm{Tor}_m(S^3\setminus K_{4_1}; \mathrm{Ad}_{\rho_{\mathrm H(m)}})=&\pm\frac{\sqrt{(e^{\mathrm H(m)}+e^{-\mathrm H(m)}-3)(e^{\mathrm H(m)}+e^{-\mathrm H(m)}+1)}}{2}\\
=&\pm\frac{\sqrt{(X+X^{-1}-3)(X+X^{-1}+1)}}{2}\\
=&\pm \frac{Y-Y^{-1}}{2},
\end{split}
\end{equation}
where the second equation comes from (\ref{hm}) and the last equation comes from (\ref{key}). Notice that the Reidemeister torsion is a quantity defined up to $\pm,$ and here we choose the  ``$+$" sign. Then by (\ref{COC}), (\ref{mum}) and (\ref{Torm}), 
\begin{equation*}
\begin{split}
\frac{16}{q}\mathrm{Tor}_\mu(S^3\setminus K_{4_1}; \mathrm{Ad}_\rho)=&\frac{16}{q} \frac{\partial \mathrm H(\mu)}{\partial \mathrm H(m)}\bigg|_\rho\cdot\mathrm{Tor}_m(S^3\setminus K_{4_1}; \mathrm{Ad}_\rho)\\
=&\frac{16}{q} \frac{Y_0-Y_0^{-1}}{2} \bigg(p-\frac{2q(X_0+X_0^{-1}+Y_0+Y_0^{-1})}{Y_0-Y_0^{-1}}\bigg)\\
=& \frac{8p}{q}(Y_0-Y_0^{-1})-16(X_0+X_0^{-1}+Y_0+Y_0^{-1}).
\end{split}
\end{equation*}
Comparing with (\ref{Hessian}), we have
$$\det(\mathrm{Hess}V^+)(x_0,y_0)=\frac{16}{q}\mathrm{Tor}_\mu(S^3\setminus K_{4_1}; \mathrm{Ad}_\rho).$$

Finally, since the adjoint twisted Reidemeister torsion $\mathrm{Tor}_\mu(S^3\setminus K_{4_1}; \mathrm{Ad}_\rho)$ is a non-zero quantity, $\det(\mathrm{Hess}V^+)(x_0,y_0)=\det(\mathrm{Hess}V^-)(-x_0,y_0)\neq0$ and the Hessian matrices  $\mathrm{Hess}V^+(x_0,y_0)$ and $\mathrm{Hess}V^-(-x_0,y_0)$ are non-singular. 
\\

For (3), by Lemma \ref{arith} (2) that $J(s^+)\equiv -J(s^-)$ $(\text{mod }2\mathbb Z),$ we have
\begin{equation}\label{sin=}
\sin\bigg(\frac{x_0}{q}-J(s^+)\pi\bigg)=-\sin\bigg(\frac{-x_0}{q}-J(s^-)\pi\bigg).
\end{equation}
Together with (2), we have
$$\frac{\sin\Big(\frac{x_0}{q}-J(s^+)\pi\Big)}{\sqrt{-\det(\mathrm{Hess}V^+)(x_0,y_0)}}=-\frac{\sin\Big(\frac{-x_0}{q}-J(s^-)\pi\Big)}{\sqrt{-\det(\mathrm{Hess}V^-)(-x_0,y_0)}},$$
and it suffices to prove that
$$\frac{\sin\Big(\frac{x_0}{q}-J(s^+)\pi\Big)}{\sqrt{-\det(\mathrm{Hess}V^+)(x_0,y_0)}}=\pm\frac{\sqrt{-q}}{8\sqrt{\mathrm {Tor}(M;\mathrm{Ad}_\rho)}}.$$

To this end, we will use the surgery formula \cite[Theorem 4.1 (iii)]{P} that 
\begin{equation}\label{surgeryformula}
\mathrm{Tor}(M;\mathrm{ Ad}_\rho)=\frac{1}{4\sinh ^2 \frac{\mathrm H(\gamma)}{2}}\mathrm{Tor}_\mu(S^3\setminus K_{4_1};\mathrm {Ad}_\rho) .
\end{equation}
By (\ref{hgamma}), we have 
$$\frac{x_0}{q}-\frac{p'\pi}{q}=\frac{\sqrt{-1}}{2}\mathrm H(\gamma).$$
Then by Lemma \ref{arith} (2) that $J(s^+)\equiv \frac{p'}{q}$ $(\text{mod }\mathbb Z),$ we have
\begin{equation}\label{=sinh}
\begin{split}\sin\bigg(\frac{x_0}{q}-J(s^+)\pi\bigg)=(-1)^{J(s^+)-\frac{p'}{q}}\sin\bigg(\frac{x_0}{q}-\frac{p'\pi}{q}\bigg)=(-1)^{J(s^+)-\frac{p'}{q}}\sqrt{-1}\sinh\frac{\mathrm H(\gamma)}{2}.
\end{split}
\end{equation}
Finally, by the surgery formula (\ref{surgeryformula}) and (2), we have 
\begin{equation*}
\begin{split}
\frac{\sin\Big(\frac{x_0}{q}-J(s^+)\pi\Big)}{\sqrt{-\det(\mathrm{Hess}V^+)(x_0,y_0)}}=&(-1)^{J(s^+)-\frac{p'}{q}}\frac{\sqrt{-q}\sinh\frac{\mathrm H(\gamma)}{2}}{\sqrt{16\mathrm{Tor}_\mu(S^3\setminus K_{4_1}; \mathrm{Ad}_\rho)}}\\
=&(-1)^{J(s^+)-\frac{p'}{q}}\frac{\sqrt{-q}}{8\sqrt{\mathrm {Tor}(M;\mathrm{Ad}_\rho)}}.
\end{split}
\end{equation*}
 \end{proof}

The following Lemma \ref{convexity} and Lemma \ref{n0} are crucial in the estimate of the Fourier coefficients in Section \ref{asymp}.

\begin{lemma}\label{convexity}
 In $\mathcal D_{\mathbb C}=D_{\mathbb C}\cup D'_{\mathbb C}\cup D''_{\mathbb C},$  $\mathrm{Im}V^\pm(x,y)$ is strictly concave down in $\mathrm{Re}(x)$ and $\mathrm{Re}(y),$ and is strictly concave up in  $\mathrm{Im}(x)$ and $\mathrm{Im}(y).$
\end{lemma}

\begin{proof}
Using (\ref{equal}), taking the second derivatives of $\mathrm{Im}V^\pm$ with respect to $\mathrm{Re}(x)$ and $\mathrm{Re}(y),$ we have
\begin{equation*}
\begin{split}
\mathrm{Hess}(\mathrm{Im}V^\pm)&= \begin{bmatrix}
-\frac{4\mathrm{Im}e^{2\sqrt{-1}(y+x)}}{|1-e^{2\sqrt{-1}(y+x)}|^2}-\frac{4\mathrm{Im}e^{2\sqrt{-1}(y-x)}}{|1-e^{2\sqrt{-1}(y-x)}|^2} & -\frac{4\mathrm{Im}e^{2\sqrt{-1}(y+x)}}{|1-e^{2\sqrt{-1}(y+x)}|^2}+\frac{4\mathrm{Im}e^{2\sqrt{-1}(y-x)}}{|1-e^{2\sqrt{-1}(y-x)}|^2}\\
&\\
-\frac{4\mathrm{Im}e^{2\sqrt{-1}(y+x)}}{|1-e^{2\sqrt{-1}(y+x)}|^2}+\frac{4\mathrm{Im}e^{2\sqrt{-1}(y-x)}}{|1-e^{2\sqrt{-1}(y-x)}|^2} & -\frac{4\mathrm{Im}e^{2\sqrt{-1}(y+x)}}{|1-e^{2\sqrt{-1}(y+x)}|^2}-\frac{4\mathrm{Im}e^{2\sqrt{-1}(y-x)}}{|1-e^{2\sqrt{-1}(y-x)}|^2}
  \end{bmatrix}\\
&=-\begin{bmatrix}
2&-2\\
2 & 2 
  \end{bmatrix}
  \begin{bmatrix}
\frac{\mathrm{Im}e^{2\sqrt{-1}(y+x)}}{|1-e^{2\sqrt{-1}(y+x)}|^2}& 0\\
0 & \frac{\mathrm{Im}e^{2\sqrt{-1}(y-x)}}{|1-e^{2\sqrt{-1}(y-x)}|^2}
  \end{bmatrix}
  \begin{bmatrix}
2&2\\
-2 & 2 
  \end{bmatrix}.
  \end{split}
  \end{equation*}
Since in $\mathcal D_{\mathbb C},$ $\mathrm{Im}e^{2\sqrt{-1}(y+x)}>0$ and $\mathrm{Im}e^{2\sqrt{-1}(y-x)}>0,$ the diagonal matrix in the middle is positive definite, and hence $\mathrm{Hess}(\mathrm{Im}V^\pm)$ is negative definite. Therefore, $\mathrm{Im}V$ is concave down in $\mathrm{Re}(x)$ and $\mathrm{Re}(y).$ Since $\mathrm{Im}V^\pm$ is harmonic, it is concave up in $\mathrm{Im}(x)$ and $\mathrm{Im}(y).$   
\end{proof}

\begin{lemma}\label{n0} $\mathrm{Im}(x_0)\neq 0.$
\end{lemma}

\begin{proof} By (\ref{m}), the holonomy of the meridian $\mathrm{H}(m)=2x_0\sqrt{-1}.$ We prove by contradiction. Suppose $\mathrm{Im}(x_0)=0,$ then $\mathrm{H}(m)$ is purely imaginary. As a consequence, $\mathrm{H}(l)=\frac{2\pi\sqrt{-1} -p\mathrm{H}(m)}{q}$ is also purely imaginary.  This implies that the holonomy of the core curve of the filled solid torus $H(\gamma)=q'\mathrm H(m)-p'\mathrm H(l)$ is purely imaginary, ie., $\gamma$ has length zero, which is a contradiction.
\end{proof}


\section{Asymptotics of the Reshetikhin-Turaev invariants}\label{asymp}

The goal of this section is to prove Theorem \ref{main1} by estimating each of the Fourier coefficients $\hat f_r(s,m,n)$ in Proposition \ref{Poisson}. In Section \ref{estleading} we estimated the two leading ones $\hat f_r(s^+,m^+,0)$ and $\hat f_r(s^-,m^-,0),$ and in Sections \ref{estother} we estimate the others. Finally in Section \ref{proof} we show that $\hat f_r(s^+,m^+,0)+\hat f_r(s^-,m^-,0)$ has the desired asymptotic behavior and the sum of all the other Fourier coefficients are neglectable, which completes the proof. We will also prove Theorem \ref{main2} in the end of Section \ref{proof}.


\subsection{Estimates of the leading Fourier coefficients}\label{estleading}

The main result of this section is Proposition \ref{leading} where we estimate the integrals in the Fourier coefficients $\hat f_r(s^+,m^+,0)$ and $\hat f_r(s^-,m^-,0)$  on the region $D_{\mathbb C,\delta}$ as defined in Section \ref{psf}, which turn out to be the leading terms in the summation of Proposition \ref{Poisson}.

We need the following Proposition \ref{saddle}, which is a generalization of the standard Steepest Descent Theorem (see eg. \cite[Theorem 7.2.8]{MH}) and is also stated and proved in Ohtsuki\,\cite[Proposition 3.5, Remarks 3.3, 3.6]{O} in a slightly different form. For the readers' convenience, we include a proof in the Appendix.

\begin{proposition}\cite{O}\label{saddle}
Let $D$ be a region in $\mathbb C^n$ and let $f(z_1,\dots, z_n)$ and $g(z_1,\dots, z_n)$ be holomorphic functions on $D$ independent of $r$. Let $f_r(z_1,\dots,z_n)$ be a holomorphic function of the form
$$ f_r(z_1,\dots, z_n) = f(z_1,\dots, z_n) + \frac{\upsilon_r(z_1,\dots,z_n)}{r^2}.$$
Let $S$ be an embedded real $n$-dimensional closed disk in $D$ and let $(c_1,\dots, c_n)$ be a point on $S.$ If
\begin{enumerate}[(1)]
\item $(c_1, \dots, c_n)$ is a critical point of $f$ in $D,$
\item $\mathrm{Re}(f)(c_1,\dots,c_n) > \mathrm{Re}(f)(z_1,\dots,z_n)$ for all $(z_1,\dots,z_n) \in S\setminus \{(c_1,\dots,c_n)\},$
\item the domain $\{(z_1,\dots,z_n)\in D\ |\ \mathrm{Re}{f}(z_1,\dots,z_n)<\mathrm{Re}{f}(c_1,\dots,c_n)\}$ deformation retracts to $S\setminus\{(c_1,\dots,c_n)\},$
\item the Hessian matrix $\mathrm{Hess}(f)(c_1,\dots,c_n)$ of $f$ at $(c_1,\dots,c_n)$ is non-singular,
\item $g(c_1,\dots,c_n) \neq 0,$ and
\item $|\upsilon_r(z_1,\dots,z_n)|$ is bounded from above by a constant independent of $r$ in $D,$
\end{enumerate}
then
\begin{equation*}
\begin{split}
 \int_S g(z_1,\dots, z_n) &e^{rf_r(z_1,\dots,z_n)} dz_1\dots dz_n\\
 &= \Big(\frac{2\pi}{r}\Big)^{\frac{n}{2}}\frac{g(c_1,\dots,c_n)}{\sqrt{-\det\mathrm{Hess}(f)(c_1,\dots,c_n)}} e^{rf(c_1,\dots,c_n)} \Big( 1 + O \Big( \frac{1}{r} \Big) \Big).
 \end{split}
 \end{equation*}
\end{proposition}

To apply Proposition \ref{saddle}, we need the following Lemma \ref{LA}. Recall from Section \ref{psf} that for $\delta \geqslant 0,$ the region

$$D_{\mathbb C,  \delta}=\Big\{(x,y)\in \mathbb C^2\ \Big|\  \delta<\mathrm{Re}(y)+\mathrm{Re}(x)<\frac{\pi}{2}- \delta,  \delta < \mathrm{Re}(y)-\mathrm{Re}(x)< \frac{\pi}{2}- \delta\Big\},$$
and for $(x,y)$ in $D_{\mathbb C,  \delta}$ the functions
\begin{equation*}
\begin{split}
V_r^\pm(x,y)=&V_r(s^\pm,m^\pm,0)-4\pi m^\pm x\\
=&\frac{-px^2\pm2\pi x}{q}-2\pi x+4xy-\varphi_r\Big(\pi-y-x-\frac{\pi}{r}\Big)+\varphi_r\Big(y-x+\frac{\pi}{r}\Big)+K(s^\pm)\pi^2,
\end{split}
\end{equation*}
 and 
 $$V^\pm(x,y)=\frac{-px^2\pm2\pi x}{q}-2\pi x+4xy-\mathrm{Li}_2\big(e^{-2\sqrt{-1}(y+x)}\big)+\mathrm{Li}_2\big(e^{-2\sqrt{-1}(y-x)}\big)+K(s^\pm)\pi^2.$$ 
 
\begin{lemma}\label{LA}
In $\big\{(x,y)\in \overline{D_{\mathbb C,\delta}}\ \big|\ |\mathrm{Im}x| < L, |\mathrm{Im}x| < L\}$ for $\delta>0$ and  $L>0,$ 
$$V^\pm_r(x,y)=V^\pm(x,y)-\frac{2\pi\sqrt{-1}\big(\log\big(1-e^{-2\sqrt{-1}(y+x)}\big)+\log\big(1-e^{2\sqrt{-1}(y-x)}\big)\big)}{r}+\frac{\upsilon_r(x,y)}{r^2}$$
with $|\upsilon_r(x,y)|$ bounded from above by a constant independent of $r.$
\end{lemma}

\begin{proof}
Expanding in $\frac{1}{r},$ we have
$$\varphi_r\Big(\pi-x-y-\frac{\pi}{r}\Big)=\varphi_r(\pi-x-y)-\varphi'_r(\pi-x-y)\cdot\frac{\pi}{r}+O\Big(\frac{1}{r^2}\Big)$$
and
$$\varphi_r\Big(y-x+\frac{\pi}{r}\Big)=\varphi_r(y-x)+\varphi'_r(y-x)\cdot\frac{\pi}{r}+O\Big(\frac{1}{r^2}\Big).$$
They by Lemma \ref{converge} (1) we have
\begin{equation*}
\begin{split}
&-\varphi_r(\pi-x-y)+\varphi_r(y-x)\\
=&-\mathrm{Li}(e^{-2\sqrt{-1}(y+x)})+\mathrm{Li}(e^{2\sqrt{-1}(y-x)})+\bigg(-\frac{2\pi^2e^{-2\sqrt{-1}(y+x)}}{3(1-e^{-2\sqrt{-1}(y+x)})}+\frac{2\pi^2e^{2\sqrt{-1}(y-x)}}{3(1-e^{2\sqrt{-1}(y-x)})}\bigg)\frac{1}{r^2}+O\Big(\frac{1}{r^4}\Big),
\end{split}
\end{equation*}
and by Lemma \ref{converge} (2) we have
\begin{equation*}
\begin{split}
&\varphi'_r(\pi-x-y)\cdot\frac{\pi}{r}+\varphi'_r(y-x)\cdot\frac{\pi}{r}\\
=&-\frac{2\pi\sqrt{-1}}{r}\log\big(1-e^{-2\sqrt{-1}(y+x)}\big)-\frac{2\pi\sqrt{-1}}{r}\log\big(1-e^{2\sqrt{-1}(y-x)}\big)+O\Big(\frac{1}{r^3}\Big).
\end{split}
\end{equation*}
The results then follows with 
$$\upsilon_r(x,y)=-\frac{2\pi^2e^{-2\sqrt{-1}(y+x)}}{3(1-e^{-2\sqrt{-1}(y+x)})}+\frac{2\pi^2e^{2\sqrt{-1}(y-x)}}{3(1-e^{2\sqrt{-1}(y-x)})}+O\Big(\frac{1}{r}\Big).$$
\end{proof}

Let $(x_0, y_0)$ be the unique critical point of $V^+$ in $D_{\mathbb C},$ and by Corollary \ref{c} $(-x_0, y_0)$ is the unique critical point of $V^-$ in $D_{\mathbb C}.$ Let $\delta$ be as in Proposition \ref{bound}, and as drawn in Figure \ref{surface} let $S^+= S^+_{\text{top}}\cup S^+_{\text{side}}\cup(D_{\frac{\delta}{2}}\setminus D_{\delta})$ be the union of $D_{\frac{\delta}{2}}\setminus D_{\delta}$ with the two surfaces 
$$S^+_{\text{top}}=\big\{(x,y)\in D_{\mathbb C,\delta}\ \big|\ (\mathrm{Im}(x),\mathrm{Im}(y))=(\mathrm{Im}(x_0),\mathrm{Im}(y_0))\big\}$$ and 
$$S^+_{\text{side}}=\big\{ (\theta_1+\sqrt{-1}t\mathrm{Im}(x_0), \theta_2+\sqrt{-1}t \mathrm{Im}(y_0)) \ \big|\ (\theta_1,\theta_2)\in \partial D_{\delta}, t\in[0,1]\big\};$$
and let  $S^-= S^-_{\text{top}}\cup S^-_{\text{side}}\cup(D_{\frac{\delta}{2}}\setminus D_{\delta})$ be the union of $D_{\frac{\delta}{2}}\setminus D_{\delta}$ with the two surfaces 
$$S^-_{\text{top}}=\big\{(x,y)\in D_{\mathbb C,\delta}\ \big|\ (\mathrm{Im}(x),\mathrm{Im}(y))=(-\mathrm{Im}(x_0),\mathrm{Im}(y_0))\big\}$$ and 
$$S^-_{\text{side}}=\big\{ (\theta_1-\sqrt{-1}t\mathrm{Im}(x_0), \theta_2+\sqrt{-1}t\mathrm{Im}(y_0)) \ \big|\ (\theta_1,\theta_2)\in \partial D_{\delta}, t\in[0,1]\big\}.$$ 
Then $S^\pm$ are homotopic to $D_{\frac{\delta}{2}};$ and since $\delta$ is sufficiently small, $S^\pm$ respectively contain $(\pm x_0,y_0).$

\begin{figure}[htbp]
\centering
\includegraphics[scale=0.3]{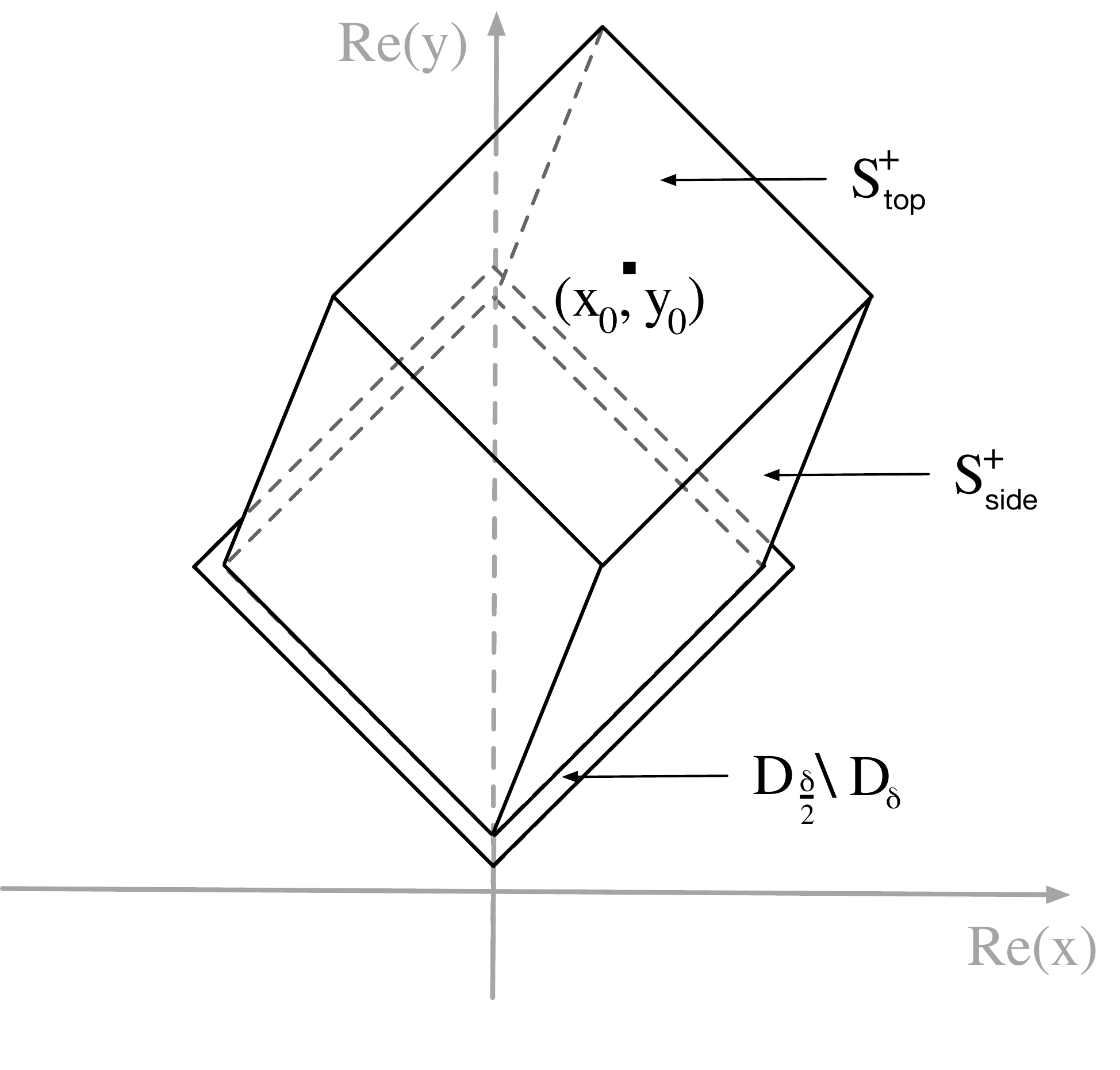}
\caption{The deformed surface $S^+$}
\label{surface} 
\end{figure}

\begin{proposition}\label{max}
$\mathrm{Im}V^+$ achieves the only absolute maximum at $(x_0,y_0)$ on $S^+,$  and $\mathrm{Im}V^-$ achieves the only absolute maximum at $(-x_0,y_0)$ on $S^-.$
\end{proposition}

\begin{proof} First, by Proposition \ref{bound} and Proposition \ref{small}, on $D_{\frac{\delta}{2}}\setminus D_{\delta},$ $\mathrm{Im}V(x,y)\leqslant \frac{1}{2}\mathrm{Vol}(\mathrm S^3\setminus K_{4_1})+\epsilon<\mathrm{Vol(M)}=\mathrm{Im}V^+( x_0,y_0).$

By Lemma \ref{convexity}, $\mathrm{Im}V^+$  is concave down on $S^{+}_{\text{top}}.$ Since $( x_0, y_0)$ is the critical point of $\mathrm{Im}V^{+},$ it is the only absolute maximum on $S^{+}_{\text{top}}.$

On the side $S^{+}_{\text{side}},$ for each $(\theta_1,\theta_2)\in \partial D_{\delta}$ consider the function $$g^{+}_{(\theta_1,\theta_2)}(t)\doteq \mathrm{Im}V^{+}(\theta_1+\sqrt{-1}t\mathrm{Im}(x_0), \theta_2+\sqrt{-1}t\mathrm{Im}(y_0))$$
on $[0,1].$
We show that $g^{+}_{(\theta_1,\theta_2)}(t)<\mathrm{Im}V^{+}(x_0,y_0)$ for each  $(\theta_1,\theta_2)\in \partial D_{\delta}$ and $t\in[0,1].$
Indeed, since $(\theta_1,\theta_2)\in\partial D_{\delta},$ $g^{+}_{(\theta_1,\theta_2)}(0)=\mathrm{Im}V^{+}(\theta_1,\theta_2)<\frac{1}{2}\mathrm{Vol}(\mathrm S^3\setminus K_{4_1})+\epsilon<\mathrm{Vol}(M)=\mathrm{Im}V^{+}(x_0,y_0);$ and since $(\theta_1+\sqrt{-1}\mathrm{Im}(x_0),\theta_2+\sqrt{-1}\mathrm{Im}(y_0))\in S^{+}_{\text{top}},$ by the previous step $g^{+}_{(\theta_1,\theta_2)}(1) = \mathrm{Im}V^{+}(\theta_1+\sqrt{-1}\mathrm{Im}(x_0),\theta_2+\sqrt{-1}\mathrm{Im}(y_0))<\mathrm{Im}V^{+}(x_0,y_0).$
Now by Lemma \ref{convexity}, $g^{+}_{(\theta_1,\theta_2)}$ is concave up, hence
$$g^{+}_{(\theta_1,\theta_2)}(t)\leqslant \max\Big\{g^{+}_{(\theta_1,\theta_2)}(0), g^{+}_{(\theta_1,\theta_2)}(1)\Big\}<\mathrm{Im}V^{+}( x_0,y_0).$$ 

Putting all these together, we have $\mathrm{Im}V^+(x,y)\leqslant \mathrm{Im}V^+( x_0,y_0)$ on $S^+.$
\\

The other case is similar. By Lemma \ref{convexity}, $\mathrm{Im}V^{-}$  is concave down on $S^{-}_{\text{top}}.$ Since $(- x_0, y_0)$ is the critical point of $\mathrm{Im}V^{-},$ it is the only absolute maximum on $S^{-}_{\text{top}}.$

On the side $S^{-}_{\text{side}},$ for each $(\theta_1,\theta_2)\in \partial D_{\delta}$ consider the function $$g^{-}_{(\theta_1,\theta_2)}(t)\doteq \mathrm{Im}V^{-}(\theta_1-\sqrt{-1}t\mathrm{Im}(x_0), \theta_2+\sqrt{-1}t\mathrm{Im}(y_0))$$
on $[0,1].$
We show that $g^{-}_{(\theta_1,\theta_2)}(t)<\mathrm{Im}V^{-}(- x_0,y_0)$ for each  $(\theta_1,\theta_2)\in \partial D_{\delta}$ and $t\in[0,1].$
Indeed, since $(\theta_1,\theta_2)\in\partial D_{\delta},$ $g^{-}_{(\theta_1,\theta_2)}(0)=\mathrm{Im}V^{-}(\theta_1,\theta_2)<\frac{1}{2}\mathrm{Vol}(\mathrm S^3\setminus K_{4_1})+\epsilon<\mathrm{Vol}(M)=\mathrm{Im}V^{-}(-x_0,y_0);$ and since $(\theta_1-\sqrt{-1}\mathrm{Im}(x_0),\theta_2+\sqrt{-1}\mathrm{Im}(y_0))\in S^{-}_{\text{top}},$ by the previous step $g^{-}_{(\theta_1,\theta_2)}(1) = \mathrm{Im}V^{-}(\theta_1-\sqrt{-1}\mathrm{Im}(x_0),\theta_2+\sqrt{-1}\mathrm{Im}(y_0))<\mathrm{Im}V^{\pm}(- x_0,y_0).$
Now by Lemma \ref{convexity}, $g^{-}_{(\theta_1,\theta_2)}$ is concave up, hence
$$g^{-}_{(\theta_1,\theta_2)}(t)\leqslant \max\Big\{g^{-}_{(\theta_1,\theta_2)}(0), g^{-}_{(\theta_1,\theta_2)}(1)\Big\}<\mathrm{Im}V^{-}(\pm x_0,y_0).$$ 

Putting all these together, we have $\mathrm{Im}V^-(x,y)\leqslant \mathrm{Im}V^-( x_0,y_0)$ on $S^-.$

\end{proof}

\begin{proposition}\label{leading}
 We have the following estimations.
\begin{enumerate}[(1)]
\item \begin{equation*}
\begin{split}
\int_{D_{\frac{\delta}{2}}}\psi(x,y)\sin\bigg(\frac{x}{q}-J(s^+)\pi\bigg)&\epsilon(x,y)e^{-x\sqrt{-1}+\frac{r}{4\pi\sqrt{-1}}\big(V_r(s^+,x,y)-4\pi m^+x\big)}dxdy\\
=&\frac{c^+_r}{\sqrt{\mathrm{Tor}(M;\mathrm{Ad}_\rho})}e^{\frac{r}{4\pi}\big(\mathrm{Vol}(M)+\sqrt{-1}\mathrm{CS}(M)\big)}\Big(1+O\Big(\frac{1}{r}\Big)\Big),
\end{split}
\end{equation*}
where $$c_r^+=-\frac{2\pi^2\sqrt{q}}{r}(-1)^{J(s^+)-\frac{p'}{q}}e^{-\frac{\pi\sqrt{-1}r}{4}\big(K(s^+)+\frac{p'}{q}\big)}.$$

\item  \begin{equation*}
\begin{split}
\int_{D_{\frac{\delta}{2}}}\psi(x,y)\sin\bigg(\frac{x}{q}-J(s^-)\pi\bigg)&\epsilon(x,y)e^{-x\sqrt{-1}+\frac{r}{4\pi\sqrt{-1}}\big(V_r(s^-,x,y)-4\pi m^-x\big)}dxdy\\
=&\frac{c^-_r}{\sqrt{\mathrm{Tor}(M;\mathrm{Ad}_\rho})}e^{\frac{r}{4\pi}\big(\mathrm{Vol}(M)+\sqrt{-1}\mathrm{CS}(M)\big)}\Big(1+O\Big(\frac{1}{r}\Big)\Big),
\end{split}
\end{equation*}
where $$c_r^-=\frac{2\pi^2\sqrt{q}}{r}(-1)^{J(s^+)-\frac{p'}{q}}e^{-\frac{\pi\sqrt{-1}r}{4}\big(K(s^-)+\frac{p'}{q}\big)}.$$
\end{enumerate}
\end{proposition}

\begin{proof}  By Lemma \ref{LA},
\begin{equation*}
\begin{split}
&e^{-x\sqrt{-1}+\frac{r}{4\pi\sqrt{-1}}\big(V_r(s^\pm,x,y)-4\pi m^\pm x\big)}\\
=&e^{-x\sqrt{-1}+\frac{r}{4\pi\sqrt{-1}}V_r^\pm(x,y)}\\
=&e^{-x\sqrt{-1}-\frac{\log\big(1-e^{-2\sqrt{-1}(y+x)}\big)}{2}-\frac{\log\big(1-e^{2\sqrt{-1}(y-x)}\big)}{2}+\frac{r}{4\pi\sqrt{-1}}\Big(V^\pm(x,y)+\frac{\upsilon_r(x,y)}{r^2}\Big)}.
\end{split}
\end{equation*}

 By analyticity, the integrals remain the same if we deform the domains from $D_{\frac{\delta}{2}}$ to $S^\pm.$ 

Now we apply Proposition \ref{saddle}, with the region 
$$D=D_L\doteq \big\{ (x,y)\in D_{\mathbb C,\frac{\delta}{2}}\ \big | |\mathrm{Im}x|<L, |\mathrm{Im}y|<L\big\}$$
 for a sufficiently large $L,$ the embedded disk $S=S^\pm,$ the functions $f(x,y)=\frac{1}{4\pi \sqrt{-1}}V^\pm(x,y),$ $f_r(x,y)=\frac{1}{4\pi \sqrt{-1}}\big(V^\pm(x,y)+\frac{\upsilon_r(x,y)}{r^2}\big),$
$$g(x,y)=\psi(x,y)\sin\bigg(\frac{x}{q}-J(s^\pm)\pi\bigg)\epsilon(x,y)e^{-x\sqrt{-1}-\frac{\log\big(1-e^{-2\sqrt{-1}(y+x)}\big)}{2}-\frac{\log\big(1-e^{2\sqrt{-1}(y-x)}\big)}{2}}$$
and  the point $(c_1,c_2)=(\pm x_0,y_0).$

Then by Corollary \ref{c}, $(\pm x_0,y_0)$ are respectively the critical points of $V^\pm,$ and hence of $f,$ and (1) of Proposition \ref{saddle} is satisfied.  By Proposition \ref{max}, $\mathrm{Re}(f)=\frac{1}{4\pi}\mathrm{Im}V^\pm$ achieves the only absolute maximum on $S^\pm$ at $(\pm x_0,y_0),$ and (2) of Proposition \ref{saddle} is satisfied.

To verify (3) of Proposition \ref{saddle}, we for each $(\theta_1,\theta_2)\in D_{\frac{\delta}{2}}$ let 
$$P_{(\theta_1,\theta_2)}=\big\{(x,y)\in D_L\ \big|\ \mathrm{Re}(x)=\theta_1, \mathrm {Re}(y)=\theta_2\big\}$$
and let 
$$D^\pm_{(\theta_1,\theta_2)}=\big\{(x,y)\in P_{(\theta_1,\theta_2)}\ \big|\ \mathrm{Im}V^\pm(x,y)<\mathrm{Im}V^\pm(\pm x_0,y_0)\big\}.$$
Then we claim that 
$$D^+_{(\mathrm{Re}(x_0),\mathrm{Re}(y_0))}=D^-_{(\mathrm{Re}(-x_0),\mathrm{Re}(y_0))}=\emptyset,$$
 $D^+_{(\theta_1,\theta_2)}$ is homeomorphic a disk for $(\theta_1,\theta_2)\neq (\mathrm{Re}(x_0),\mathrm{Re}(y_0)),$ and  $D^-_{(\theta_1,\theta_2)}$ is homeomorphic to a disk for $(\theta_1,\theta_2)\neq (\mathrm{Re}(-x_0),\mathrm{Re}(y_0)),$ from which we conclude that the domains  
$$\big\{(x,y)\in D_L\ \big|\ \mathrm{Im}V^\pm(x,y)< \mathrm{Im}V^\pm(\pm x_0,y_0)\big\}$$
respectively deformation retract to $S^\pm\setminus (\pm x_0,y_0)$ by shrinking each $D^\pm_{(\theta_1,\theta_2)}$ respectively to $\{(\theta_1\pm\sqrt{-1}\mathrm{Im}(x_0),\theta_2+\sqrt{-1}\mathrm{Im}(y_0)\},$ verifying (3) of Proposition \ref{saddle}.

To prove the claim, we by Lemma \ref{convexity} have that on $P_{ (\mathrm{Re}(\pm x_0),\mathrm{Re}(y_0))},$ $\mathrm{Im}V^\pm$ respectively achieve the absolute minimum at $(\pm x_0,y_0),$ hence $D^\pm_{ (\mathrm{Re}(\pm x_0),\mathrm{Re}(y_0))}=\emptyset.$ For $(\theta_1,\theta_2)\neq (\mathrm{Re}(\pm x_0),\mathrm{Re}(y_0)),$ we have
$$\min_{P_{(\theta_1,\theta_2)}}\mathrm{Im} V^\pm\leqslant \mathrm{Im} V^\pm (\theta_1\pm \sqrt{-1}\mathrm {Im}(x_0), \theta_2+\sqrt{-1}\mathrm{Im}(y_0))< \mathrm{Im}V^\pm (\pm x_0,y_0),$$
where the last inequality comes from that both $(\theta_1\pm \sqrt{-1}\mathrm {Im}(x_0), \theta_2+\sqrt{-1}\mathrm{Im}(y_0))$ and $(\pm x_0,y_0)$ are on $S^\pm_{\text{top}}$ and $\mathrm{Im}V^\pm$ achieve the only  maximum at $(\pm x_0,y_0).$ Then by Lemma \ref{convexity} that $\mathrm{Im}V^\pm$ is concave up on $P_ {(\theta_1,\theta_2)},$ $D_{(\theta_1,\theta_2)}$ is a convex subset of $P_ {(\theta_1,\theta_2)},$ which is homeomorphic to a disk. This proves the claim, and verifies (3) of Proposition \ref{saddle}.

By Proposition \ref{Vol} (2), $\det(\mathrm{Hess}f)(c_1,c_2)=-\frac{1}{16\pi^2}\det(\mathrm{Hess}V^\pm)(\pm x_0,y_0)\neq 0,$  and (4) of Proposition \ref{saddle} is satisfied. At the critical points $(\pm x_0,y_0),$ by the first equation of the system of critical equations (\ref{C}), we have
$$\mp x_0\sqrt{-1}-\frac{\log\big(1-e^{-2\sqrt{-1}(y_0\pm x_0)}\big)}{2}-\frac{\log\big(1-e^{2\sqrt{-1}(y_0\mp x_0)}\big)}{2}=0.$$ Together with (\ref{=sinh}) and (\ref{sin=}), we have
$$g(c_1,c_2)=2\sin\bigg(\frac{\pm x_0}{q}-J(s^\pm)\pi\bigg)=\pm(-1)^{J(s^+)-\frac{p'}{q}}2\sqrt{-1}\sinh\frac{\mathrm H(\gamma)}{2}\neq 0,
$$
where the inequality comes from that $\mathrm{Re}(\mathrm H(\gamma))$ is the length of the core curve $\gamma$ of the filled solid tori, which is non-zero. Hence (5) of Proposition \ref{saddle} is satisfied. (6) of Proposition \ref{saddle} follows from Lemma \ref{LA}. Therefore, all the conditions of Proposition \ref{saddle} are satisfied.

By Proposition \ref{Vol} (1), the critical values $$f(c_1,c_2)=\frac{V^\pm(\pm x_0,y_0)}{4\pi\sqrt{-1}}=\frac{1}{4\pi\sqrt{-1}}\bigg(\sqrt{-1}\big(\mathrm{Vol}(M)+\sqrt{-1}\mathrm{CS}(M)\big)+\Big(K(s^\pm)+\frac{p'}{q}\Big)\pi^2\bigg),$$ 
and by Proposition \ref{Vol} (3),
$$\frac{g(c_1,c_2)}{\sqrt{-\det\mathrm{Hess}(f)(c_1,c_2)}}=\frac{2\sin\Big(\frac{\pm x_0}{q}-J(s^\pm)\pi\Big)}{\sqrt{\frac{1}{16\pi^2}\det\mathrm{Hess}(V^\pm)(\pm x_0,y_0)}}=\mp\frac{(-1)^{J(s^+)-\frac{p'}{q}}\pi\sqrt{q}}{\sqrt{\mathrm{Tor}(M;\mathrm{Ad}_\rho})},$$
from which the result follows.
\end{proof}


\subsection{Estimate of other Fourier coefficients}\label{estother}

For $s\in\{0,\dots,|q|-1\}$ and $(x,y)\in \mathcal D_\mathbb{C}=D_{\mathbb C}\cup D'_{\mathbb C}\cup D''_{\mathbb C},$
 let
$$V(s,x,y)=\frac{-px^2}{q}+I(s)\frac{2\pi x}{q}+4xy-\mathrm{Li}_2\big(e^{-2\sqrt{-1}(y+x)}\big)+\mathrm{Li}_2\big(e^{-2\sqrt{-1}(y-x)}\big)+K(s)\pi^2.$$
Then similar to Lemma \ref{LA}, we have 
\begin{lemma}\label{LA'}
In $\big\{(x,y)\in \overline{D_{\mathbb C,\delta}}\cup \overline{D'_{\mathbb C,\delta}}\cup \overline{D''_{\mathbb C,\delta}}\ \big|\ |\mathrm{Im}x| < L, |\mathrm{Im}y| < L\}$ for $\delta>0$ and  $L>0,$ 
$$V_r(s,x,y)=V(s,x,y)-\frac{2\pi\sqrt{-1}\big(\log\big(1-e^{-2\sqrt{-1}(y+x)}\big)+\log\big(1-e^{2\sqrt{-1}(y-x)}\big)\big)}{r}+\frac{\upsilon_r(x,y)}{r^2}$$
with $|\upsilon_r(x,y)|$ bounded from above by a constant independent of $r.$
\end{lemma}


\subsubsection{Estimate on $D_{\frac{\delta}{2}}$}

\begin{proposition}\label{k1} There is an $\epsilon>0$ such that for each triple $(s,m,0)$ with $(s,m)\neq (s^+,m^+)$ and  $(s,m)\neq (s^-,m^-),$
\item \begin{equation*}
\begin{split}
\bigg|\int_{D_{\frac{\delta}{2}}}\psi(x,y)\sin\bigg(\frac{x}{q}-J(s)\pi\bigg)&\epsilon(x,y)e^{-x\sqrt{-1}+\frac{r}{4\pi\sqrt{-1}}\big(V_r(s,x,y)-4\pi mx\big)}dxdy\bigg|\leqslant O\Big(e^{\frac{r}{4\pi}\big(\mathrm{Vol}(M)-\epsilon\big)}\Big).
\end{split}
\end{equation*}
\end{proposition}

\begin{proof}  We recall the surface
$S^\pm= S^\pm_{\text{top}}\cup S^\pm_{\text{side}}\cup(D_{\frac{\delta}{2}}\setminus D_{\delta})$ from Section \ref{estleading}, where
$$S^\pm_{\text{top}}=\big\{(x,y)\in D_{\mathbb C,\delta}\ \big |\ (\mathrm{Im}(x),\mathrm{Im}(y))=(\pm\mathrm{Im}(x_0),\mathrm{Im}(y_0))\big\}$$ and 
$$S^\pm_{\text{side}}=\big\{ (\theta_1\pm\sqrt{-1}t\mathrm{Im}(x_0), \theta_2+\sqrt{-1}t \mathrm{Im}(y_0)) \ \big|\ (\theta_1,\theta_2)\in \partial D_{\delta}, t\in[0,1]\big\},$$
with $(\pm x_0,y_0)$  the critical points of $V^\pm$ given in Corollary \ref{c}.
We will prove that for some $\epsilon>0,$
$$\mathrm{Im}\big(V(s,x,y)-4\pi m x\big)<\mathrm{Vol}(M)-\epsilon$$
 either on $S^+$ or on $S^-.$ 
Then the result follows form Lemma \ref{LA'} and the analyticity of $V_r$ that the domain of integral $D_{\frac{\delta}{2}}$ could be deformed to $S^\pm.$

By Lemma \ref{n0}, $\mathrm{Im}(x_0)\neq 0.$  Without loss of generality we may assume that $\mathrm{Im}(x_0)>0,$ since otherwise we can consider $\mathrm{Im}(-x_0).$

Now for $s\in\{0,\dots, |q|-1\}$ and $m\in\mathbb Z,$  let
$$k^+(s,m)=-\frac{I(s)-1}{q}-1+2m$$
and 
$$k^-(s,m)=-\frac{I(s)+1}{q}-1+2m.$$
Then by a direct computation, 
\begin{equation*}
\begin{split}
V(s,x,y)-4\pi mx&=V^+(x,y)-2\pi k^+(s,m)x+\big(K(s)-K(s^+)\big)\pi^2\\
&=V^-(x,y)-2\pi k^-(s,m)x+\big(K(s)-K(s^-)\big)\pi^2.
\end{split}
\end{equation*}

By Lemma \ref{arith} (1), $(s,m)=(s^+,m^+)$ is the only pair such that $k^+(s,m)=0,$ and $(s,m)=(s^-,m^-)$ is the only pair such that $k^-(s,m)=0.$ By Lemma \ref{arith} (1), elements of 
$$\{k^\pm(s,m)\ |\ s\in\{0,\dots, |q|-1\}\text{ and }m\in \mathbb Z\}$$
 differ by a multiple of $\frac{2}{q};$ and  for each $(s,m),$ $k^+(s,m)-k^-(s,m)=\frac{2}{q}.$ This implies that for each pair $(s,m)$ other than $(s^\pm,m^\pm),$ $k^+(s,m)$ and $k^-(s,m)$ are either both strictly positive, or both strictly negative.

By Proposition \ref{max},  for any $(x,y)\in S^\pm_{\text{top}}$ we respectively have 
$$\mathrm{Im}V^\pm(x,y)\leqslant\mathrm{Im}V^\pm(\pm x_0,y_0)=\mathrm{Vol}(M).$$
 Then for $(s,m)$ with $k^\pm(s,m)>0,$ we have on $S^+_{\text{top}}$ that 
\begin{equation*}
\mathrm{Im}\Big(V(s,x,y)-4\pi m x\Big)=\mathrm{Im}V^+(x,y)-2\pi k^+(s,m)\mathrm{Im}(x_0)<\mathrm{Vol}(M)-\epsilon
\end{equation*}
for some $\epsilon>0;$ and  for $(s,m)$ with $k^\pm(s,m)<0,$ we have on $S^-_{\text{top}}$ that 
\begin{equation*}
\mathrm{Im}\Big(V(s,x,y)-4\pi m x\Big)=\mathrm{Im}V^-(x,y)+2\pi k^-(s,m)\mathrm{Im}(x_0)<\mathrm{Vol}(M)-\epsilon
\end{equation*}
for some $\epsilon>0.$

By Proposition \ref{bound} and Proposition \ref{small}, we have on $D_{\frac{\delta}{2}}\setminus D_{\delta}$ that
 $$\mathrm{Im}\big(V(x,y)-4\pi mx\big)=\mathrm{Im}V(x,y)\leqslant \frac{1}{2}\mathrm{Vol}(\mathrm S^3\setminus K_{4_1})+\epsilon<\mathrm{Vol(M)}-\epsilon.$$

On $S^\pm_{\text{side}},$ we notice  that $V(s,x,y)-4\pi mx$ differs from $V^\pm$ by a linear function. Hence by Lemma \ref{convexity}, for each $(\theta_1, \theta_2)\in\partial D_{\delta}$ the function 
$$g^\pm_{(\theta_1,\theta_2)}(t)\doteq \mathrm{Im}\Big(V\big(s,\theta_1\pm\sqrt{-1}t\mathrm{Im}(x_0),\theta_2+\sqrt{-1}t\mathrm{Im}(y_0)\big)-4\pi m \big(\theta_1\pm\sqrt{-1}t\mathrm{Im}(x_0)\big)\Big)$$
is concave up on $[0,1].$ Therefore, for $(s,m)$ with $k^\pm(s,m)>0,$ we have  on $S^+_{\text{side}}$ that
 $$\mathrm{Im}\Big(V(s,x,y)-4\pi m x\Big)=g^+_{(\theta_1,\theta_2)}(t)\leqslant \max\Big\{g^+_{(\theta_1,\theta_2)}(0), g^+_{(\theta_1,\theta_2)}(1)\Big\}<\mathrm{Vol}(M)-\epsilon;$$
  and  for $(s,m)$ with $k^\pm(s,m)<0,$ we have  on $S^-_{\text{side}}$ that  
  $$\mathrm{Im}\Big(V(s,x,y)-4\pi m x\Big)=g^-_{(\theta_1,\theta_2)}(t)\leqslant \max\Big\{g^-_{(\theta_1,\theta_2)}(0), g^-_{(\theta_1,\theta_2)}(1)\Big\}<\mathrm{Vol}(M)-\epsilon.$$
  
Putting all these together, we have for $(s,m)$ with $k^\pm(s,m)>0,$ $\mathrm{Im}\big(V(s,x,y)-4\pi m x\big)<\mathrm{Vol}(M)-\epsilon$ on $S^+,$ and for $(s,m)$ with $k^\pm(s,m)<0,$ $\mathrm{Im}\big(V(s,x,y)-4\pi m x\big)<\mathrm{Vol}(M)-\epsilon$ on $S^-.$ 
\end{proof}

\begin{proposition}\label{k2}  There is an $\epsilon>0$ such that for each triple $(s,m,n)$ with $n\neq 0,$ 
\begin{equation*}
\begin{split}
\bigg|\int_{D_{\frac{\delta}{2}}}\psi(x,y)\sin\bigg(\frac{x}{q}-J(s)\pi\bigg)&\epsilon(x,y)e^{-x\sqrt{-1}+\frac{r}{4\pi\sqrt{-1}}\big(V_r(s,x,y)-4\pi mx-4\pi ny\big)}dxdy\bigg|\leqslant O\Big(e^{\frac{r}{4\pi}\big(\mathrm{Vol}(M)-\epsilon\big)}\Big).
\end{split}
\end{equation*}
\end{proposition}

\begin{proof} For $L>0,$ we let 
$$S^\pm_L=S^\pm_{L,\text{top}}\cup S^\pm_{L,\text{side}}\cup(D_{\frac{\delta}{2}}\setminus D_{\delta}),$$
where
$$S^\pm_{L, \text{top}}=\big\{(x,y)\in D_{\mathbb C,\delta}\ \big|\ \mathrm {Im} (x)=0, \mathrm {Im} (y)=\pm L \big\}$$
and
$$S^\pm_{L, \text{side}}=\big\{(\theta_1,\theta_2\pm\sqrt{-1}l)\ \big|\ (\theta_1,\theta_2)\in \partial D_\delta, l\in [0,L] \big\}.$$
Then $S^\pm_L$ are  homotopic to  $D_{\frac{\delta}{2}}.$

 We want to show a stronger statement  that: If $L$ is sufficiently large, then there is an $\epsilon>0$ such that for each triple $(s,m,n)$ with $n\neq 0,$  $$\mathrm{Im}\big(V(s,x,y)-4\pi m x-4\pi n y\big)<\mathrm{Vol}(M)-\epsilon$$
either on $S^+_L$ or on $S^-_L.$

Then the result follows form Lemma \ref{LA'} and the analyticity of $V_r.$

To this end, first on $D_{\frac{\delta}{2}}\setminus D_{\delta},$ we by Proposition \ref{bound} and Proposition \ref{small} have 
 $$\mathrm{Im}\big(V(x,y)-4\pi mx-4\pi ny\big)=\mathrm{Im}V(x,y)\leqslant \frac{1}{2}\mathrm{Vol}(\mathrm S^3\setminus K_{4_1})+\epsilon<\mathrm{Vol(M)}-\epsilon.$$

In $D_{\mathbb C},$  we have
$$0<\arg (1-e^{-2\sqrt{-1}(y+x)}) <\pi -2(\mathrm{Re}(y)+\mathrm{Re}(x))$$
and
$$2(\mathrm{Re}(y)-\mathrm{Re}(x))-\pi <\arg(1-e^{2\sqrt{-1}(y-x)})<0.$$

For $n>0,$ let $y=\mathrm{Re}(y)+\sqrt{-1}l.$ Then
\begin{equation*}
\begin{split}
\frac{\partial \mathrm{Im}\big(V(s,x,y)-4\pi m x-4\pi n y\big)}{\partial l}&=4\mathrm{Re}(x)+2\arg (1-e^{-2\sqrt{-1}(y+x)})+2\arg(1-e^{2\sqrt{-1}(y-x)})-4n\pi\\
&<4\mathrm{Re}(x)+2(\pi-2(\mathrm{Re}(y)+\mathrm{Re}(x)))+0-4n\pi\\
&=2\pi -4\mathrm{Re}(y)-4n\pi<-2\pi,
\end{split}
\end{equation*}
where the last inequality comes from that $0<\mathrm{Re}(y)<\frac{\pi}{2}$ and $n>0.$ Therefore, pushing the domain $D_{\delta}$ along the $\sqrt{-1}l$ direction far enough (without changing $\mathrm{Im}(x)$), the imaginary part of $V(s,x,y)-4\pi m x-4\pi n y$ becomes as small as possible. In particular,  
 for a sufficiently large $L,$ there is an $\epsilon>0$ such that
$$V(s,x,y)-4\pi m x-4\pi n y<\mathrm{Vol}(M)-\epsilon$$
on $S^+_{L, \text{top}}.$

Since $\mathrm{Im}\big(V(s,x,y)-4\pi m x-4\pi n y\big)$ is already smaller than the volume of $M$ on $\partial D_{\delta}$ and on $\partial S^+_{L,\text{top}},$  by Lemma \ref{convexity} it becomes even smaller on the side. Ie., 
$$V(s,x,y)-4\pi m x-4\pi n y<\mathrm{Vol}(M)-\epsilon$$
on $S^+_{L, \text{side}}.$

 Putting all together, we have for a sufficiently large $L,$ 
$$\mathrm{Im}\big(V(s,x,y)-4\pi m x-4\pi n y\big)<\mathrm{Vol}(M)-\epsilon$$
on $S^+_L$
for each triple $(s,m,n)$ with $n>0.$

For $n<0,$ let $y=\mathrm{Re}(y)-\sqrt{-1}l.$ Then
\begin{equation*}
\begin{split}
\frac{\partial  \mathrm{Im}\big(V(s,x,y)-4\pi m x-4\pi n y\big)}{\partial l}&=-4\mathrm{Re}(x)-2\arg (1-e^{-2\sqrt{-1}(y+x)})-2\arg(1-e^{2\sqrt{-1}(y-x)})+4n\pi\\
&<-4\mathrm{Re}(x)-0-2(2(\mathrm{Re}(y)-\mathrm{Re}(x))-\pi)+4n\pi\\
&=2\pi -4\mathrm{Re}(y)+4n\pi<-2\pi,
\end{split}
\end{equation*}
where the last inequality comes from that $0<\mathrm{Re}(y)<\frac{\pi}{2}$ again and $n<0.$ Therefore, pushing the domain $D_{\delta}$ along the $-\sqrt{-1}l$ direction far enough (without changing $\mathrm{Im}(x)$), the imaginary part of $V(s,x,y)-4\pi m x-4\pi n y$ becomes as small as possible. In particular,  
 for a sufficiently large $L,$ there is an $\epsilon>0$ such that
$$V(s,x,y)-4\pi m x-4\pi n y<\mathrm{Vol}(M)-\epsilon$$
on $S^-_{L, \text{top}}.$ 

Since $\mathrm{Im}\big(V(s,x,y)-4\pi m x-4\pi n y\big)$ is already smaller than the volume of $M$ on $\partial D_{\delta}$ and on $\partial S^-_{L,\text{top}},$  by Lemma \ref{convexity} it becomes even smaller on the side. Ie., 
$$V(s,x,y)-4\pi m x-4\pi n y<\mathrm{Vol}(M)-\epsilon$$
on $S^-_{L, \text{side}}.$

 Putting all together, we have for a sufficiently large $L,$ 
$$\mathrm{Im}\big(V(s,x,y)-4\pi m x-4\pi n y\big)<\mathrm{Vol}(M)-\epsilon$$
on $S^-_L$
for each triple $(s,m,n)$ with $n<0.$ 
\end{proof}


\subsubsection{Estimate on $D_{\frac{\delta}{2}}'$}

\begin{proposition}\label{D'}  There is an $\epsilon>0$ such that  for each triple $(s,m,n),$
\begin{equation*}
\begin{split}
\bigg|\int_{D'_{\frac{\delta}{2}}}\psi(x,y)\sin\bigg(\frac{x}{q}-J(s)\pi\bigg)&\epsilon(x,y)e^{-x\sqrt{-1}+\frac{r}{4\pi\sqrt{-1}}\big(V_r(s,x,y)-4\pi mx-4\pi ny\big)}dxdy\bigg|\leqslant O\Big(e^{\frac{r}{4\pi}\big(\mathrm{Vol}(M)-\epsilon\big)}\Big).
\end{split}
\end{equation*}
\end{proposition}

\begin{proof} For $L>0,$ we let 
$$S'^\pm_L=S'^\pm_{L,\text{top}}\cup S'^\pm_{L,\text{side}}\cup(D'_{\frac{\delta}{2}}\setminus D'_{\delta}),$$
where
$$S'^\pm_{L, \text{top}}=\big\{(x,y)\in D'_{\mathbb C,\delta}\ \big|\ \mathrm {Im} (x)=0, \mathrm {Im} (y)=\pm L \big\}$$
and
$$S'^\pm_{L, \text{side}}=\big\{(\theta_1,\theta_2\pm\sqrt{-1}l)\ \big|\ (\theta_1,\theta_2)\in \partial D'_\delta, l\in [0,L] \big\}.$$
Then $S'^\pm_L$ are  homotopic to  $D'_{\frac{\delta}{2}}.$

 We want to show a stronger statement  that: If $L$ is sufficiently large, then there is an $\epsilon>0$ such that for each triple $(s,m,n),$ 
$$\mathrm{Im}\big(V(s,x,y)-4\pi m x-4\pi n y\big)<\mathrm{Vol}(M)-\epsilon$$
either on $S'^+_L$ or on $S'^-_L.$

Then the result follows form Lemma \ref{LA'} and the analyticity of $V_r.$

To this end, first on $D'_{\frac{\delta}{2}}\setminus D'_{\delta},$ we by Proposition \ref{bound} and Proposition \ref{small} have 
 $$\mathrm{Im}\big(V(x,y)-4\pi mx-4\pi ny\big)=\mathrm{Im}V(x,y)\leqslant \frac{1}{2}\mathrm{Vol}(\mathrm S^3\setminus K_{4_1})+\epsilon<\mathrm{Vol(M)}-\epsilon.$$

Then in $D'_{\mathbb C,\delta},$ we have
$$0<\arg (1-e^{-2\sqrt{-1}(y+x)}) <\pi -2(\mathrm{Re}(y)+\mathrm{Re}(x))$$
and
$$2(\mathrm{Re}(y)-\mathrm{Re}(x))-3\pi <\arg(1-e^{2\sqrt{-1}(y-x)})<0.$$

For $n\geqslant 0,$ let $y=\mathrm{Re}(y)+\sqrt{-1}l.$ Then
\begin{equation*}
\begin{split}
\frac{\partial \mathrm{Im}\big(V(s,x,y)-4\pi m x-4\pi n y\big)}{\partial l}&=4\mathrm{Re}(x)+2\arg (1-e^{-2\sqrt{-1}(y+x)})+2\arg(1-e^{2\sqrt{-1}(y-x)})-4n\pi\\
&<4\mathrm{Re}(x)+2(\pi-2(\mathrm{Re}(y)+\mathrm{Re}(x)))+0-4n\pi\\
&=2\pi -4\mathrm{Re}(y)-4n\pi<-2\delta,
\end{split}
\end{equation*}
where the last inequality comes from that $\frac{\pi}{2}+\frac{\delta}{2}<\mathrm{Re}(y)<\pi-\frac{\delta}{2}$ and $n\geqslant0.$ Therefore, pushing the domain $D_{\delta}'$ along the $\sqrt{-1}l$ direction far enough (without changing $\mathrm{Im}(x)$), the imaginary part of $V(s,x,y)-4\pi m x-4\pi n y$ becomes as small as possible. In particular,  
 for a sufficiently large $L,$  there is an $\epsilon>0$ such that
$$V(s,x,y)-4\pi m x-4\pi n y<\mathrm{Vol}(M)-\epsilon$$
on $S'^+_{L, \text{top}}.$ 

Since $\mathrm{Im}\big(V(s,x,y)-4\pi m x-4\pi n y\big)$ is already smaller than the volume of $M$ on $\partial D'_{\delta}$ and on $\partial S'^+_{L,\text{top}},$  by Lemma \ref{convexity} it becomes even smaller on the side. Ie., 
$$V(s,x,y)-4\pi m x-4\pi n y<\mathrm{Vol}(M)-\epsilon$$
on $S'^+_{L, \text{side}}.$

 Putting all together, we have for a sufficiently large $L,$ 
$$\mathrm{Im}\big(V(s,x,y)-4\pi m x-4\pi n y\big)<\mathrm{Vol}(M)-\epsilon$$
on $S'^+_L$
for each triple $(s,m,n)$ with $n\geqslant 0.$

For $n<0,$ let $y=\mathrm{Re}(y)-\sqrt{-1}l.$ Then
\begin{equation*}
\begin{split}
\frac{\partial \mathrm{Im}\big(V(s,x,y)-4\pi m x-4\pi n y\big)}{\partial l}&=-4\mathrm{Re}(x)-2\arg (1-e^{-2\sqrt{-1}(y+x)})-2\arg(1-e^{2\sqrt{-1}(y-x)})+4n\pi\\
&<-4\mathrm{Re}(x)-0-2(2(\mathrm{Re}(y)-\mathrm{Re}(x))-3\pi)+4n\pi\\
&=6\pi -4\mathrm{Re}(y)+4n\pi<-2\delta,
\end{split}
\end{equation*}
where the last inequality comes from that $\frac{\pi}{2}+\frac{\delta}{2}<\mathrm{Re}(y)<\pi-\frac{\delta}{2}$ again and $n<0.$ Therefore, pushing the  domain $D_{\delta}'$ along the $-\sqrt{-1}l$ direction far enough (without changing $\mathrm{Im}(x)$), the imaginary part of $V(s,x,y)-4\pi m x-4\pi n y$  becomes as small as possible. In particular,  
 for a sufficiently large $L,$  there is an $\epsilon>0$ such that
$$V(s,x,y)-4\pi m x-4\pi n y<\mathrm{Vol}(M)-\epsilon$$
on $S'^-_{L, \text{top}}.$ 

Since $\mathrm{Im}\big(V(s,x,y)-4\pi m x-4\pi n y\big)$ is already smaller than the volume of $M$ on $\partial D'_{\delta}$ and on $\partial S'^-_{L,\text{top}},$  by Lemma \ref{convexity} it becomes even smaller on the side. Ie., 
$$V(s,x,y)-4\pi m x-4\pi n y<\mathrm{Vol}(M)-\epsilon$$
on $S'^-_{L, \text{side}}.$

 Putting all together, we have for a sufficiently large $L,$ 
$$\mathrm{Im}\big(V(s,x,y)-4\pi m x-4\pi n y\big)<\mathrm{Vol}(M)-\epsilon$$
on $S'^-_L$
for each triple $(s,m,n)$ with $n<0.$ 
\end{proof}


\subsubsection{Estimate on $D_{\frac{\delta}{2}}''$}

\begin{proposition}\label{D''}  There is an $\epsilon>0$ such that  for each triple $(s,m,n),$
\begin{equation*}
\begin{split}
\bigg|\int_{D''_{\frac{\delta}{2}}}\psi(x,y)\sin\bigg(\frac{x}{q}-J(s)\pi\bigg)&\epsilon(x,y)e^{-x\sqrt{-1}+\frac{r}{4\pi\sqrt{-1}}\big(V_r(s,x,y)-4\pi mx-4\pi ny\big)}dxdy\bigg|\leqslant O\Big(e^{\frac{r}{4\pi}\big(\mathrm{Vol}(M)-\epsilon\big)}\Big).
\end{split}
\end{equation*}
\end{proposition}

\begin{proof} For $L>0,$ we let 
$$S''^\pm_L=S''^\pm_{L,\text{top}}\cup S''^\pm_{L,\text{side}}\cup(D''_{\frac{\delta}{2}}\setminus D''_{\delta}),$$
where
$$S''^\pm_{L, \text{top}}=\big\{(x,y)\in D''_{\mathbb C,\delta}\ \big|\ \mathrm {Im} (x)=0, \mathrm {Im} (y)=\pm L \big\}$$
and
$$S''^\pm_{L, \text{side}}=\big\{(\theta_1,\theta_2\pm\sqrt{-1}l)\ \big|\ (\theta_1,\theta_2)\in \partial D''_\delta, l\in [0,L] \big\}.$$
Then $S''^\pm_L$ are  homotopic to  $D''_{\frac{\delta}{2}}.$

 We want to show a stronger statement  that: If $L$ is sufficiently large, then there is an $\epsilon>0$ such that for each triple $(s,m,n),$ 
$$\mathrm{Im}\big(V(s,x,y)-4\pi m x-4\pi n y\big)<\mathrm{Vol}(M)-\epsilon$$
either on $S''^+_L$ or on $S''^-_L.$

Then the result follows form Lemma \ref{LA'} and the analyticity of $V_r.$

To this end, first on $D''_{\frac{\delta}{2}}\setminus D''_{\delta},$ we by Proposition \ref{bound} and Proposition \ref{small} have 
 $$\mathrm{Im}\big(V(x,y)-4\pi mx-4\pi ny\big)=\mathrm{Im}V(x,y)\leqslant \frac{1}{2}\mathrm{Vol}(\mathrm S^3\setminus K_{4_1})+\epsilon<\mathrm{Vol(M)}-\epsilon.$$

Then in $D''_{\mathbb C,\delta},$ we have
$$0<\arg (1-e^{-2\sqrt{-1}(y+x)}) <3\pi -2(\mathrm{Re}(y)+\mathrm{Re}(x))$$
and
$$2(\mathrm{Re}(y)-\mathrm{Re}(x))-\pi <\arg(1-e^{2\sqrt{-1}(y-x)})<0.$$

For $n>0,$ let $y=\mathrm{Re}(y)+\sqrt{-1}l.$ Then
\begin{equation*}
\begin{split}
\frac{\partial  \mathrm{Im}\big(V(s,x,y)-4\pi m x-4\pi n y\big)}{\partial l}&=4\mathrm{Re}(x)+2\arg (1-e^{-2\sqrt{-1}(y+x)})+2\arg(1-e^{2\sqrt{-1}(y-x)})-4n\pi\\
&<4\mathrm{Re}(x)+2(3\pi-2(\mathrm{Re}(y)+\mathrm{Re}(x)))+0-4n\pi\\
&=6\pi -4\mathrm{Re}(y)-4n\pi<-2\delta,
\end{split}
\end{equation*}
where the last inequality comes from that $\frac{\pi}{2}+\frac{\delta}{2}<\mathrm{Re}(y)<\pi-\frac{\delta}{2}$ and $n>0.$ Therefore, pushing the domain $D_{\delta}''$  along the $\sqrt{-1}l$ direction far enough (without changing $\mathrm{Im}(x)$), the imaginary part of $V(s,x,y)-4\pi m x-4\pi n y$ becomes as small as possible. In particular,  
 for a sufficiently large $L,$  there is an $\epsilon>0$ such that
$$V(s,x,y)-4\pi m x-4\pi n y<\mathrm{Vol}(M)-\epsilon$$
on $S''^+_{L, \text{top}}.$ 

Since $\mathrm{Im}\big(V(s,x,y)-4\pi m x-4\pi n y\big)$ is already smaller than the volume of $M$ on $\partial D''_{\delta}$ and on $\partial S''^+_{L,\text{top}},$  by Lemma \ref{convexity} it becomes even smaller on the side. Ie., 
$$V(s,x,y)-4\pi m x-4\pi n y<\mathrm{Vol}(M)-\epsilon$$
on $S''^+_{L, \text{side}}.$

 Putting all together, we have for a sufficiently large $L,$ 
$$\mathrm{Im}\big(V(s,x,y)-4\pi m x-4\pi n y\big)<\mathrm{Vol}(M)-\epsilon$$
on $S''^+_L$
for each triple $(s,m,n)$ with $n>0.$

For $n\leqslant 0,$ let $y=\mathrm{Re}(y)-\sqrt{-1}l.$ Then
\begin{equation*}
\begin{split}
\frac{\partial  \mathrm{Im}\big(V(s,x,y)-4\pi m x-4\pi n y\big)}{\partial l}&=-4\mathrm{Re}(x)-2\arg (1-e^{-2\sqrt{-1}(y+x)})-2\arg(1-e^{2\sqrt{-1}(y-x)})+4n\pi\\
&<-4\mathrm{Re}(x)-0-2(2(\mathrm{Re}(y)-\mathrm{Re}(x))-\pi)+4n\pi\\
&=2\pi -4\mathrm{Re}(y)+4n\pi<-2\delta,
\end{split}
\end{equation*}
where the last inequality comes from that $\frac{\pi}{2}+\frac{\delta}{2}<\mathrm{Re}(y)<\pi-\frac{\delta}{2}$ again and $n\leqslant 0.$ Therefore, pushing the domain $D_{\delta}''$  along the $-\sqrt{-1}l$ direction far enough (without changing $\mathrm{Im}(x)$), the imaginary part of $V(s,x,y)-4\pi m x-4\pi n y$  becomes as small as possible. In particular,  
 for a sufficiently large $L,$  there is an $\epsilon>0$ such that
$$V(s,x,y)-4\pi m x-4\pi n y<\mathrm{Vol}(M)-\epsilon$$
on $S''^-_{L, \text{top}}.$ 

Since $\mathrm{Im}\big(V(s,x,y)-4\pi m x-4\pi n y\big)$ is already smaller than the volume of $M$ on $\partial D''_{\delta}$ and on $\partial S''^-_{L,\text{top}},$  by Lemma \ref{convexity} it becomes even smaller on the side. Ie., 
$$V(s,x,y)-4\pi m x-4\pi n y<\mathrm{Vol}(M)-\epsilon$$
on $S''^-_{L, \text{side}}.$

 Putting all together, we have for a sufficiently large $L,$ 
$$\mathrm{Im}\big(V(s,x,y)-4\pi m x-4\pi n y\big)<\mathrm{Vol}(M)-\epsilon$$
on $S''^-_L$
for each triple $(s,m,n)$ with $n\leqslant 0.$ 
\end{proof}


\subsection{Proof of Theorem \ref{main1} and Theorem \ref{main2} }\label{proof}

Theorem \ref{main1} follows from the following proposition.

\begin{proposition}\label{addup} 
\begin{enumerate}[(1)] 
\item The sum of the two leading Fourier coefficients 
\begin{equation*}
\begin{split}
\hat f_r(s^+,m^+,0)+\hat f_r(s^-,m^-,0)=\frac{c_r}{\sqrt{\mathrm{Tor}(M;\mathrm{Ad}_\rho)}}e^{\frac{r}{4\pi}\big(\mathrm{Vol}(M)+\sqrt{-1}\mathrm{CS}(M)\big)}\Big(1+O\Big(\frac{1}{r}\Big)\Big),
\end{split}
\end{equation*}
where 
$$c_r=-(-1)^{m^++J(s^+)-\frac{p'}{q}}e^{-\frac{\pi\sqrt{-1}r}{4}\big(K(s^+)+\frac{p'}{q}\big)}r\sqrt{q}\neq 0.$$

\item The sum of all the other Fourier coefficients
\begin{equation*}
\begin{split}
\sum_{(s,m,n)\neq (s^\pm,m^\pm,0)} \Big|\hat f_r(s,m,n)\Big|\leqslant O\Big(e^{\frac{r}{4\pi}\big(\mathrm{Vol}(M)-\epsilon\big)}\Big)
\end{split}
\end{equation*}
for some $\epsilon>0.$
\end{enumerate}
\end{proposition}

\begin{proof} For (1), recall that $\mathcal D_{\frac{\delta}{2}}=D_{\frac{\delta}{2}}\cup D'_{\frac{\delta}{2}}\cup D''_{\frac{\delta}{2}}.$ Then by Propositions \ref{Poisson}, \ref{leading}, \ref{D'} and \ref{D''}, we have
\begin{equation*}
\begin{split}
&\hat f_r(s^+,m^+,0)+\hat f_r(s^-,m^-,0)\\
=&(-1)^{m^+}\Big(\frac{r}{2\pi}\Big)^2\int_{\mathcal D_{\frac{\delta}{2}}}\psi(x,y)\sin\bigg(\frac{x}{q}-J(s^+)\pi\bigg)\epsilon(x,y)e^{-x\sqrt{-1}+\frac{r}{4\pi\sqrt{-1}}\big(V_r(s^+,x,y)-4\pi m^+x\big)}dxdy\\
&+(-1)^{m^-}\Big(\frac{r}{2\pi}\Big)^2\int_{\mathcal D_{\frac{\delta}{2}}}\psi(x,y)\sin\bigg(\frac{x}{q}-J(s^-)\pi\bigg)\epsilon(x,y)e^{-x\sqrt{-1}+\frac{r}{4\pi\sqrt{-1}}\big(V_r(s^-,x,y)-4\pi m^-x\big)}dxdy\\
=&\Big(\frac{r}{2\pi}\Big)^2\frac{\Big((-1)^{m^+}c^+_r+(-1)^{m^-}c^-_r\Big)}{\sqrt{\mathrm{Tor}(M;\mathrm{Ad}_\rho})}e^{\frac{r}{4\pi}\big(\mathrm{Vol}(M)+\sqrt{-1}\mathrm{CS}(M)\big)}\Big(1+O\Big(\frac{1}{r}\Big)\Big)+O\Big(e^{\frac{r}{4\pi}\big(\mathrm{Vol}(M)-\epsilon\big)}\Big),
\end{split}
\end{equation*}
where
 $$c_r^+=-\frac{2\pi^2\sqrt{q}}{r}(-1)^{J(s^+)-\frac{p'}{q}}e^{-\frac{\pi\sqrt{-1}r}{4}\big(K(s^+)+\frac{p'}{q}\big)}$$
and
$$c_r^-=\frac{2\pi^2\sqrt{q}}{r}(-1)^{J(s^+)-\frac{p'}{q}}e^{-\frac{\pi\sqrt{-1}r}{4}\big(K(s^-)+\frac{p'}{q}\big)}$$
are the constants in Proposition \ref{leading};
and we are left to prove that 
$$\Big((-1)^{m^+}c^+_r+(-1)^{m^-}c^-_r\Big)\Big(\frac{r}{2\pi}\Big)^2=c_r.$$
We claim that
\begin{equation}\label{6.2}
\begin{split}
-(-1)^{m^+}e^{-\frac{\pi\sqrt{-1}r}{4}K(s^+)}=(-1)^{m^-}e^{-\frac{\pi\sqrt{-1}r}{4}K(s^-)},
\end{split}
\end{equation}
from which the result follows. Indeed, by the definition of $K$ and Lemma \ref{arith} (1), we have
\begin{equation*}
\begin{split}
K(s^+)-K(s^-)&=-\frac{2}{q}\big(I(s^+)+I(s^-)\big)(s^+-s^-)=-4(m^++m^--1)(s^+-s^-).
\end{split}
\end{equation*}
Then
\begin{equation*}
\begin{split}
\frac{-(-1)^{m^+}e^{-\frac{\pi\sqrt{-1}r}{4}K(s^+)}}{(-1)^{m^-}e^{-\frac{\pi\sqrt{-1}r}{4}K(s^-)}}&=(-1)^{m^++m^--1}e^{-\frac{\pi\sqrt{-1}r}{4}\big(K(s^+)-K(s^-)\big)}\\
&=(-1)^{(m^++m^--1)\big(1+r(s^+-s^-)\big)}.
\end{split}
\end{equation*}
The result will follow if we can prove that  $(m^++m^--1)\big(1+r(s^+-s^-))$ is even. For this, by Lemma \ref{arith} (1) and computing $I(s^+)-I(s^-),$ we have
$$C_{k-1}(s^+-s^-)+(m^+-m^-)q=-1.$$ Therefore, $s^+-s^-$ and $m^+-m^-$ cannot be both even. As a consequence, since $r$ is odd, $m^++m^--1$ and $1+r(s^+-s^-)$ cannot be both odd. This completes the proof of (1).
\\

For (2), let 
$$\mathcal S=\Big\{ (s,m,n)\in \{0,\dots, |q|-1\}\times\mathbb Z^2\ \Big|\ (s,m,n)\neq (s^+,m^+,0), (s^-,m^-,0), \text{ and }(m,n)\neq (0,0)\Big\}.$$
Then $\mathcal S$ contains all but finitely many triples in $\{0,\dots, |q|-1\}\times\mathbb Z^2.$ By Propositions \ref{k1}, \ref{k2}, \ref{D'} and \ref{D''}, for each $(s,m,n)\neq (s^\pm,m^\pm,0),$ especially for those finitely many that are not in $\mathcal S,$
$$\Big|\hat f_r(s,m,n)\Big|\leqslant O\Big(e^{\frac{r}{4\pi}\big(\mathrm{Vol}(M)-\epsilon\big)}\Big)$$
for some $\epsilon>0.$
Therefore, it suffices to prove that
\begin{equation*}
\begin{split}
\sum_{(s,m,n)\in\mathcal S}\Big|\hat f_r(s,m,n)\Big|\leqslant O\Big(e^{\frac{r}{4\pi}\big(\mathrm{Vol}(M)-\epsilon\big)}\Big)
\end{split}
\end{equation*}
for some $\epsilon>0.$

Now for each $s,$ let 
$$h_r(s,x,y)=\psi(x,y)\sin\bigg(\frac{x}{q}-J(s)\pi\bigg)\epsilon(x,y)e^{-x\sqrt{-1}-\frac{\log\big(1-e^{-2\sqrt{-1}(y+x)}\big)}{2}-\frac{\log\big(1-e^{2\sqrt{-1}(y-x)}\big)}{2}+\frac{\upsilon_r(x,y)}{4\pi\sqrt{-1}r}}.$$
Then for each $(s,m,n)$ in $\mathcal S,$ since $\psi(x,y)$ vanishes outside of $\mathcal D,$  by the integration by parts we have
\begin{equation*}
\begin{split}
&r^4(m^4+n^4)\int_{\mathcal D}h_r(s,x,y)e^{\frac{r}{4\pi\sqrt{-1}}\big(V(s,x,y)-4\pi mx-4\pi ny\big)}dxdy\\
=& \int_{\mathcal D}h_r(s,x,y)e^{\frac{r}{4\pi\sqrt{-1}}V(s,x,y)}\bigg(\Big(\frac{\partial ^4}{\partial x^4}+\frac{\partial ^4}{\partial y^4}\Big)e^{\frac{r}{4\pi\sqrt{-1}}\big(-4\pi mx-4\pi ny\big)}\bigg)dxdy\\
=& \int_{\mathcal D}\bigg(\Big(\frac{\partial ^4}{\partial x^4}+\frac{\partial ^4}{\partial y^4}\Big)\Big(h_r(s,x,y)e^{\frac{r}{4\pi\sqrt{-1}}V(s,x,y)}\Big)\bigg)e^{\frac{r}{4\pi\sqrt{-1}}\big(-4\pi mx-4\pi ny\big)}dxdy\\
=&\int_{\mathcal D}\tilde h_r(s,x,y)e^{\frac{r}{4\pi\sqrt{-1}}\big(V(s,x,y)-4\pi mx-4\pi ny\big)}dxdy,
\end{split}
\end{equation*}
where
$$\tilde h_r(s,x,y)=\frac{\Big(\frac{\partial ^4}{\partial x^4}+\frac{\partial ^4}{\partial y^4}\Big)\Big(h_r(s,x,y)e^{\frac{r}{4\pi\sqrt{-1}}V(s,x,y)}\Big)}{e^{\frac{r}{4\pi\sqrt{-1}}V(s,x,y)}}$$
is a smooth function independent of $m$ and $n,$ and has the form
$$\tilde h_r(s,x,y)=\tilde h(s,x,y)\cdot r^4+O(r^3)$$
for a smooth function $\tilde h(s,x,y)$ independent of $r.$ Therefore, on the compact subset $S^\pm\cup S^\pm_L\cup S'^\pm_L\cup S''^\pm_L$
of $\mathcal D_\mathbb{C},$ $\frac{\tilde h_r(s,x,y)}{r^4}$ is bounded from above by some $C>0$ independent of $(s,m,n),$ where $S\pm,$ $S^\pm_L,$ $S'^\pm_L$ and $S''^\pm_L$ are the surfaces constructed in the proof of Propositions  \ref{k1}, \ref{k2}, \ref{D'} and \ref{D''}, Let $A$ be the maximum of the areas of these surfaces, and let $\epsilon'$ be the minimum of those $\epsilon$'s in Propositions \ref{k1}, \ref{k2}, \ref{D'} and \ref{D''}. Then we have
 \begin{equation*}
\begin{split}
\sum_{(s,m,n)\in\mathcal S}\Big|\hat f_r(s,m,n)\Big|&=\Big(\frac{r}{2\pi}\Big)^2\sum_{(s,m,n)\in\mathcal S}\Big|\int_{\mathcal D}h_r(s,x,y)e^{\frac{r}{4\pi\sqrt{-1}}\big(V(s,x,y)-4\pi mx-4\pi ny\big)}dxdy\Big|\\
&=\Big(\frac{r}{2\pi}\Big)^2\sum_{(s,m,n)\in\mathcal S}\frac{1}{m^4+n^4}\Big|\int_{\mathcal D}\frac{\tilde h_r(s,x,y)}{r^4}e^{\frac{r}{4\pi\sqrt{-1}}\big(V(s,x,y)-4\pi mx-4\pi ny\big)}dxdy\Big|\\
&\leqslant \Big(\frac{r}{2\pi}\Big)^2\bigg(\sum_{(s,m,n)\in\mathcal S}\frac{3AC}{m^4+n^4}\bigg)O\Big(e^{\frac{r}{4\pi}\big(\mathrm{Vol}(M)-\epsilon'\big)}\Big)\leqslant  O\Big(e^{\frac{r}{4\pi}\big(\mathrm{Vol}(M)-\frac{\epsilon'}{2}\big)}\Big).
\end{split}
\end{equation*}
The summation converges because $\sum_{(m,n)\neq (0,0)}\frac{1}{m^4+n^4}$ does; and the proof is completed with $\epsilon=\frac{\epsilon'}{2}.$
\end{proof}

\begin{proof}[Proof of Theorem \ref{main1}] By Proposition \ref{Poisson} and Proposition \ref{addup}, we only need to compute the constant $C_r.$ By a direct computation, we have
 \begin{equation*}
\begin{split}
\kappa_rc_r&=\frac{(-1)^{\frac{3(k+1)}{4}+\sum_{i=1}^ka_i}e^{\frac{\pi\sqrt{-1}}{r}\big(3\sigma(L)-\sum_{i=1}^ka_i-\sum_{i=2}^k\frac{1}{C_{i-1}C_i}\big)+\frac{\pi\sqrt{-1}r}{4}\big(\sigma(L)+3a_k\big)}}{2r\sqrt{q}}\\
&\quad\quad\quad\quad\quad\cdot\Big(-(-1)^{m^++J(s^+)-\frac{p'}{q}}e^{-\frac{\pi\sqrt{-1}r}{4}\big(K(s^+)+\frac{p'}{q}\big)}r\sqrt{q}\Big)\\
&=\frac{(-1)^{\frac{3k-1}{4}+\sum_{i=1}^ka_i+m^++J(s^+)-\frac{p'}{q}}e^{\frac{\pi\sqrt{-1}}{r}\big(3\sigma(L)-\sum_{i=1}^ka_i-\sum_{i=2}^k\frac{1}{C_{i-1}C_i}\big)}}{2}\cdot e^{\frac{\pi\sqrt{-1}r}{4}\big(\sigma(L)+3a_k-K(s^+)-\frac{p'}{q}\big)}.
\end{split}
\end{equation*}
Therefore, $C_r=2\kappa_rc_r$ has norm equal to $1.$ 

To compute the exponential growth rate, we have by Lemma \ref{arith} (3) that $K(s^+)+\frac{p'}{q}$ is an integer, and
 \begin{equation*}
\begin{split}
\lim_{r\to\infty}\frac{4\pi}{r}\mathrm{RT}_r(M)&=\mathrm{Vol}(M)+\sqrt{-1}\mathrm{CS}(M)+\sqrt{-1}\Big(\sigma(L)+3a_k-K(s^+)-\frac{p'}{q}\Big)\pi^2\\
&\equiv \mathrm{Vol}(M)+\sqrt{-1}\mathrm{CS}(M)\quad(\text{mod }\sqrt{-1}\pi^2\mathbb Z).
\end{split}
\end{equation*}
\end{proof}

\begin{proof}[Proof of Theorem \ref{main2}] At the root of unity $q=e^{\frac{2\pi\sqrt{-1}}{r}}$ for an odd integer $r,$ let $\mathrm{TV}_r'(M)$ be the $\mathrm{SO}(3)$ Turaev-Viro invariants as defined in \cite{DKY}, namely, summing over even colors instead of all integral colors in the definition. Then by \cite[Theorem 2.9]{DKY},  
$$\mathrm{TV}_r(M)=2^{b_2(M)-b_1(M)}\mathrm{TV}'_r(M);$$ 
and by the same argument as in \cite{Tu, W, Ro, DKY}, we have
$$\mathrm{TV}'_r(M)=|\mathrm{RT}'(M)|^2,$$
where $\mathrm{RT}'_r(M)$ is the $\mathrm{SO}(3)$ Reshetikhin-Turaev invariants. (See \cite{RT91, BHMV, L}, and for the exact definition used here, see \cite[Definition 2.1]{DKY}.) 

Next, we show that  
$$\mathrm{RT}'_r(M)=2\mathrm{RT}_r(M).$$
Indeed, in the $\mathrm{SO}(3)$ theory, the Kirby coloring
$$\omega'_r=\sum_{n=0}^{\frac{r-3}{2}}[2n+1]e_{2n}.$$
Then by \cite[Lemma 6.3 (iii)]{BHMV},
$$\omega_r=2\omega'_r$$
in the Kauffman bracket skein module $\mathrm K_r(D^2\times S^1)$ of the solid torus, and the constant
$$\mu'_r=\big(\langle \omega'_r\rangle _{U_+}\langle \omega'_r\rangle _{U_-}\big)^{-\frac{1}{2}}=\frac{2\sin\frac{2\pi}{r}}{\sqrt{r}}=2\mu_r,$$
where $U_\pm$ are the diagram of the unknot with framing $\pm 1.$
As a consequence,
$$\mu'_r\omega'_r=\mu_r\omega_r$$
in $\mathrm K_r(D^2\times S^1),$ and
 \begin{equation*}
\begin{split}
\mathrm{RT}'_r(M)&=\mu'_r \langle \mu'_r\omega'_r, \dots, \mu'_r\omega'_r\rangle_{D(L)}\langle \mu'_r \omega'_r\rangle _{U_+}^{-\sigma(L)}\\
&=2\mu_r \langle \mu_r\omega_r, \dots, \mu_r\omega_r\rangle_{D(L)}\langle \mu_r \omega_r\rangle _{U_+}^{-\sigma(L)}=2\mathrm{RT}_r(M).
\end{split}
\end{equation*}

Putting all these together, we have
$$\mathrm{TV}_r(M)=2^{b_2(M)-b_1(M)+2}|\mathrm{RT}(M)|^2,$$
and by Theorem \ref{main1} we have
$$ \mathrm{TV}_r(M)=\frac{2^{b_2(M)-b_0(M)}}{\big|\mathrm{Tor}(M;\mathrm{Ad}_\rho)\big|}e^{\frac{r}{2\pi}\mathrm{Vol}(M)}\Big(1+O\Big(\frac{1}{r}\Big)\Big).$$
\end{proof}


\appendix
\section{A proof of  Proposition \ref{saddle}}

To prove Proposition \ref{saddle}, we need the follow Lemmas, where the first one is the standard Complex Morse Lemma (see. eg. \cite[1.6]{Zo}). 

\begin{lemma}[Complex Morse Lemma] Let $f:\mathbb C^n\to \mathbb C$ be a holomorphic function with a non-degenerate critical point at $(c_1,\dots, c_n).$ Then there exists a holomorphic change of variables $(z_1,\dots, z_n)=\psi(Z_1,\dots, Z_n)$ on a neighborhood $V$ of $(c_1,\dots, c_n)$ such that $\psi(0,\dots, 0)=(c_1,\dots, c_n),$ 
$$f(\psi(Z_1,\dots, Z_n))=f(c_1,\dots, c_n)-Z_1^2 - \cdots - Z_n^2,$$
and
$$\det\mathrm D\psi(0,\dots,0)=\frac{2^{\frac{n}{2}}}{\sqrt{-\det\mathrm{Hess}(f)(c_1,\dots, c_n)}}.$$
\end{lemma}

\begin{lemma}\label{estimate} For any $\epsilon>0,$ there exists a $\delta>0$ such that 
\begin{enumerate}[(1)]
\item $$\int_{-\epsilon}^\epsilon e^{-rz^2}dz=\sqrt{\frac{\pi}{r}}+O(e^{-\delta r}),$$
and
\item $$\int_{-\epsilon}^\epsilon z^2e^{-rz^2}dz=\frac{1}{2}\sqrt{\frac{\pi}{r^3}}+O(e^{-\delta r}).$$
\end{enumerate}
\end{lemma}

\begin{proof} For (1), we  have
$$\int_{-\epsilon}^\epsilon e^{-rz^2}dz=\int_{-\infty}^\infty e^{-rz^2}dz-\int_{-\infty}^{-\epsilon} e^{-rz^2}dz-\int_{\epsilon}^\infty e^{-rz^2}dz,$$
where the first term 
$$\int_{-\infty}^\infty e^{-rz^2}dz=\sqrt{\frac{\pi}{r}}$$
is a Gaussian integral, and the other two terms
$$\int_{-\infty}^{-\epsilon} e^{-rz^2}dz=\int_{\epsilon}^\infty e^{-rz^2}dz\leqslant \int_{\epsilon}^\infty e^{-r\epsilon z}dz=\frac{e^{-r\epsilon^2}}{r\epsilon}=O(e^{-\delta r}).$$

For (2), by integration by parts, we have
$$\int_{-\epsilon}^\epsilon e^{-rz^2}dz=ze^{-rz^2}\Big|_{-\epsilon}^\epsilon+2r\int_{-\epsilon}^\epsilon z^2e^{-rz^2}dz,$$
hence by (1)
$$\int_{-\epsilon}^\epsilon z^2e^{-rz^2}dz=\frac{1}{2r}\Big(\int_{-\epsilon}^\epsilon e^{-rz^2}dz-2\epsilon e^{-r\epsilon^2}\Big)=\frac{1}{2}\sqrt{\frac{\pi}{r^3}}+O(e^{-\delta r}).$$
\end{proof}

\begin{proof}[Proof of Proposition \ref{saddle}] For simplicity, we use the bold letters $\mathbf z=(z_1,\dots, z_n),$ $d\mathbf z=dz_1\dots dz_n,$ $\mathbf c=(c_1,\dots,c_n)$ and $\mathbf 0=(0,\dots,0).$ 

We first consider a special case $\mathbf c=\mathbf 0,$ $S=[-\epsilon,\epsilon]^n\subset \mathbb R^n\subset \mathbb C^n,$
and 
$$f(\mathbf z)=-\sum_{i=1}^nz_i^2.$$
Let
$$\sigma_r(\mathbf z)=\upsilon_r(\mathbf z)\int_0^1e^{\frac{\upsilon_r(\mathbf z)}{r}s}ds.$$
Then we can write 
$$e^{\frac{\upsilon_r(\mathbf z)}{r}}=1+\frac{\sigma_r(\mathbf z)}{r},$$
and 
\begin{equation}\label{A1}
g(\mathbf z)e^{rf_r(\mathbf z)}=g(\mathbf z)e^{rf(\mathbf z)}+\frac{1}{r}g(\mathbf z)\sigma_r(\mathbf z)e^{rf(\mathbf z)}.
\end{equation}
Since $|\upsilon_r(\mathbf z)|<M$ for some $M>0$ independent of $r,$ 
$$|\sigma_r(\mathbf z)|<M\int_0^1e^{\frac{M}{r}s}ds=M\bigg(\frac{e^{\frac{M}{r}}-1}{\frac{M}{r}}\bigg)<2M.$$
If $M$ is big enough, then $|g(\mathbf z)|<M$ on $S,$ and by Lemma \ref{estimate} (1) we have
\begin{equation}\label{A2}
\begin{split}
\Big|\int_S\frac{1}{r}g(\mathbf z)\sigma_r(\mathbf z)e^{rf(\mathbf z)}\Big|&<\frac{2M^2}{r}\int_Se^{rf(\mathbf z)}d\mathbf z\\
&=\frac{2M^2}{r}\Big(\frac{\pi}{r}\Big)^{\frac{n}{2}}+O(e^{-\delta r})=O\Big(\frac{1}{\sqrt{r^{n+2}}}\Big).
\end{split}
\end{equation}
By the Taylor Theorem, we have
$$g(\mathbf z)=g(\mathbf 0)+\sum_{i=1}^ng_{z_i}(\mathbf 0)z_i+\sum_{i\neq j}g_{z_iz_j}(\mathbf 0)z_iz_j+\sum_{i=1}^nh_i(\mathbf z)z_i^2$$
for some holomorphic functions $h_i(\mathbf z),$ $i=1,\dots, n,$ on $D.$ 
Then by Lemma \ref{estimate} (1), we have
\begin{equation}\label{A3}
\int_Sg(\mathbf 0)e^{rf(\mathbf z)}d\mathbf z=g(\mathbf 0)\Big(\frac{\pi}{r}\Big)^{\frac{n}{2}}+O(e^{-\delta r}).
\end{equation}
Since each $z_ie^{-rz_i^2}$ is odd, we have
$$\int_{-\epsilon}^\epsilon z_ie^{-rz_i^2}dz_i=0.$$
As a consequence, we have 
\begin{equation}\label{A4}
\int_{-\epsilon}^\epsilon g_{z_i}(\mathbf 0)z_ie^{rf(\mathbf z)}d\mathbf z=0,
\end{equation}
for each $i,$ and
\begin{equation}\label{A5}
\int_{-\epsilon}^\epsilon g_{z_iz_j}(\mathbf 0)z_iz_je^{rf(\mathbf z)}d\mathbf z=0
\end{equation}
for each $i\neq j.$
If $M$ is big enough, then $|h_i(\mathbf z)|<M$ for all $i\in\{1,\dots, n\}$ on $S,$ and by Lemma \ref{estimate} we have for each $i\in\{1,\dots, n\}$ 
\begin{equation}\label{A6}
\begin{split}
\Big|\int_Sh_i(\mathbf z)z_i^2e^{rf(\mathbf z)}d\mathbf z\Big|<M\Big(\int_{-\epsilon}^\epsilon z_i^2e^{-rz_i^2}dz_i\Big)\prod_{j\neq i}\Big( \int_{-\epsilon}^\epsilon e^{-rz_j^2}dz_j\Big)=O\Big(\frac{1}{\sqrt{r^{n+2}}}\Big).
\end{split}
\end{equation}
Putting (\ref{A1}), (\ref{A2}), (\ref{A3}), (\ref{A4}), (\ref{A5}) and (\ref{A6}) together, we have the result.

For the general case, let $(V,\psi)$ be the change of variable for $f$ in the Complex Morse Lemma, and let $U\subset V$ such that
$$\psi^{-1}(U)=\prod_{i=1}^n\big\{Z_i\in\mathbb C\ \big|\ -\epsilon<\mathrm{Re}(Z_i)<\epsilon, -\epsilon<\mathrm{Im}(Z_i)<\epsilon \big\}.$$
Let $A$ be the volume of $S\setminus U.$ By the compactness and by conditions (2) and (5), there exist constants $M>0$ and $\delta>0$ such that 
$$|g(\mathbf z)|<M$$ 
and
$$\mathrm{Re} f_r(\mathbf z)<\mathrm{Re} f(\mathbf c)-\delta$$
on $S\setminus U.$ Then
\begin{equation}\label{A7}
\Big|\int_{S\setminus U}g(\mathbf z)e^{rf_r(\mathbf z)}d\mathbf z\Big|<MAe^{r(\mathrm{Re}f(\mathbf c)-\delta)}=O\Big(e^{r(\mathrm{Re}f(\mathbf c)-\delta)}\Big).
\end{equation}
In Figure \ref{sad} below, the shaded region is where $\mathrm{Re}(-\sum_{i=1}^nZ_i^2)<0.$ In $\overline{\psi^{-1}(U)},$ there is a homotopy $H$ from $\overline{\psi^{-1}(S\cap U)}$ to $[-\epsilon,\epsilon]^n\subset \mathbb R^n$ defined by ``pushing everything down'' to the real part.  This is where we need condition (3). Let $S'=H(\partial \psi^{-1}(S\cap U)\times [0,1]).$ Then $\overline{\psi^{-1}(S\cap U)}$ is homotopic to $S'\cup [-\epsilon,\epsilon]^n.$
\begin{figure}[htbp]
\centering
\includegraphics[scale=0.5]{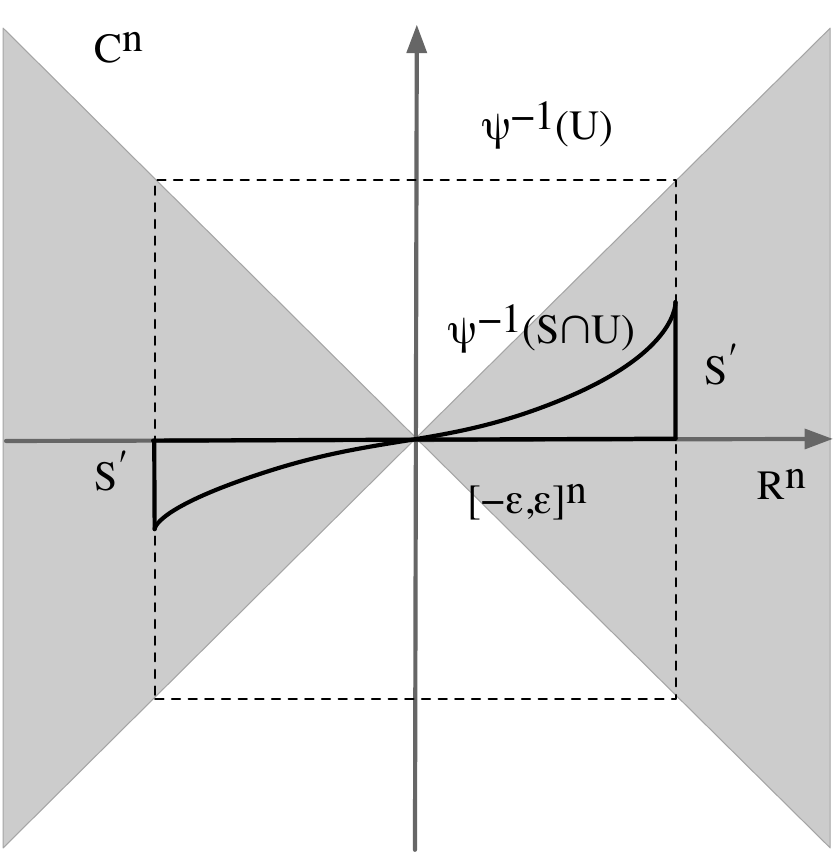}
\caption{}
\label{sad}
\end{figure}
Denote $\mathbf Z=(Z_1,\dots,Z_n).$ Then by analyticity, 
\begin{equation}\label{A8}
\begin{split}
\int_{S\cap U}g(\mathbf z)&e^{rf_r(\mathbf z)}d\mathbf z=\int_{\psi^{-1}(S\cap U)}g(\psi(\mathbf Z))\det \mathrm D(\psi(\mathbf Z))e^{rf_r(\psi(\mathbf Z))}d\mathbf Z\\
&=\int_{S'}g(\psi(\mathbf Z))\det \mathrm D(\psi(\mathbf Z))e^{rf_r(\psi(\mathbf Z))}d\mathbf Z+\int_{[-\epsilon,\epsilon]^n}g(\psi(\mathbf Z))\det \mathrm D(\psi(\mathbf Z))e^{rf_r(\psi(\mathbf Z))}d\mathbf Z.
\end{split}
\end{equation}
Since $\psi(S')\subset S\setminus U,$ 
\begin{equation}\label{A9}
\int_{S'}g(\psi(\mathbf Z))\det \mathrm D(\psi(\mathbf Z))e^{rf_r(\psi(\mathbf Z))}d\mathbf Z=\int_{\psi(S')}g(\mathbf z)e^{rf_r(\mathbf z)}d\mathbf z=O\Big(e^{r(\mathrm{Re}f(\mathbf c)-\delta)}\Big);
\end{equation}
and by the special case
\begin{equation*}
\begin{split}
&\int_{[-\epsilon,\epsilon]^n}g(\psi(\mathbf Z))\det \mathrm D(\psi(\mathbf Z))e^{rf_r(\psi(\mathbf Z))}d\mathbf Z\\
=&e^{rf(\mathbf c)}\int_{[-\epsilon,\epsilon]^n}g(\psi(\mathbf Z))\det \mathrm D(\psi(\mathbf Z))e^{r\big(-\sum_{i=1}^nZ_i^2+\frac{\upsilon_r(\psi(\mathbf Z)}{r^2}\big)}d\mathbf Z\\
=&e^{rf(\mathbf c)}g(\mathbf \psi(\mathbf 0))\det\mathrm D(\psi(\mathbf 0))\Big(\frac{\pi}{r}\Big)^{\frac{n}{2}}\Big(1+O\Big(\frac{1}{r}\Big)\Big)\\
=&\Big(\frac{2\pi}{r}\Big)^{\frac{n}{2}}\frac{g(\mathbf c)e^{rf(\mathbf c)}}{\sqrt{-\det\mathrm{Hess}(f)(\mathbf c)}}  \Big( 1 + O \Big( \frac{1}{r} \Big) \Big).
\end{split}
\end{equation*}
Together with (\ref{A7}), (\ref{A8}) and (\ref{A9}), we have the result.
\end{proof}


\noindent
Ka Ho Wong\\
Department of Mathematics\\  Texas A\&M University\\
College Station, TX 77843, USA\\
(daydreamkaho@math.tamu.edu)
\\

\noindent
Tian Yang\\
Department of Mathematics\\  Texas A\&M University\\
College Station, TX 77843, USA\\
(tianyang@math.tamu.edu)


\end{document}